\documentclass[reqno,dvipsnames]{amsart}
\usepackage{fullpage}
\usepackage{todonotes}
\usepackage[all,2cell]{xy}
\usepackage{graphicx, amsmath, amssymb, amsthm}
\usepackage{newclude}
\usepackage{mathrsfs}
\usepackage[hidelinks,backref]{hyperref}
\usepackage{lscape}
\usepackage{array}
\usepackage{enumitem}
\usepackage{tikz}
\usepackage{bbding}
\usepackage{stmaryrd}

\let\tops=\texorpdfstring

\newcommand{\conn}{\operatorname{conn}}

\newcommand{\gT}{\THH}
\newcommand{\TPhi}{\operatorname{T\Phi}}

\newcommand{\spoke}{\Yright}

\DeclareMathOperator{\rad}{rad}

\renewcommand{\th}{\text{th}}

\newcommand{\gh}{\mathrm{gh}}

\newcommand{\MWitt}{
 \begin{tikzpicture}[x=1.2ex,y=1.2ex]
    \draw (0,0) rectangle +(1.6,1.6);
 \end{tikzpicture}
}

\newcommand{\Ma}{
  \begin{tikzpicture}[x=1.2ex,y=1.2ex]
    \fill (0,0) circle (.8);
  \end{tikzpicture}
}
\newcommand{\Mb}{
  \begin{tikzpicture}[x=1.2ex,y=1.2ex]
    \fill (0,0) -- (1.6,0) -- (.8,{1.6*-sqrt(3)/2}) -- cycle;
  \end{tikzpicture}
}
\providecommand\Mc{}
\renewcommand{\Mc}{
  \begin{tikzpicture}[x=1.2ex,y=1.2ex]
    \draw (0,0) rectangle +(1.6,1.6);
    \fill (0,0) -- (1.6,0) -- (1.6,1.6) -- cycle;
  \end{tikzpicture}
}
\newcommand{\Md}{
  \begin{tikzpicture}[x=1.2ex,y=1.2ex]
    \fill (0,.8) -- (.8, 0) -- (1.6, .8) -- (.8, 1.6) -- cycle;
  \end{tikzpicture}
}
\newcommand{\Me}{\bigstar}
\newcommand{\Mf}{\spadesuit}
\newcommand{\Mg}{\clubsuit}

\newcommand{\Mta}{
  \begin{tikzpicture}[x=1.2ex,y=1.2ex]
    \draw (0,0) circle (.8);
  \end{tikzpicture}
}
\newcommand{\Mtb}{
  \begin{tikzpicture}[x=1.2ex,y=1.2ex]
    \draw (0,0) -- (1.6,0) -- (.8,{1.6*-sqrt(3)/2}) -- cycle;
  \end{tikzpicture}
}
\newcommand{\Mtc}{
  \begin{tikzpicture}[x=1.2ex,y=1.2ex]
    \draw (0,0) rectangle +(1.6,1.6);
    \draw (0,0) -- (1.6,1.6);
  \end{tikzpicture}
}
\newcommand{\Mtd}{
  \begin{tikzpicture}[x=1.2ex,y=1.2ex]
    \draw (0,.8) -- (.8, 0) -- (1.6, .8) -- (.8, 1.6) -- cycle;
  \end{tikzpicture}
}
\newcommand{\Mte}{\text{\FiveStarOpen}}

\newcommand{\ds}{\displaystyle}
\renewcommand{\ss}{\scriptstyle}

\newcommand{\<}{\langle}
\let\>=\undefined
\newcommand{\>}{\rangle}

\newcommand{\lc}{\left\lceil}
\newcommand{\rc}{\right\rceil}
\newcommand{\lf}{\left\lfloor}
\newcommand{\rf}{\right\rfloor}

\newcommand{\defeq}{\mathrel{:=}}
\newcommand{\tnsr}{\otimes}

\newcommand{\conv}{\Rightarrow}

\newcommand{\iso}{\morph^\sim}
\newcommand{\isom}{\cong}

\newcommand{\mmod}{/\!\!/}

\newcommand{\psr}[1]{[\![#1]\!]}

\DeclareMathOperator{\Hom}{Hom}
\DeclareMathOperator{\Map}{Map}

\DeclareMathOperator{\Aut}{Aut}

\DeclareMathOperator{\Spec}{Spec}

\DeclareMathOperator{\Gal}{Gal}

\newcommand{\Fil}{\mathrm{F}}

\newcommand{\morph}{\mathop{\longrightarrow}\limits}

\newcommand{\xra}{\xrightarrow}

\renewcommand{\projlim}{\mathop{\operatorname*{lim}\limits_{\longleftarrow}}\limits}
\renewcommand{\lim}{\projlim}
\newcommand{\colim}{\mathop{\operatorname*{lim}\limits_{\longrightarrow}}\limits}

\DeclareMathOperator{\tr}{tr}

\DeclareMathOperator{\im}{im}
\newcommand{\res}{\mathrm{res}}
\newcommand{\can}{\mathrm{can}}

\newcommand{\aInd}[2]{\mathop{\uparrow^{#2}_{#1}}\nolimits}
\newcommand{\aRes}[2]{\mathop{\downarrow}^{#1}_{#2}\nolimits}

\newcommand{\C}{\mathbb C}
\newcommand{\F}{\mathbb F}
\newcommand{\G}{\mathbb G}
\newcommand{\N}{\mathbb N}
\renewcommand{\O}{\mathcal O}
\newcommand{\Q}{\mathbb Q}
\newcommand{\R}{\mathbb R}
\let\sec=\S
\renewcommand{\S}{\mathbb S}
\newcommand{\Z}{\mathbb Z}

\renewcommand{\L}{\mathbb L}
\newcommand{\E}{\mathbb E}

\newcommand{\cF}{\mathcal F}
\newcommand{\cP}{\mathcal P}
\newcommand{\rmS}{\mathrm S}
\newcommand{\frN}{\mathfrak N}

\newcommand{\Zpcycl}{\Z_p^{\mathrm{cycl}}}
\newcommand{\Qpcycl}{\Q_p^{\mathrm{cycl}}}

\newcommand{\MU}{\mathrm{MU}}

\newcommand{\K}{{\mathrm{K}}}

\newcommand{\KU}{{\mathrm{KU}}}
\newcommand{\ku}{{\mathrm{ku}}}

\newcommand{\HH}{\mathrm{HH}}
\newcommand{\HC}{\mathrm{HC}}
\newcommand{\HP}{\mathrm{HP}}

\newcommand{\THH}{\mathrm{THH}}
\newcommand{\TR}{\mathrm{TR}}
\newcommand{\TC}{\mathrm{TC}}
\newcommand{\TCmin}{\mathrm{TC}^-}
\newcommand{\TP}{\mathrm{TP}}
\newcommand{\TF}{\mathrm{TF}}

\newcommand{\TM}{\mathrm{TM}}
\newcommand{\TV}{\mathrm{TV}}

\newcommand{\T}{\mathbb T}

\newcommand{\BK}{\mathfrak S}

\newcommand{\CycFL}{{\Sp^\psi}}
\newcommand{\CycFLp}{{\Sp^\psi_p}}
\newcommand{\CycSp}{{\Sp^\varphi}}
\newcommand{\CycSpp}{{\Sp^\varphi_p}}
\newcommand{\CyclSp}{{\Sp^\xi}}

\newcommand{\Ainf}{{\mathbf{A}_{\mathrm{inf}}}}

\renewcommand{\~}{\tilde}

\usepackage{relsize}
\usepackage[bbgreekl]{mathbbol}
\DeclareSymbolFontAlphabet{\mathbb}{AMSb} 
\DeclareSymbolFontAlphabet{\mathbbl}{bbold}

\newcommand{\RO}{RO}
\newcommand{\RU}{RU}
\newcommand{\m}{\underline}
\newcommand{\mpi}{\underline\pi}
\newcommand{\rog}{\bigstar}

\newcommand{\reg}{{\mathsf P}}
\newcommand{\regslice}[1]{\reg^{#1}_{#1}}
\newcommand{\regconn}[1]{\reg_{#1}}
\newcommand{\regtrunc}[1]{\reg^{#1}}

\newcommand{\cls}{{\mathsf P}^{\mathrm{cls}}}
\newcommand{\clsslice}[1]{\cls^{#1}_{#1}}
\newcommand{\clsconn}[1]{\cls_{#1}}

\newcommand{\drep}[1]{[#1]_\lambda}
\newcommand{\crep}[1]{\lambda[n-1]_\lambda}

\newcommand{\Ab}{\mathrm{Ab}}
\newcommand{\Spc}{{\mathcal S}}
\newcommand{\Sp}{\mathrm{Sp}}
\newcommand{\Top}{\mathrm{Top}}
\newcommand{\Mod}{\mathrm{Mod}}

\newcommand{\op}{\mathrm{op}}

\renewcommand{\H}{\mathrm H}
\newcommand{\tate}{\widehat\H}
\newcommand{\fib}{\operatorname{fib}}

\newcommand{\mup}[1]{\ar@/_1em/[u]_-{#1}}
\newcommand{\mdown}[1]{\ar@/_1em/[d]_-{#1}}
\newcommand{\mcup}[2][]{\ar@[#1]@/_1em/[u]_-{\color{#1}#2}}
\newcommand{\mcdown}[2][]{\ar@[#1]@/_1em/[d]_-{\color{#1}#2}}

\newcommand{\sLewiss}[7]{\xymatrix@=1.8em{
  #1 \mdown{#4}\\#2 \mup{#5} \mdown{#6} \\#3 \mup{#7}
}}

\newdir{ >}{{}*!/-10pt/@{>}}
\newcommand{\pullback}{\ar@{}[dr]|<<{\mbox{\Huge$\lrcorner$}}}
\newcommand{\pushout}{\ar@<2pt>@{}[ul]|<{\mbox{\Huge$\ulcorner$}}}

\newcommand{\adjnctn}[4]{\xymatrix@1{
  #1 \ar@<1ex>[r]^-{#3} \ar@{}[r]|-{\bot} & #2 \ar@<1ex>[l]^-{#4}
}}
\newcommand{\dadjnctn}[6][2em]{
  \xymatrix@1@C=#1{
    #2 \ar@<2ex>[r]^-{#4} \ar@<-2ex>[r]_-{#6}
    \ar@{}@<1.2ex>[r]|-{\scriptscriptstyle\bot} \ar@{}@<-1.2ex>[r]|-{\scriptscriptstyle\bot}
    &
    #3 \ar[l]|-{#5}
  }
}

\newcommand{\dimorph}[4]{\xymatrix@1{
  #1 \ar@<1ex>[r]^-{#3} \ar@<-1ex>[r]_-{#4} & #2
}}
\newcommand{\twoCell}[5]{\xymatrix@1{
  #1 \ar@<1ex>[r]^-{#3} \ar@{}[r]|{\Downarrow #5} \ar@<-1ex>[r]_-{#4} & #2
}}

\newcommand{\ndo}[2]{\xymatrix@1{
  #1 \ar[r]^-{#2} & #1
}}

\newtheorem{theorem}{Theorem}[section]
\newtheorem*{theorem*}{Theorem}

\newtheorem{lemma}[theorem]{Lemma}
\newtheorem{corollary}[theorem]{Corollary}
\newtheorem{proposition}[theorem]{Proposition}

\newtheorem*{conjecture*}{Conjecture}

\theoremstyle{definition}
\newtheorem{definition}[theorem]{Definition}
\newtheorem*{definition*}{Definition}

\newtheorem{observation}[theorem]{Observation}

\newtheorem{remark}[theorem]{Remark}
\newtheorem*{remark*}{Remark}

\newtheorem{example}[theorem]{Example}
\newtheorem*{example*}{Example}

\newtheorem{question}[theorem]{Question}
\newtheorem{openproblem}[theorem]{Open Problem}
\newtheorem{warning}[theorem]{Warning}
\newtheorem*{warning*}{Warning}

\newtheorem*{heuristic*}{Heuristic}
\newtheorem{abuse}[theorem]{Abusive Notation}

\makeatletter
\providecommand\@dotsep{5}
\renewcommand{\listoftodos}[1][\@todonotes@todolistname]{%
  \@starttoc{tdo}{#1}}
\makeatother

\SelectTips{cm}{}
\UseAllTwocells
\NoCompileMatrices

\usetikzlibrary{external}

\begin{document}

\title{A slice refinement of B\"okstedt periodicity}

\author[Y.~J.~F.~Sulyma]{Yuri~J.~F. Sulyma}
\address{Brown University \\ Providence, RI 02912}
\email{yuri\_sulyma@brown.edu}
\thanks{Sulyma was supported in part by NSF grants DMS-1564289 and DMS-1151577}

\begin{abstract}
Let $R$ be a perfectoid ring. Hesselholt and Bhatt-Morrow-Scholze have identified the Postnikov filtration on $\THH(R;\Z_p)$: it is concentrated in even degrees, generated by powers of the B\"okstedt generator $\sigma$, generalizing classical B\"okstedt periodicity for $R=\F_p$. We study an equivariant generalization of the Postnikov filtration, the \emph{regular slice filtration}, on $\THH(R;\Z_p)$. The slice filtration is again concentrated in even degrees, generated by $RO(\mathbb T)$-graded classes which can loosely be thought of as the \emph{norms} of $\sigma$. The slices are expressible as $\RO(\mathbb T)$-graded suspensions of Mackey functors obtained from the Witt Mackey functor. We obtain a sort of filtration by $q$-factorials. A key ingredient, which may be of independent interest, is a close connection between the Hill-Yarnall characterization of the slice filtration and Ansch\"utz-le Bras' $q$-deformation of Legendre's formula.
\end{abstract}

\maketitle
\tableofcontents

\section{Introduction}

In \cite{BMS2}, the authors construct ``motivic'' filtrations on $\THH$ and its variants, applying this to construct (completed) prismatic cohomology equipped with the Nygaard filtration. Their construction works by quasi-syntomic descent to the case of perfectoid rings, where the filtration is given by the Postnikov filtration. In this case, they show that for a perfectoid ring $R$,
\[ \pi_*\THH(R;\Z_p) = R[\sigma],\quad|\sigma|=2, \]
generalizing earlier work of Hesselholt for the case $R=\O_{\C_p}$ \cite{LarsOCp} and B\"okstedt's foundational calculation for $R=\F_p$ \cite{Bokstedt}.

Their work depends on the work of Nikolaus-Scholze \cite{NikolausScholze}, who give (in the bounded below case) a description of cyclotomic spectra in terms of Borel equivariant homotopy theory. However, for most of history cyclotomic spectra were studied via \emph{genuine} equivariant homotopy theory. In the genuine toolbox, a powerful way to study $G$-spectra is via the (regular, equivariant) \emph{slice filtration}. Modelled on the slice filtration from motivic homotopy theory, this was developed by Dugger \cite{Dugger} in the case $G=C_2$, and later generalized to finite groups by Hill-Hopkins-Ravenel as one of the main tools in their solution of the Kervaire invariant one problem \cite{HHR}. We will be interested in the ``regular'' slice filtration, first described by Ullman \cite{Ullman}, which has better multiplicative properties than the ``classical'' version used by HHR.

A simple but crucial point is that the slice filtration restricts to the Postnikov filtration on underlying spectra, and thus constitutes an equivariant generalization of the Postnikov filtration---one which is however quite different from the Postnikov $t$-structure on $G$-spectra, as it mixes in more of the representation theory of $G$. In 2018, Hill asked what happens if, in the BMS construction, one replaces the Postnikov filtration with the slice filtration. This gives a filtration which is sensitive to the genuine structure of $\THH$ but not its cyclotomic structure, so is in some sense intermediary between the BMS filtration and the cyclotomic filtration constructed by Antieau-Nikolaus \cite{TCart}.

In this paper we carry out the local calculation needed to answer Hill's question, identifying the slice filtration on $\THH(R;\Z_p)$ for a perfectoid ring $R$. On the arithmetic side, we present evidence that this is strongly related to $q$-divided powers. Before stating the results, let us say more about the problem from a homotopical point of view.

Our understanding of the slice filtration has come a long way since \cite{HHR}, but is still in its infancy. In particular, most slice computations to date have been of spectra closely related to $\MU_\R$, or of Mackey functors; $\THH$ is of a quite different flavor than these, yet is still a very reasonable spectrum. It is fair to be skeptical that methods in the highly specialized setting of cyclotomic spectra will work for general $G$-spectra, but aside from the calculational foothold afforded by cyclotomicity, our arguments only rely on connectivity of geometric fixed points and the Segal conjecture. Thus, while our investigations were motivated by arithmetic considerations, we hope that they will shed light on the slice filtration in general.

One novelty is that this is the first time the slice filtration has been considered for a compact Lie group (which we define so that it restricts to the slice filtration on all finite subgroups). This causes some peculiarities: for example, we no longer get periodicity with respect to the regular representation. However, there is a good replacement: writing $\T$ for the circle group, let $\lambda^i$ denote the one-dimensional complex $\T$-representation in which $z\in\T$ acts as multiplication by $z^i$. Then set
\begin{align*}
  [n]_\lambda &= \lambda^0 \oplus \dotsb \oplus \lambda^{n-1}\\
  \{n\}_\lambda &= \lambda^1 \oplus \dotsb \oplus \lambda^n
\end{align*}
A standard decomposition using roots of unity shows that $[n]_\lambda$ restricts to the complex regular representation of $C_n$, while $\{n\}_\lambda$ restricts to the \emph{reduced} complex regular representation of $C_{n+1}$. The $\{n\}_\lambda$ representations also appear when calculating the $K$-theory of truncated polynomial algebras \cite{HMCyclicPolytopes,SpeirsTrunc} or of coordinate axes \cite{HesselholtAxes,SpeirsAxes}, and many of the formulas there bear a striking resemblance to ours.

To state the main theorem, we require some further notation. Let $\m W(R)$ denote the Mackey functor of $p$-typical Witt vectors of $R$, and let $\tr_{C_n} \m W(R)$ be the sub-Mackey functor generated under transfers by restriction to $C_{p^{v_p(n)}}$. Quotients of Mackey functors are to be interpreted levelwise. Finally, we will assume for the rest of the Introduction that $R$ is a $\Zpcycl=\Z_p[\zeta_{p^\infty}]^\wedge_p$-algebra; the results are valid for any perfectoid ring, but the meaning of the $q$-analogues needs to be clarified. Recall that $\Ainf(\Zpcycl)=\Z_p[q^{1/p^\infty}]^\wedge_{(p,q-1)}$.

\begin{theorem}
\label{thm:main-1}
Let $R$ be a perfectoid $\Zpcycl$-algebra. The slice covers of $\THH(R;\Z_p)$ are given by
\begin{align*}
  \regconn{2n}\THH(R;\Z_p) &= \Sigma^{[n]_\lambda} \THH(R;\Z_p).\\
  \intertext{The slices are given by}
  \regslice{2n}\THH(R;\Z_p)
    &= \Sigma^{\{n\}_\lambda} \tr_{C_n}\m W(R)\\
    &= \Sigma^{[n]_\lambda} \m W(R)/[pn]_{q^{1/p}}.
\end{align*}
\end{theorem}
For comparison, Ullman \cite[Theorem 5.1]{Ullman} shows that $\regconn{2n} \ku_G=\Sigma^{[n]_\lambda} \ku_G$ (for $G$ finite cyclic), so this is a reasonable answer. As a $C_n$-spectrum the representation sphere $S^{[n]_\lambda}$ is just the norm $N_e^{C_n} S^2$, so we can roughly think of the above filtration as being generated by the \emph{norms} of the B\"okstedt generator $\sigma$. We do not know whether either of the two expressions for $\regslice{2n}\THH(R;\Z_p)$ is more natural.

\begin{remark}
For a thorough comparison with \cite{BMS2} and prismatic cohomology, one should additionally compute the slice filtration on the $\T$-spectra $\THH(R;\Z_p)^{E\T_+}$, $\widetilde{E\T}\tnsr \THH(R;\Z_p)^{E\T_+}$, and $\widetilde{E\T}\tnsr\THH(R;\Z_p)$. We will address this in future work.
\end{remark}

As the Theorem indicates, in order to work with the slice filtration it is necessary to have a handle on the $\RO(\T)$-graded homotopy $\mpi_\rog\THH(R;\Z_p)$. For virtual representations of the form $\rog=*-V$ with $*\in\Z$ and $V$ an actual representation, these groups come up when studying $\THH$ of monoid algebras, and can already be found in \cite[Proposition 9.1]{HMFinite}. However, the slice filtration tends to demand study of the degrees $\rog=*+V$. Gerhardt \cite{TeenaTR} computed $\pi_\rog\TR^n(\F_p)$ additively for general $\rog\in\RU(\T)$ (i.e., even-dimensional virtual representations), and Angeltveit-Gerhardt \cite{ROS1TR} extended this to all $\rog\in\RO(\T)$ (as well as to $\THH$ of $\Z$ and the Adams summand $\ell$). Their method is an $\RO(\T)$-graded version of the isotropy separation sequence, which they slickly package as the ``homotopy orbits to $\TR$ spectral sequence'' (HOTRSS).

In the course of applying their method to the slice filtration and generalizing it to perfectoid rings, we discovered several enhancements, incorporating recent advances in equivariant homotopy theory.
\begin{enumerate}
\item As we learned from work of Zeng \cite[\sec6]{Zeng}, the isotropy separation sequence can be packaged using the so-called \emph{gold elements} $a_{\lambda_i},\,u_{\lambda_j}$ of equivariant homotopy theory \cite[\sec3]{HHR_KR}. While this is in some sense just a change of notation, it makes the calculations considerably more transparent, and makes it easy to track the multiplicative and Mackey structure, which are important for us.

\item As we learned from Hill, using Nikolaus-Scholze/BMS techniques, one can work with $\TF$ instead of $\TR^n$. This greatly simplifies calculations by removing a lot of distracting torsion, which can be put back in later (if desired). For example, there is a sharp distinction between even- and odd-dimensional $\TF$ groups, but \emph{both} of these contribute to the even-dimensional $\TR^n$ groups studied by Gerhardt.

\item In order to generalize the method to perfectoid rings, our key lemma is a $q$-deformation of the \emph{gold relation} of \cite[Lemma 3.6(vii)]{HHR_KR}, describing the interaction of the aforementioned gold elements $a_{\lambda_i},u_{\lambda_j}$. Our ``$q$-gold relation'' (Lemma \ref{lem:q-au}) says that, for $0\le i < j$,
\begin{align*}
  \sigma a_{\lambda_i} &= [p^{i+1}]_{q^{1/p}} u_{\lambda_i}\\
  a_{\lambda_j} u_{\lambda_i} &= \phi^{i+1}([p^{j-i}]_{q^{1/p}}) a_{\lambda_i} u_{\lambda_j}.
\end{align*}
The case of a torsionfree perfectoid ring is in fact easier than the case of $\F_p$ (except notationally), since $[p]_{q^{1/p}}$ and $[p]_q$ do not interfere with one another, which removes another source of torsion.
\end{enumerate}

The HOTRSS played a central role in our early investigations; in the final product, we have reduced to considering a few very nice virtual representations, which can be computed in an ad hoc way. Thus, we will use Angeltveit-Gerhardt's insights to derive the $q$-gold relation, but then use cell structures to carry out the actual calculations, avoiding the HOTRSS. In \cite{SulROGTHH}, we will deploy the HOTRSS (or rather the HOTFSS) to compute the entire $\RO(\T)$-graded ring $\pi_\rog\TF(R;\Z_p)$.


The $q$-gold relation together with some knowledge of $\pi_\rog\TF(R;\Z_p)$ allows us to read off the effect of the slice filtration on homotopy. We describe the filtration on both $\pi_{2i}\TF(R;\Z_p)$ and $\pi_{[i]_\lambda}\TF(R;\Z_p)$, since Theorem \ref{thm:main-1} as well as the answer for $\pi_{2i}$ make $\pi_{[i]_\lambda}$ seem more natural.

\begin{theorem}
\label{thm:main-2}
The slice filtration takes the following form on homotopy. When $j\le0$ or $i=0$, $\Fil^{2j}_\rmS\pi_{[i]_\lambda}\TF(R;\Z_p)$ is all of $\pi_{[i]_\lambda}\TF(R;\Z_p)$. Otherwise, $\Fil^{2j}_\rmS\pi_{[i]_\lambda}\TF(R;\Z_p)$ is generated by
\[ \frac{[p(i+j-1)]_{q^{1/p}}!}{[p(i-1)]_{q^{1/p}}!}. \]
When $j\le0$ or $i=0$, $\Fil^{2j}_\rmS\pi_{2i}\TF(R;\Z_p)$ is all of $\pi_{2i}\TF(R;\Z_p)$. Otherwise, $\Fil^{2j}_\rmS\pi_{2i}\TF(R;\Z_p)$ is generated by
\[ \frac{[p(i+j-1)]_{q^{1/p}}!}{[p^r]_{q^{1/p}}^{i-1} \phi^r\left(\left[\lf\tfrac{i+j-1}{p^{r-1}}\rf\right]_{q^{1/p}}!\right)}, \]
where $r=\lc\log_p\left(\frac{i+j}i\right)\rc$.

In particular, taking $i=1$ in either case gives
\begin{align*}
  \Fil^{2j}_\rmS\pi_2 \TF(R;\Z_p)
    &= ([pj]_{q^{1/p}}!)\\
    &= ([p]_{q^{1/p}}^j [j]_q!).
\end{align*}
\end{theorem}

\begin{remark}
The arithmetic significance of this filtration is currently mysterious, but we will share some speculation. As explained in \cite[\sec17.5.3]{HillHandbook}, the slice filtration is the universal filtration such that the equivariant \emph{norm} functors scale filtration by the index of the subgroup. It is not clear what this means in the case of a compact Lie group, but a good place to start is the Witt vector norm maps
\[ W_n(R) \morph^N W_{n+1}(R); \]
these are natural from the point of view of equivariant homotopy, but have received little attention in number theory (aside from the Teichm\"uller lift, which is a special case). The description of the Witt vector norm is due to Angeltveit and Borger \cite{AngeltveitNorm}, but they are working with $W(R)$ rather than $\Ainf(R)$; although the norm maps can be lifted to $\Ainf(R)$, it is not clear if there is a canonical way to do so. However, a preferred lift is available in the $q$-crystalline case: the $q$-Pochhammer symbols
\[ (x,-y;q)_n \defeq (x-y)(x-qy)\dotsm(x-q^{n-1}y) \]
are lifts of $N(x-y)$, at least for $x$ and $y$ of rank one (Proposition \ref{prop:norms-q}). In particular, our filtration on $\pi_2\TF(R;\Z_p)$ should be compared with \cite[Proposition 4.9]{AClBTrace}. We discovered this connection as the paper was being finalized, and have not yet really explored it.
\end{remark}

We can also understand this filtration via the regular slice spectral sequence (RSSS). Here, for the first time, we see different behavior depending on the nature of $R$. When $R$ is $p$-torsionfree, the $E_2$ page is concentrated in even degrees, so the RSSS collapses at $E_2$. However, when $R$ is a perfect $\F_p$-algebra, there are differentials on every page arising from the ``collision'' of $\xi$ and $\phi(\xi)$ ($[p]_{q^{1/p}}$ and $[p]_q$). For a Mackey functor $\m M$, we let $\Phi^{C_n}\m M$ denote the cokernel of $\tr_{C_n}\m M\to \m M$.

\begin{theorem}
\label{thm:main-3}
The homotopy Mackey functors of the slices are given in even degrees by
\[
  \mpi_{2i}\regslice{2n} \gT =
  \begin{cases}
    \m W & 0=i=n\\
    \m R & 0<i=n\\
    \Phi^{C_{p^m}}\m W / [p^{h+1}]_{q^{1/p}} & 0< i < n
  \end{cases}
\]
where $\m R$ is the constant Mackey functor on $R$, and
\[
  m = \lc\log_p(n/i)\rc-1,
  \quad
  h=\begin{cases}
    \min\{v_p(n),\lf\log_p(n/i)\rf\} & n/i\text{ not a power of }p\\
    \lf\log_p(n/i)\rf & n/i\text{ a power of }p.
  \end{cases}
\]

If $R$ is $p$-torsionfree, then these are the only non-vanishing homotopy Mackey functors. If $R$ is a perfect $\F_p$-algebra, then
\[
  \mpi_{2i+1}\regslice{2n}\gT =
  \begin{cases}
    \tr_{C_{p^{m+h+1}}} \Phi^{C_{p^m}} \m W & n/i\text{ not a power of }p\\
    \tr_{C_{p^{m+h+1}}} \Phi^{C_{p^{m+1}}} \m W & n/i\text{ a power of }p.
  \end{cases}
\]
\end{theorem}

\begin{proposition}
\label{thm:main-4}
The homotopy Mackey functors $\mpi_{[i]_\lambda}$ of the slices are
\[
  \mpi_{[i]_\lambda}\regslice{2n} \gT =
  \begin{cases}
    \m W & 0=i=n\\
    \m W/[pn]_{q^{1/p}} & 0<i=n\\
    \Phi^{C_{p^{\ell(i,n)}}}\m W / [pn]_{q^{1/p}} & 0 < i < n
  \end{cases}
\]
where $\ell(i,n)=\max\{v_p(i),\dotsc,v_p(n-1)\}=\min\{r \mid \lc n/p^r \rc = \lc i/p^r \rc\}$.
\end{proposition}

We provide several charts illustrating these filtrations at the end of \sec\ref{sub:rsss}. Although the RSSS for a perfect $\F_p$-algebra is very complicated, the $E_\infty$ page can be inferred from Theorem \ref{thm:main-2}, since we have explicitly identified the slice tower.

\begin{corollary}
\label{cor:main-5}
Let $R$ be a perfect $\F_p$-algebra, and define $h=h(n,i)$ as in Theorem \ref{thm:main-2}. The entry on the $E_\infty$ page of the RSSS corresponding to $\mpi_{2i}\regslice{2n}\THH(R)$ is $\Phi^{C_{p^{f+1}}} \m W/p^{h(n,i)+1}$, where $f=\sum_{i\le m<n} (h(m,i)+1)$.
\end{corollary}

\subsection{\tops{$q$}{q}-analogues and \tops{$p$}{p}-typification}
The above results require a great deal of preliminaries to state, and some ugly (but straightforward) calculations to prove. Our conceptual results are much simpler to explain. A basic combinatorial problem is to determine the sequence of $p$-adic valuations $(v_p(1),v_p(2),\dotsc,v_p(n))$. The answer is that there are $\lf\tfrac n{p^k}\rf$ multiples of $p^k$ in the set $\{1,\dotsc,n\}$. Three instances where this problem, and hence the expression $\lf\frac n{p^k}\rf$, arise are:
\begin{enumerate}
\renewcommand{\labelenumi}{(\alph{enumi})}
\item the $p$-adic valuation of a factorial, and the $q$-analogue of this problem \cite[Lemma 4.8]{AClBTrace};

\item the decomposition of big Witt vectors in terms of $p$-typical Witt vectors;

\item the $p$-typical decomposition of the representation $\{n\}_\lambda$.
\end{enumerate}
A similar expression appears in the work of Hill-Yarnall \cite{NewSlices} characterizing the slice filtration in terms of the Postnikov filtration. In our case, their result says that
\begin{enumerate}
\renewcommand{\labelenumi}{(\alph{enumi})}
\setcounter{enumi}3

\item a $\T$-spectrum $X$ is slice $n$-connective if and only if the geometric fixed points $X^{\Phi C_{p^k}}$ are $\lc\frac n{p^k}\rc$-connective (in the ordinary sense) for all cyclic subgroups $C_{p^k}\le\T$.
\end{enumerate}
(This discrepancy between floor and ceiling functions is a persistent phenomenon, and parallels the discrepancy between the representations $\{n\}_\lambda$ and $[n]_\lambda$. But they are of course very closely related, and things ultimately work out in our favor.\footnote{One could avoid this discrepancy by using the classical slice filtration, but the regular slice filtration is more natural.})

A second overall theme is that of $q$-analogues. We have intentionally chosen the notation $[n]_\lambda$ to emphasize that these representations are $\lambda$-analogues, and in fact the connection between (c) and (a) is most clearly seen when considering $q$-factorials. We also argue over the course of \sec\ref{sec:bg-arith} that $q$-analogues (and prisms) are natural from the point of view of equivariant homotopy theory.

The point of the paper is to connect these varied occurrences. Since it is easy to take geometric fixed points of cyclotomic spectra, the slice filtration is essentially determined by combinatorics. This allows us to relate (d) to (c) and prove Theorem \ref{thm:main-1}. To relate (c) to (a), we study the $\RO(\T)$-graded homotopy Mackey functors of $\THH$, where we find that the \emph{$q$-gold relation} implements a close dictionary between representations and $q$-analogues. Theorem \ref{thm:main-2} through Corollary \ref{cor:main-5} then follow fairly quickly from Theorem \ref{thm:main-1} and some unpleasant algebra.

We warn the reader that we have not yet incorporated example (b). We consider $\THH$ as a \emph{$p$-typical} cyclotomic spectrum here, although our formulas seem more natural when it is viewed as an ``integral'' cyclotomic spectrum. (So we are really working with the group $C_{p^\infty}$ rather than $\T$.) Many formulas are stated integrally, but interpreted and proven $p$-typically. We expect that our results hold in the integral case, but we have not yet checked this carefully. This would partly clarify the relation to the $K$-theory of truncated polynomial algebras / coordinate axes, which are given in terms of big de Rham-Witt forms.

\subsection{Overview}
Part \ref{part:background} collects the needed technical background and notation. Unfortunately, there is a lot; we have tried to make the paper accessible (and self-contained) to both arithmetic geometers unfamiliar with the pre-Nikolaus-Scholze formulation of cyclotomic spectra, and to homotopy theorists who are interested in slice computations but who have not studied \cite{BMS2} extensively. Homotopical background is presented in \sec\ref{sec:bg-htpy}, and arithmetic background is presented in \sec\ref{sec:bg-arith}. We hope that collecting all of this information in one place will be helpful to the field.

Our notation is recapitulated at the beginning of Part \ref{part:results}, so one may begin there and refer back as needed. However, even experts should read \sec\ref{sub:q-an}; while there are no really new technical results there, it contains the key observations at the conceptual heart of this paper. There is also a little bit of new material in \sec\ref{sub:prisms}.

In \sec\ref{sec:rog}, we compute certain $\RO(\T)$-graded homotopy Mackey functors of $\THH$, taking a geodesic route to the computations which are needed for describing the slice filtration. The main theorems are proved in \sec\ref{sec:slice}. In \sec\ref{sec:future}, we propose some natural followup questions.

\subsection{Acknowledgements}
We are grateful to Mike Hill for suggesting this problem; to Andrew Blumberg for supervision and guidance; to Aaron Royer for patiently fielding many questions about equivariant stable homotopy theory; to Teena Gerhardt for clarifying a confusion about the Tate spectral sequence; to Ben Antieau and Thomas Nikolaus for help sorting out a confusion about the $A$-algebra structure on $\TR^n$; and to Dylan Wilson for the crucial suggestion to use the gold elements. Figures \ref{fig:q-lgndr-2} and \ref{fig:q-lgndr-3} were originally created for Kate Stange's Number Theory and Friends group as a part of ICERM's Illustrating Mathematics semester. Additionally, this paper owes a tremendous intellectual debt to \cite{ROS1TR} and \cite{Zeng}.

Preliminary versions of the results in this paper appeared as part of the author's PhD thesis at the University of Texas at Austin.

\newpage
\part{Background}
\label{part:background}
\section{Homotopical background}
\label{sec:bg-htpy}

In this section we give a crash course in equivariant stable homotopy theory, with an eye towards computational aspects. We do not intend to give a complete or fully rigorous account of this subject---which would be almost impossible in the given space---but simply to recall the main points and indicate some of the pitfalls. For a comprehensive introduction that covers everything we need here (except for \sec\sec\ref{sub:thh-friends} and \ref{sub:T-repns}), we strongly recommend \cite{HillHandbook}. Another course account is \cite{BlumbergNotes}; a modern reference is \sec2 and Appendices A and B of \cite{HHR}; the original source is \cite{LMS}; and \cite{tomDieck1979} is also very nice, for example arithmetic geometers will enjoy the material on $\lambda$-rings.

In \sec\ref{sub:equivariant}, we recall the general definitions of equivariant stable homotopy theory. The main point is to explain the difference between ``naive'' and ``genuine'' equivariant homotopy theory, and the interactions between the four different types of fixed points in the genuine stable theory. Finally, we introduce the cases of interest to us, the homotopy theories of cyclonic and cyclotomic spectra.

\sec\ref{sub:thh-friends} documents the plethora of spectra that can be obtained from $\THH$.

In \sec\ref{sub:mackey} and \sec\ref{sub:rog-grading}, we recall the definition of Mackey functors and of the $RO(G)$-graded homotopy Mackey functor $\mpi_\rog X$ of a $G$-spectrum $X$. In particular, we discuss the equivariant Euler classes $a_V$ and their relation to isotropy separation squares. As we are interested in the groups $\T$ and $C_{p^n}$, we review in \sec\ref{sub:T-repns} their representation theory as well as cell structures for their representation spheres. Lastly, in \sec\ref{sub:slice-filtration} we introduce the regular slice filtration.


\subsection{Equivariant stable homotopy theory}
\label{sub:equivariant}

At the beginning of this project, the author found equivariant homotopy theory extremely confusing. Now, the author finds equivariant homotopy theory only very confusing. We have tried to explain things in the way that made the subject make sense to us.

We start in \sec\ref{subsub:unstable} with the unstable theory, but our only interest in this is on the way to the stable theory. What we need to explain is the difference between ``naive'' and ``genuine'' equivariance. This results in there being \emph{two} types of fixed points in equivariant unstable homotopy theory: the \emph{homotopy fixed points} $X^{hG}$ and the \emph{categorical fixed points} $X^G$.

There are then \emph{four} types of fixed points once we pass to equivariant \emph{stable} homotopy theory:
\begin{itemize}
\item the \emph{homotopy fixed points} $X^{hG}$;
\item the \emph{categorical fixed points} $X^G$;
\item the \emph{Tate fixed points} $X^{tG}$;
\item the \emph{geometric fixed points} $X^{\Phi G}$.
\end{itemize}
The categorical and geometric fixed points arise in passing from naive to genuine equivariant homotopy theory, while the Tate and geometric fixed points have to do with passing from unstable to stable homotopy theory. Categorical fixed points are to homotopy fixed points as geometric fixed points are to Tate fixed points, in a sense made precise by the isotropy separation square \eqref{eq:isotropy-separation-Cn}. Ultimately, the difference between the unstable and stable cases comes from the existence of the \emph{trace} (often called the norm).

\subsubsection{The unstable theory}
\label{subsub:unstable}
Let $\Spc$ denote the $\infty$-category of spaces, and let $\Top$ denote the (a) \emph{topological category} of spaces. Given a compact Lie group $G$, we wish to define the $\infty$-category $G\Spc$ of $G$-spaces. As we will see, there are several different options for what this might mean. We will examine how to do this using the point-set model, then explain what this means $\infty$-categorically.

\begin{definition}
A \emph{$G$-topological space} is a topological space $X$ equipped with a continuous left action of $G$. The \emph{topological category} of $G$-topological spaces is denoted $G\Top$.
\end{definition}

If $G$ is finite, then we may identify $G\Top$ with $\Top^{BG}$. \emph{However, it is very much not the case that $G\Spc = \Spc^{BG}$.}

If we want to do homotopy theory, it is not enough to specify the category of $G$-topological spaces: we must also specify the weak equivalences. Here, there is a choice to be made. Given a $G$-topological space $X$ and a subgroup $H\le G$, write
\[ X^H = \{x\in X \mid h\cdot x = x \quad \forall h\in H\} \]
for the subspace of $H$-fixed points, and observe that the functor $(-)^H$ is representable by $G/H$.

\begin{definition}
A map $X \morph^f Y$ of $G$-topological spaces is called
\begin{itemize}
\item a \emph{naive weak equivalence} if the underlying map $X^e \morph^{f^e} Y^e$ is a weak equivalence;

\item a \emph{genuine weak equivalence} if $X^H \morph^{f^H} Y^H$ is a weak equivalence for all subgroups $H\le G$.
\end{itemize}
More generally, given a family $\cF$ of subgroups of $G$ closed under conjugation and passage to subgroups, an \emph{$\cF$-weak equivalence} is a map $f$ such that $f^H$ is a weak equivalence for all $H\in\cF$.

We write $G\Spc_\text{naive}$, $G\Spc$, $G\Spc_{\cF}$ for the $\infty$-categories obtained by localizing $G\Top$ with respect to the naive, genuine, or $\cF$-weak equivalences; $G\Spc_\text{naive} = G\Spc_{\{1\}}$ and $G\Spc = G\Spc_{\{\text{all}\}}$. We call these the $\infty$-categories of \emph{naive $G$-spaces}, \emph{genuine $G$-spaces}, and \emph{$\cF$-genuine $G$-spaces}.
\end{definition}

In other words, each subgroup $H$ of $G$ gives a functor $G\Top \morph^{(-)^H} \Top$, and by specifying $\cF$ we are deciding which of these should be homotopically meaningful, descending to a functor $G\Spc_{\cF} \to \Spc$. It turns out that this is essentially all of the homotopical data contained in a $G$-space.

\begin{definition}
Let $G$ be a finite group. The \emph{orbit category} $\O(G)$ of $G$ is the full subcategory of $\Spc^{BG}$ spanned by the nonempty transitive $G$-sets. $\O(G)$ is equivalent to its full subcategory $\{G/H\}_{H\le G}$.
\end{definition}

\begin{remark}
A family $\cF$ of subgroups of $G$ closed under conjugation and passage to subgroups is essentially the same thing as a downwards-closed subcategory of $\O(G)$.
\end{remark}

\begin{theorem}[{\cite{Elmendorf}}]
The restricted Yoneda embedding \[ G\Top \to \Top^{\O(G)^\op} \] gives an equivalence of $\infty$-categories
\begin{align*}
  G\Spc_\mathrm{naive} &= \Spc^{BG}\\
  G\Spc &= \Spc^{\O(G)^\op}\\
  G\Spc_{\cF} &= \Spc^{\cF^\op}
\end{align*}
\end{theorem}

In particular, $G/H$ represents the functor $G\Spc \morph^{(-)^H} \Spc$ (now basically by fiat). In fact, this factors through $\Spc^{\Aut(G/H)} = \Spc^{W_G(H)}$, where $W_G(H) = N_G(H)/H$ is the Weyl group.

In summary, a $G$-space has \emph{two} different notions of fixed point:
\begin{itemize}
\item the \emph{homotopy fixed points} $X^{hH}$, obtained by restricting $X$ to $\Sp^{BH}$ and then right Kan extending along $BH\to *$;

\item the \emph{categorical fixed points} $X^H$.
\end{itemize}

\subsubsection{The stable theory}
\label{subsub:stable}
The first guess for the stable theory is to simply stabilize the unstable theory from the previous subsection. This gives two options:
\begin{itemize}
\item the stabilization $\Sp(\Spc^{BG})$ of naive $G$-spaces is equivalent to the functor category $\Sp^{BG}$. These are variously called \emph{Borel $G$-spectra}, \emph{coarse $G$-spectra}, \emph{FS-$G$-spectra}, \emph{naive $G$-spectra}, or \emph{doubly naive $G$-spectra} in the literature. We shall use the term \emph{naive $G$-spectra}.

\item the stabilization $\Sp(\Spc^{\O(G)^\op})$ of genuine $G$-spaces is equivalent to the functor category $\Sp^{\mathcal \O(G)^\op}$. These are sometimes called \emph{naive $G$-spectra} in the literature (clashing with the above); we shall use the nonstandard term \emph{ersatz $G$-spectra}. One can of course replace $\O(G)$ with $\cF$.
\end{itemize}

Naive $G$-spectra are useful, but ersatz $G$-spectra, as the name suggests, are the wrong notion to consider. The reason they are wrong is the same reason we get more fixed point functors.

If $K \to H$ is an inclusion (or subconjugacy relation) of subgroups of $G$, then there is a restriction map $X^H \to X^K$ between the fixed points. In the unstable setting, this is the only relation we should expect between fixed point spaces, in general. But in the presence of addition, there is an easy way to produce $H$-fixed points from $K$-fixed points: simply sum over conjugates, indexed by $H/K$; this is known as a \emph{transfer} map. Thus, in the stable setting, we should require transfers $\tr^H_K \colon X^K \to X^H$ in addition to restriction maps $\res^H_K \colon X^H \to X^K$. These are not present in $\Sp^{\O(G)^\op}$.

Instead, recalling that $\O(G)$ was defined as the subcategory of $\Spc^{BG}$ spanned by the nonempty transitive $G$-sets (equivalently, by the orbits $G/H$), we define $\mathcal A_G$ to be the subcategory of $\Sp^{BG}$ spanned by $X_+$, where $X$ is a nonempty transitive $G$-set, equivalently by the orbits $G/H_+$. $\mathcal A_G$ is called the \emph{Burnside $\infty$-category} of $G$; the classical (or algebraic) Burnside category is $\mathcal B_G = \pi_0\mathcal A_G$. A \emph{Mackey functor} is an additive functor $\m M\colon \mathcal B_G^\op \to \Ab$; Mackey functors are reviewed in \sec\ref{sub:mackey}.

The category of enriched functors $\mathcal A(G)^\op \to \Sp$ turns out to give the correct notion of genuine $G$-spectrum: that is, genuine $G$-spectra are \emph{spectral Mackey functors}. This is revisionist: traditionally, genuine $G$-spectra are defined by starting with genuine $G$-spaces and inverting all \emph{representation} spheres (ersatz $G$-spectra only invert the ordinary spheres). The spectral Mackey functor formulation is due to Guillou-May \cite{GuillouMay}, and was used by Barwick \cite{BarwickMackeyI} to provide a fully $\infty$-categorical treatment of $G$-spectra. Kaledin also has a homological analogue \cite{KaledinMackey}.

Summing over conjugates provides easy examples of fixed points; we would like to distill the interesting ones. The \emph{Tate fixed point spectrum} is defined as the cofiber of the \emph{trace map} $T$ from the homotopy orbits to the homotopy fixed points:
\[ X_{hG} \morph^T X^{hG} \morph X^{tG} \]
which on $\pi_0$ is $T(x) = \sum_{g\in G} x$. The theory of the Tate spectrum is due to Greenlees and May \cite{GMTate}.

\begin{remark}
$T$ is often written as $N$, and called the \emph{norm} map; however, Hill has pointed out that that is more properly reserved for the multiplicative notion.
\end{remark}

\begin{example}
Let $M$ be the free $\Z$-module $\Z\<x,y\>$, and let $C_2$ act on $M^{\tnsr 2}$ in the obvious way. Then both $x^2$ and $xy + yx$ are in $\H^0(C_2, M)$, but only $x^2$ is nonzero in $\tate^0(C_2, M) \isom \Z/2\<x^2,y^2\>$.
\end{example}

The \emph{geometric fixed points} $X^{\Phi G}$ are a ``genuine'' version of the Tate fixed points\footnote{More accurately of the \emph{generalized} Tate construction, c.f.\ \cite{AMGRgenuine}.} $X^{tG}$. To define them, let $\cF$ be a family of subgroups of $G$, and let $E\cF$ be a $G$-space with
\[
  (E\cF)^H =
  \begin{cases}
  \emptyset & H\notin\cF\\
  * & H\in\cF
  \end{cases}
\]
(Such a $G$-space exists by Elmendorf's theorem, and $E\cF$ is in fact determined by this condition.) Then let $\widetilde{E\cF}$ be the space given by the cofiber sequence
\[ E\cF_+ \to S^0 \to \widetilde{E\cF}. \]
For any spectrum $X$, we get a diagram
\begin{equation}
\label{eq:isotropy-separation}
\vcenter{\xymatrix{
  E\cF_+\tnsr X \ar[r] \ar[d] & X \ar[r] \ar[d] & \widetilde{E\cF}\tnsr X \ar[d]\\
  E\cF_+\tnsr X^{E\cF_+} \ar[r] & X^{E\cF_+} \ar[r] & \widetilde{E\cF}\tnsr X^{E\cF_+}
}}
\end{equation}
with the following properties:
\begin{itemize}
\item it can be shown that the left vertical map is an equivalence, so the right square is a pullback/pushout;

\item the terms in the left column are concentrated on the subgroups in $\cF$;

\item the terms in the right column vanish on the subgroups in $\cF$;

\item the bottom row depends only on the restriction of $X$ to $\cF$.
\end{itemize}

There are two cases of particular interest. When $\cF=\{e\}$, we write $EG$ for $E\cF$, and set
\begin{align*}
  X_h &\defeq EG_+\tnsr X\\
  X^h &\defeq X^{EG_+}\\
  X^t &\defeq \widetilde{EG}\tnsr X^{EG_+}
\end{align*}
This is because one can show that $(X_h)^H = X_{hH}$ for subgroups $H\le G$, where the right-hand side is a priori defined as orbits for the restriction of $X$ to a Borel spectrum $X\in\Sp^{BH}$, and similarly for $X^h$ and $X^t$.

The other particular case is when $\cF=\cP$ is the family of all \emph{proper} subgroups. In this case, we define the \emph{geometric fixed points} $X^{\Phi G}$ by
\[ X^{\Phi G} \defeq (\widetilde{E\cP} \tnsr X^{E\cP_+})^G. \]
This has the effect of destroying all transfers from subgroups, deleting the ``cheap'' fixed points. For a subgroup $H\le G$, we define $X^{\Phi H}$ by first restricting $X$ to $H\Sp$, then applying the above construction. The formal properties of the different fixed-point functors are:
\begin{align*}
  \Omega^\infty(X^H) &= (\Omega^\infty X)^H\\
  \Sigma^\infty(X^H) &= (\Sigma^\infty X)^{\Phi H}\\
  (X\tnsr Y)^{\Phi H} &= X^{\Phi H} \tnsr Y^{\Phi H}
\end{align*}
In a slogan, the spectra $X^{\Phi H}$ sort fixed points according to their stabilizers, while the spectra $X^H$ mix them all together.

\begin{remark}
Using isotropy separation inductively, one can describe $G$-spectra in terms of their geometric fixed points rather than their categorical fixed points. In place of the diagram perspective of spectral Mackey functors, this describes $G$-spectra in terms of a \emph{stratification}, making it look somewhat like the category of contructible sheaves on some (mythical) stack. This is due to Glasman \cite{Glasman} and Ayala--Mazel-Gee--Rozenblyum \cite{AMGRgenuine}; an excellent exposition is \cite[\sec\sec1.1--1.2]{Wilson}.
\end{remark}

Note that when $G=C_p$, these two distinguished families coincide. Consequently, when $G$ is $C_n$ or $\T$, we can rewrite $(\widetilde{EG}\tnsr X^{EG_+})^G$ as $(X^{\Phi C_p})^{G/C_p}$. This gives a square
\begin{equation}
  \label{eq:isotropy-separation-Cn}
  \vcenter{\xymatrix{
    X_{hG} \ar[r] \ar@{=}[d] & X^G \ar[r] \ar[d] & (X^{\Phi C_p})^{G/C_p} \ar[d]\\
    X_{hG} \ar[r]_-T & X^{hG} \ar[r] & X^{tG}
  }}
\end{equation}
which is the basis for many THH calculations.

\begin{remark}
It may also be useful to consider the fiber of the \emph{vertical} maps in \eqref{eq:isotropy-separation-Cn}, which is $(X^{\widetilde{EG}})^G$. Our interest in this comes from the following observation. The map $\THH = \THH^{\Phi C_p} \morph^\varphi \THH^{tC_p}$ is a spectral version of the derived Frobenius $A \morph^\varphi A \mmod p \defeq A \tnsr^\L_{\Z} \F_p$. If $A$ is a perfectoid ring which is either $p$-torsion or $p$-torsionfree, containing an element $\pi$ such that $\pi^p = pu$ for a unit $u$ (we allow $\pi=0$), then
\[ A \morph^\pi A \morph^\varphi A \mmod p \]
is an exact triangle \cite[Lemma 3.10]{BMS1}. In general, the fiber of $A \morph^\varphi A \mmod p$, which is also the fiber of $W_2(A) \morph^F A$, ``contains all information about $p$-divided powers in $A$''. More generally, the kernel of Frobenius on $W(-)$ is the divided-power completion of $\G_a$ at the origin; this is at the basis of Drinfeld's approach to crystalline cohomology through stacks \cite{DrinfeldCrystallizationA1}.
\end{remark}

Now we define the homotopy theories of interest to us.

\begin{definition}
A \emph{cyclonic spectrum}\footnote{This terminology is due to Barwick-Glasman \cite{BGcyclonic}.} is a $\T$-spectrum genuine for the finite subgroups of $\T$. We denote the $\infty$-category of cyclonic spectra by $\CyclSp$.
\end{definition}

\begin{definition}
A \emph{cyclotomic spectrum} is a naive $\T$-spectrum $X$ together with $\T$-equivariant maps
\[ X \morph^{\varphi_p} X^{tC_p} \]
for all $p$, where $X^{tC_p}$ has the $\T\simeq\T/C_p$ action. These are \emph{not} required to be compatible for varying $p$. We denote the $\infty$-category of cyclotomic spectra by $\CycSp$. We also write $\CycSpp$ for the $\infty$-category of \emph{$p$-typical} cyclotomic spectra, where we only ask for a single $\varphi_p$.
\end{definition}

Classically \cite{HMcyclotomic}, a cyclotomic spectrum was defined as a cyclonic spectrum $X$ equipped with equivalences
\[ X \morph^\sim X^{\Phi C_p} \]
which now are required to be compatible. The homotopy theory of cyclotomic spectra was first constructed in \cite{BMCyclotomic}; the definition above is due to \cite{NikolausScholze}. In the bounded below case, this agrees with the classical notion.

\begin{definition}
A \emph{cyclotomic spectrum with Frobenius lifts} is a naive $\T$-spectrum $X$ together with compatible $\T$-equivariant maps
\[ X \morph^\psi X^{hC_p} \]
for all $p$. We denote the $\infty$-category of cyclotomic spectra with Frobenius lifts by $\CycFL$. We also write $\CycFLp$ for the $\infty$-category of \emph{$p$-typical} cyclotomic spectra with Frobenius lifts, where we only ask for a single $\psi_p$.
\end{definition}

There are forgetful functors
\[ \CycFL \to \CycSp \to \CyclSp \to \Sp^{B\T} \to \Sp. \]

\subsection{THH and friends}
\label{sub:thh-friends}

The author has found that the proliferation of T acronyms in this subject confuses a great number of people---arithmetic geometers and homotopy theorists alike---so in this section we document them all. (To keep the net confusion constant, we also suggest some new ones.) This is merely a collection of definitions; for an introduction to the subject we suggest \cite{KrauseNikolaus}, \cite{HNHandbook}; the modern reference is \cite{NikolausScholze}. The pre-Nikolaus-Scholze surveys \cite{MayTHH} and \cite{Madsen} are also highly recommended.

Recall that ordinary Hochschild homology $\HH(A/k)$ gives an object in $\Mod_k^{B\T}$\footnote{In fact, it even gives a cyclonic $k$-module, but we do not know how to access the fixed-point information in the absence of cyclotomicity.}. Topological Hochschild homology $\THH(A) \defeq \HH(A/\S)$ has more structure, and gives an object of $\CycSp$. To begin, we denote
\begin{alignat*}{2}
  \HC &\defeq \HH_{h\T} &\quad& \text{cyclic homology}\\
  \HC^- &\defeq \HH^{h\T} && \text{negative cyclic homology}\\
  \HP &\defeq \HH^{t\T} && \text{periodic cyclic homology}
\end{alignat*}
These are shown to be equivalent to the classical description in terms of bicomplexes in \cite{Hoyois}. Analogously, we make the following definitions in the topological case:
\begin{alignat*}{2}
  \TC^- &\defeq \THH^{h\T} &\quad& \text{topological negative cyclic homology}\\
  \TP &\defeq \THH^{t\T} && \text{topological periodic cyclic homology}
\end{alignat*}
See Remark \ref{rmk:TM} below. There are exact triangles
\begin{gather*}
  \Sigma\HC \to \HC^- \to \HP\\
  \Sigma\THH_{h\T} \to \TC^- \to \TP
\end{gather*}
The shift comes from working in the compact Lie case.

\begin{warning}
Despite the name, $\TP_*(A)$ is not periodic in general. But this is the case when $A$ lives over a single prime.
\end{warning}

The $\TR^n$ spectra are defined using the categorical fixed points:
\[ \TR^{n+1} \defeq \THH^{C_{p^n}} \quad\text{length } n+1 \text{ topological Restriction homology} \]
\begin{remark}
For any ring $A$, $\TR^{n+1}_0(A) = W_{n+1}(A)$ \cite[Theorem 3.3]{HMFinite}. 
\end{remark}
\begin{remark}
This is more properly called $p$-typical $\TR$, denoted $\TR^{n+1}(-;p)$, whereas ``big $\TR^{n+1}$'' is $\THH^{C_{n+1}}$. We will only use $p$-typical $\TR^{n+1}$ in this paper.

The numbering of $p$-typical $\TR^{n+1}$ is chosen to agree with Witt vectors: for any ring $A$, $\TR^{n+1}_0(A) = W_{n+1}(A)$. As explained in \cite[2.5]{Borger}, the mismatch in indexing is thus the number theorists' fault.
\end{remark}

The $\TR^n$ spectra are related by various maps mimicking the structure of Witt vectors:
\begin{align*}
  \TR^{n+1} &\morph^F \TR^n\\
  \TR^n &\morph^V \TR^{n+1}\\
  \TR^{n+1} &\morph^R \TR^n
\end{align*}
Here $F$ is the equivariant restriction, $V$ is the equivariant transfer, and $R$ comes from the cyclotomic structure. This is interesting: classically one thinks of the $R$ map on Witt vectors as the ``easy'' one and the $F$ map as the ``exotic'' one, whereas the opposite is true from the perspective of equivariant homotopy theory.

With these notations, the isotropy separation sequence \eqref{eq:isotropy-separation-Cn} takes the form
\[\xymatrix{
  \THH_{hC_{p^n}} \ar[r] \ar@{=}[d] & \TR^{n+1} \ar[r]^-R \ar[d] & \TR^n \ar[d]^-\varphi\\
  \THH_{hC_{p^n}} \ar[r] & \THH^{hC_{p^n}} \ar[r] & \THH^{tC_{p^n}}
}\]
We can thus express $\TR^{n+1}$ in Nikolaus-Scholze language as the iterated pullback
\[\xymatrix@1{
  \TR^{n+1} \ar[r] \ar[d] \pullback & \dots \ar[r] \ar[d] \pullback & \TR^3 \ar[r]^-R \ar[dd] \pullback & \TR^2 \ar[r]^-R \ar[d] \pullback \ar[r] & \TR^1 \ar[d]^-\varphi\\
  \vdots \ar[d] & \vdots \ar[d] && \THH^{hC_p} \ar[r]_-\can \ar[d]^-{\varphi^{hC_p}} & \THH^{tC_p}\\
  \vdots \ar[d] & \vdots \ar[d] & \THH^{hC_{p^2}} \ar[r]_-\can \ar[d] & \THH^{tC_{p^2}}\\
  \vdots \ar[d] & \vdots \ar[d] \ar[r] & \vdots\\
  \THH^{hC_{p^n}} \ar[r]_-\can & \THH^{tC_{p^n}}
}\]
This should be compared to the Cartier-Dieudonn\'e-Dwork lemma, and with Borger's perspective on Witt vectors \cite{Borger}.

\begin{remark}
Angeltveit shows there is also an $N$ map, reflecting the Tambara functor structure \cite{AngeltveitNorm}. We make some remarks about this in \sec\ref{sub:prisms}.
\end{remark}

We then define
\begin{alignat*}{2}
  \TR &\defeq \lim_{n,R} \TR^n &\quad& \text{topological Restriction homology}\\
  \TF &\defeq \lim_{n,F} \TR^n && \text{topological Frobenius homology}
\end{alignat*}

More conceptually, $\THH \mapsto \TR$ is the right adjoint to the forgetful functor $\CycFLp \to \CycSpp$, applied to $\THH$ \cite[Proposition 10.3]{KrauseNikolaus}. This is analogous to $W(R)$ being the cofree $\delta$-ring on $R$. Furthermore, $\TF = \Hom_\CyclSp(\S,\THH)$. At infinite level, the isotropy separation sequence \eqref{eq:isotropy-separation-Cn} becomes
\[\xymatrix{
  \Sigma\THH_{h\T} \ar[r] \ar@{=}[d] & \TF \ar[r]^-R \ar[d] & \TF \ar[d]^-\varphi\\
  \Sigma\THH_{h\T} \ar[r] & \TC^- \ar[r]_-\can & \TP
}\]

\begin{remark}
The upper-right copy of $\TF$ behaves differently from the upper-left copy, and is in some sense a Frobenius untwist of it. In calculations we have found it convenient to write the upper-right copy as $\TPhi$, but we fear that trying to introduce this notation would provoke outrage.
\end{remark}

\begin{remark}
It would be natural to define
\[ \TV = \colim_{n, V} \TR^n \quad \text{topological Verschiebung homology} \]
which is a topological analogue of the \emph{unipotent coWitt vectors} \cite[Chapitre II]{Fpdiv}. To our knowledge this has not yet been exploited.
\end{remark}

\begin{openproblem}
Come up with better names for $\TR$, $\TF$, and $\TV$.
\end{openproblem}

Finally, $\TC \defeq \Hom_\CycSp(\S,\THH)$ is obtained by trivializing all the cyclotomic structure. This definition is really a theorem, conjectured by Kaledin \cite{KaledinICM} and proven by Blumberg-Mandell \cite{BMCyclotomic}. There are various equivalent descriptions of this using the above spectra:
\begin{alignat*}{2}
  \TC &\defeq \Hom_\CycSp(\S, \THH) &\quad& \text{topological cyclic homology}\\
  &\phantom:= \fib(\TR \xra{1-F} \TR)\\
  &\phantom:= \fib(\TF \xra{1-R} \TF)\\
  &\phantom:= \fib(\TC^- \xra{\can-\varphi} \TP^\wedge) && \text{``Nikolaus-Scholze formula''}
\end{alignat*}

\begin{remark}
\label{rmk:TM}
It is really $\THH_{h\T}$ which deserves to be called topological cyclic homology, and denoted $\TC$. We have advertised the idea of calling $\Hom_\CycSp(\S,\THH)$ topological motivic cohomology ($\TM$); however, we shall not indulge that here.
\end{remark}

\subsection{Mackey functors}
\label{sub:mackey}

If $M$ is an abelian group acted on by $G$, there are several maps relating the fixed-point modules $M^H$ for subgroups $H\le G$.
\begin{itemize}
\item The Weyl group $W_G(H) = \Aut(G/H) = N_G(H)/H$ acts on each $M^H$.

\item If $K\le H$, there is a restriction map $\res^H_K \colon M^H \to M^K$.

\item If $K\le H$, there is a transfer map $\tr^H_K \colon M^H \to M^K$.
\end{itemize}

The notion of a Mackey functor axiomatizes this structure. In other words, a Mackey functor $\m M$ assigns to each subgroup $H\le G$ an abelian group $\m M(G/H)$, with an action of $W_G(H)$, along with restriction and transfer maps, satisfying appropriate compatibilities. We will not spell these out here; the most important one for us is that
\[ \res^H_K \tr^H_K(x) = \sum_{g\in W_H(K)} g\cdot x \]
A complete, explicit definition can be found in \cite[Definition 1.1.2]{Mazur}. More abstractly, a Mackey functor is an additive functor $\mathcal B_G^\op \to \Ab$. \emph{Mackey functors are to $G$-modules are to abelian groups as genuine $G$-spectra are to naive $G$-spectra are to spectra.}

\begin{example}
For any genuine $G$-spectrum $X$ and $n\in\Z$, the homotopy Mackey functor $\mpi_n(X)$ is defined by
\[ \mpi_n(X)(G/H) = [S^n \tnsr G/H_+, X]. \]
\end{example}

\begin{example}
The \emph{Burnside Mackey functor} $\m A$ is the representable functor $\mathcal B_G(-, G/G)$. For any finite group $G$, $\m A(X)$ is $\K_0$ of the category of finite $G$-sets. In fact, $\m A = \mpi_0\S$.
\end{example}

\begin{example}
Given a $G$-module $M$, the fixed point Mackey functor $\m M$ is defined by $\m M(G/H) = M^H$. Restrictions are inclusions of fixed points, and transfers are summations over cosets. This is a Mackey analogue of the right adjoint to the forgetful functor $G\Sp \to \Sp^{BG}$.
\end{example}

\begin{example}
Given a $G$-module $M$, the orbit Mackey functor $O(M)$ is defined as $O(M)(G/H) = M/H$, the orbit of $H \subset G$. Transfers are quotient maps, and restrictions are summations over representatives. This is a Mackey analogue of the left adjoint to the forgetful functor $G\Sp \to \Sp^{BG}$.
\end{example}

All of our $C_{p^n}$-representations will be restricted from $\T$-representations. Since $\T$ is connected, this means all the Weyl actions will be trivial, and the Mackey functor condition amounts to $\res^G_H\circ\tr^G_H = |G:H|$.

There is a closed symmetric monoidal structure on Mackey functors, which we do not discuss. The monoidal unit is $\m A$ (and not $\m \Z$). A Mackey functor is a $\m\Z$-module if and only if it satisfies the condition $\tr^H_K\res^H_K = |H:K|$ \cite[Proposition 16.3]{ZmodCrit}.

\begin{example}
For any ring $R$, the \emph{Witt Mackey functor} $\m W(R)$ is a $C_{p^n}$-Mackey functor (for any $n$), whose value on $C_{p^n}/C_{p^k}$ is $W_{k+1}(R)$. Restrictions are given by $F$ and transfers are given by $V$. In fact, $\m W(R) = \mpi_0\THH(R)$ \cite[Theorem 3.3]{HMFinite}. $\m W(R)$ should not be confused with the constant Mackey functor $\m{W(R)}$. If $R$ is an $\F_p$-algebra, then $VF = p$ and $\m W(R)$ is a $\m\Z$-module, but this is not true for general $R$.
\end{example}

A common way of describing a Mackey functor is by means of a \emph{Lewis diagram}. If $\m M$ is a Mackey functor, we draw the modules $\m M(G/H)$ with $\m M(G/G)$ on top and $\m M(G/e)$ on the bottom. Restrictions are indicated by maps going downward, and transfers by maps going upward. For example, a Lewis diagram for $C_p$ looks like this:
\[\xymatrix{
    {\m M(C_p/C_p)} \ar@/_1em/[d]_-{\res^{C_p}_e}\\
    {\m M(C_p/e)} \ar@/_1em/[u]_-{\tr^{C_p}_e} \ar@(dl,dr)_-{\text{Weyl group action}}
}\]

All of the Mackey functors we encounter will have trivial Weyl action.

\subsection{\texorpdfstring{$\RO(G)$}{RO(G)}-graded homotopy}
\label{sub:rog-grading}

Let $\alpha$ and $\beta$ be actual representations of $G$. If $X$ is a genuine $G$-spectrum, we define the \emph{$\RO(G)$-graded homotopy $\pi^H_{\alpha-\beta}(X)$ of $X$} as
\[ \pi^H_{\alpha-\beta}(X) = [S^\alpha \tnsr G/H_+, S^\beta \tnsr X] \]

\begin{warning}
This notation is abusive: it depends on the choice of $\alpha$ and $\beta$, not only on $\alpha - \beta$. This is discussed at length in \cite[\sec6]{Adams}. This is not an equivariant peculiarity; even for ordinary spectra, there is a sign involved in trying to identify $S^{|V|}$ with $S^V$ for an abstract real vector space $V$.

We address this by choosing a \emph{specific} irreducible representation in each isomorphism class of irreducible representations: the $\lambda^i$ introduced in \sec\ref{sub:T-repns}. Thus when we say $\RO(\T)$ we really mean $\Z[\lambda^i \mid i \ge 1]$, or $\Z_{(p)}[\lambda_i \mid i \ge 0]$ when we work $p$-locally. Similarly, $\RO(C_{p^n})$ really means $\Z[\lambda^i \mid 1\le i \le \frac{p^n-1}2]$, or $\Z_{(p)}[\lambda_i \mid 0 \le i \le n-1]$ when we work $p$-locally; when $p=2$, we must also include the sign representation $\varsigma_{n-1}$.
\end{warning}

Varying $H$, these fit into a Mackey functor $\mpi_{\alpha-\beta}(X)$. Varying the representation, we obtain the \emph{$RO(G)$-graded homotopy Mackey functor} $\mpi_\rog(X)$, which is the fundamental computational invariant of the $G$-spectrum $X$. It is conventional to use $*$ for $\Z$-grading, and $\rog$ for $\RO(G)$-grading.

A very important source of $\RO(G)$-graded classes is the \emph{Euler classes} $a_V$.

\begin{definition}
Let $V$ be an actual representation of $G$. Suspending the inclusion $\{0\} \hookrightarrow V$ gives a map $S^0 \to S^V$, which we denote by $a_V$. If $V^G \ne 0$, then $a_V$ is nulhomotopic. Otherwise, $a_V$ refines to a class $a_V \in \pi^G_{-V}\S$, and we continue to denote by $a_V$ its Hurewicz image in $\pi^G_{-V} X$ for any $G$-spectrum $X$.
\end{definition}

The significance of these classes is as follows. Consider the isotropy family of $V$,
\[ \cF_V = \{H\le G \mid V^H \ne 0\}. \]
Then $S(\infty V)$ is a model of $E\cF_V$, while $S^{\infty V}$ is a model of $\widetilde{E\cF_V}$. Note that
\[ S^{\infty V} = \colim\left(S^0 \morph^{a_V} S^V \morph^{a_V} S^{2V} \morph^{a_V} \dotsb\right), \]
so smashing with $S^{\infty V}$ has the effect of inverting $a_V$. Consequently, the isotropy separation square \eqref{eq:isotropy-separation} can be viewed as an \emph{arithmetic} square for $a_V$.

\subsection{Representations of \texorpdfstring{$\T$}{T} and \texorpdfstring{$C_{p^n}$}{Cpn}}
\label{sub:T-repns}

Let $\lambda^i$ be the $\T$-representation where $z\in\T\subset\C^\times$ acts as $z^i$. These exhaust the nontrivial irreducible real representations of $\T$. The representation spheres $S^{\lambda^i}$ and $S^{\lambda^j}$ are all integrally inequivalent \cite{Kawakubo}, but in the $p$-local setting we have $S^{\lambda^i} \isom S^{\lambda^j}$ whenever the $p$-adic valuations of $i$ and $j$ agree. Thus, let $\lambda_r = \lambda^{p^r}$. 

Note that $1\in R(\T)$ corresponds to $2\in \RO(\T)$; to avoid confusion, we will often write $\lambda_\infty$ for the one-dimensional \emph{complex} representation with trivial action.

We use the same notation for the restriction of $\lambda^i$ to a subgroup $C_{p^n}$. When $p\ne2$ these exhaust the irreducible real representations of $C_{p^n}$, and $\lambda_r$ is trivial for $r\ge n$. When $p=2$, there is additionally the sign representation of $C_{2^n}$, which we denote $\varsigma_{n-1}$, satisfying $\lambda_{n-1} = 2\varsigma_{n-1}$. The regular representations then decompose as
\begin{align*}
  \rho_{C_{p^n}} &= 1 + \bigoplus_{i=1}^{\frac{p^n-1}2} \lambda^i && p\ne2\\
    &= 1 + \sum_{i=0}^{n-1} \frac{p^{n-1-i}(p-1)}2 \lambda_i && \text{$p$-locally}\\
  \rho_{C_{2^n}} &= 1 + \varsigma_{n-1} + \bigoplus_{i=1}^{2^{n-1}-1} \lambda^i && p=2\\
  &= 1 + \varsigma_{n-1} + \sum_{i=0}^{n-2}2^{n-2-i}\lambda_i &&\text{2-locally}
\end{align*}

\begin{warning}
In the literature, one sometimes sees the notation $\lambda_d$ for what we call $\{d\}_\lambda$.
\end{warning}

As $C_n$ spectra, the representation spheres have cell structures
\[\xymatrix{
  S^0 \ar[r] & S^\spoke \ar[r] \ar[d] & S^{\lambda_r} \ar[d]\\
  & S^1 \tnsr C_{p^n}/C_{p^r+} \ar[r] & (\dots) \ar[d]\\
  && S^2 \tnsr C_{p^n}/C_{p^r+}
}\]
The attaching map $S^1\tnsr C_{p^n}/C_{p^r+} \to S^1 \tnsr C_{p^n}/C_{p^r+}$ is given by $1-\gamma$, where $\gamma$ is a chosen generator of $C_{p^n}$. The one-skeleton $S^\spoke$ is not actually a representation sphere, except in the case $G=C_{2^n}$, when $S^{\lambda_{n-1}/2} = S^{\varsigma_{n-1}}$. Taking duals gives a dual cell structure
\[\xymatrix{
  S^{-2} \tnsr C_{p^n}/C_{p^r+} \ar[r] & (\dots) \ar[d] \ar[r] & S^{-\lambda_r} \ar[d]\\
  & S^{-1} \tnsr C_{p^n}/C_{p^r+} \ar[r] & S^{-\spoke} \ar[d]\\
  && S^0
}\]
on $S^{-\lambda_r}$. Tensoring these together for various values of $r$, and using the fact that $S^{\lambda_s} \tnsr C_{p^n}/C_{p^r+} = S^2 \tnsr C_{p^n}/C_{p^r+}$ for $r\le s$, gives cellular structures for all $\alpha\in\RO(C_{p^n})$, and hence spectral sequences to compute $\mpi_\alpha X$ for any $C_{p^n}$-spectrum $X$. This is discussed in \cite[\sec1.2]{HHR_HZ}. When $\alpha$ is an actual representation (or the negative of one), the spectral sequences become chain complexes of Mackey functors, which are determined by the underlying chain complex of abelian groups.

\begin{example}
Let $\m M$ be a $C_p$-Mackey functor, and write $\m M(C_p/C_p) = M_1$, $\m M(C_p/e) = M_0$, and trivial Weyl action on $M_0$. Then $\mpi_{2\lambda_0+*}\m M$ is the homology of the following complex:
\[\xymatrix{
  M_1 \ar@/_1ex/[d]_-\res \ar[r]^-\res & M_0 \ar@/_1ex/[d]_-\Delta \ar[r]^-0 & M_0 \ar@/_1ex/[d]_-\Delta \ar[r]^-p & M_0 \ar@/_1ex/[d]_-\Delta \ar[r]^-0 & M_0 \ar@/_1ex/[d]_-\Delta\\
  M_0 \ar@/_1ex/[u]_-\tr \ar[r]_-\Delta & M_0[C_p] \ar@/_1ex/[u]_-\nabla \ar[r]_-{1-\gamma} & M_0[C_p] \ar[r]_-T \ar@/_1ex/[u]_-\nabla & M_0[C_p] \ar@/_1ex/[u]_-\nabla \ar[r]_-{1-\gamma} & M_0[C_p] \ar@/_1ex/[u]_-\nabla\\
  \ss 0 & \ss-1 & \ss-2 & \ss-3 & \ss-4
}\]
\end{example}

\begin{example}
\label{ex:rog-m-neg}
Let $\m M$ be a $C_{p^2}$-Mackey functor, and write $\m M(C_{p^2}/C_{p^2}) = M_2$, $\m M(C_{p^2}/C_p) = M_1$, $\m M(C_{p^2}/e) = M_0$, all with trivial Weyl action. Then $\mpi_{*-\lambda_0-\lambda_1}\m M$ is the homology of the following complex:
\[\xymatrix{
  M_0 \ar@/_1ex/[d]_-\Delta \ar[r]^-0 & M_0 \ar@/_1ex/[d]_-\Delta \ar[r]^-\tr & M_1 \ar@/_1ex/[d]_-\Delta \ar[r]^-0 & M_1 \ar@/_1ex/[d]_-\Delta \ar[r]^-\tr & M_2 \ar@/_1ex/[d]_-\res\\
  M_0[C_{p^2}/C_p] \ar@/_1ex/[d]_-\Delta \ar@/_1ex/[u]_-\nabla \ar[r]^-{1-\gamma} & M_0[C_{p^2}/C_p] \ar@/_1ex/[u]_-\nabla \ar@/_1ex/[d]_-\Delta \ar[r]^-\tr & M_1[C_{p^2}/C_p] \ar@/_1ex/[u]_-\nabla \ar@/_1ex/[d]_-\res \ar[r]^-{1-\gamma} & M_1[C_{p^2}/C_p] \ar@/_1ex/[u]_-\nabla \ar@/_1ex/[d]_-\res \ar[r]^-\nabla & M_1 \ar@/_1ex/[u]_-\tr \ar@/_1ex/[d]_-\res\\
  M_0[C_{p^2}] \ar@/_1ex/[u]_-\nabla \ar[r]_-{1-\gamma} & M_0[C_{p^2}] \ar@/_1ex/[u]_-\nabla \ar[r]_-\nabla & M_0[C_{p^2}/C_p] \ar@/_1ex/[u]_-\tr \ar[r]_-{1-\gamma} & M_0[C_{p^2}/C_p] \ar@/_1ex/[u]_-\tr \ar[r]_-\nabla & M_0 \ar@/_1ex/[u]_-\tr\\
  \ss 4 & \ss3 & \ss2 & \ss1 & \ss0
}\]
\end{example}

These cell structures imply the following very useful bound on homology.
\begin{lemma}
\label{lem:mack-bound}
Let $M$ be a $\T$-Mackey functor and let $\alpha$ be a fixed-point-free virtual representation of $\T$. Write $\alpha=\beta - \gamma$ for actual representations $\beta$, $\gamma$ (still assumed to be fixed-point-free). Then $\mpi_{*+\alpha}\m M$ is concentrated in $*\in[-2d_0(\beta), 2d_0(\gamma)]$.
\end{lemma}

An important observation, which we learned from Mike Hill, is that these have simpler cell structures \emph{as $\T$-spectra}.
\begin{observation}
In $\CyclSp$ there is an exact triangle
\[ \T/C_{p^r+} \to S^0 \to S^{\lambda_r}. \]
\end{observation}

\subsection{The slice filtration}
\label{sub:slice-filtration}

The equivariant slice filtration is a filtration on equivariant spectra first introduced by Dugger in the $C_2$ case \cite{Dugger}, then generalized to finite $G$ by Hill-Hopkins-Ravenel as a key tool in their solution of the Kervaire invariant one problem \cite{HHR}. It is modeled on the motivic slice filtration of Voevodsky, hence the qualifier ``equivariant''. For an introduction to the slice filtration, the original \cite[\sec4]{HHR} is still highly recommended, as well as the survey \cite{Primer}. However, significant advances have been made since then: \cite{NewSlices} gave a much easier characterization of slice connectivity, and \cite{Wilson} provided an algebraic description of categories of slices, as well as a general recipe for computing slices. A thorough treatment of the slice spectral sequence is given in \cite{Ullman}, and of course there are many computations in the literature one can learn from, such as \cite{HHR_HZ} or \cite{YarnallCyclic}.

We caution the reader that there is no single slice filtration; \cite{HHR} and \cite{Primer} use the \emph{classical} slice filtration, whereas \cite{NewSlices} and \cite{Ullman} treat the \emph{regular} slice filtration. We shall be concerned with the regular slice filtration. A general framework for slice filtrations is given in \cite[\sec1.3]{Wilson}.

\begin{definition}
Let $G$ be a finite group. A \emph{regular slice cell of dimension $n$} is a $G$-spectrum of the form
\[ \uparrow^G_H S^{k\rho_H}, \]
where $H$ is a subgroup of $G$, $\rho_H$ is the regular representation of $H$, and $k|H|=n$. The \emph{regular slice filtration} is the filtration generated by the regular slice cells. More explicitly,
\begin{itemize}
\item We say that a $G$-spectrum $X$ is \emph{slice $n$-connective}, and write $X\ge n$, if $X$ is in the localizing subcategory generated by the regular slice cells of dimension $\ge n$. Equivalently \cite{NewSlices}, $X$ is slice $n$-connective if and only if for all subgroups $H\le G$, the geometric fixed points $X^{\Phi H}$ are in the localizing subcategory of ordinary spectra generated by $S^{\lceil n/|H|\rceil}$.

We write $G\Sp_{\ge n}$ for the slice $n$-connective spectra. This is a coreflective subcategory, and we denote by $\regconn n X \to X$ the coreflection, the \emph{slice $n$-connective cover of $X$}.

\item We say that a $G$-spectrum $X$ is \emph{slice $n$-truncated} and write $X\le n$ if $\Map_{G\Sp}(Y, X)=*$ whenever $Y\ge n+1$. It suffices to check this for $Y$ a regular slice cell of dimension $\ge n+1$.

The slice $n$-truncated spectra are reflective, and we denote by $X \to \regtrunc n X$ the localization functor, the \emph{slice $n$-truncation of $X$}.

\item We say that a $G$-spectrum $X$ is an \emph{$n$-slice} if $X\ge n$ and $X\le n$. There is an exact triangle $\regconn{n+1} X \to X \to \regtrunc n X$, whence a canonical equivalence $\regtrunc n\regconn n X = \regconn n\regtrunc n X$. We write $\regslice n X$ for either of these, the \emph{$n$-slice of $X$}.

\item The \emph{slice spectral sequence} takes the form
\[ E^2_{s,t} = \mpi_{t-s} \regslice t X \conv \mpi_{t-s} X. \]
This follows Adams grading, so the $E^r_{s,t}$ term is placed in the plane in position $(t-s,s)$. The $d_r$ differential has bidegree $(r, r-1)$, or $(-1, r)$ in terms of the plane display. There is also an $\RO(G)$-graded version
\[ E^2_{s,t} = \mpi_{\alpha-s} \regslice{\dim\alpha} X \conv \mpi_{\alpha-s} X. \]
\end{itemize}
\end{definition}

The slice filtration is generally not associated to a $t$-structure, but rather a sequence of $t$-structures \cite[Definition 1.43]{Wilson}.

The above definition is only for finite groups, and it is not immediately clear how to interpret the slice filtration for cyclonic or cyclotomic spectra. We have chosen to interpret this as the $C_{p^n}$-slice filtration on the restriction to a $C_{p^n}$-spectrum for all $n$, but we do not claim this is the only or best option.

\begin{warning}
The slice covers $\regconn n$ give a \emph{descending} filtration, while the slice truncations $\regtrunc n$ give an \emph{ascending} filtration. This is opposed to the usual super/subscript convention for filtrations.
\end{warning}

\begin{remark}
In \cite{TCart}, Antieau and Nikolaus produce a $t$-structure on $\CycSp$ by stipulating that the forgetful functor $\CycSp \to \Sp$ reflect $n$-connective objects. A natural idea is to transport the slice filtration along $\CycSp \to \CyclSp$ in the same way. However, the Hill-Yarnall characterization implies that the slice connectivity of a cyclotomic spectrum is equal to the connectivity of its underlying spectrum, so this reproduces the cyclotomic $t$-structure.
\end{remark}

\newpage

\section{Arithmetic background}
\label{sec:bg-arith}
In this section we collect the needed background from number theory. In \sec\ref{sub:q-an} we introduce $q$-analogues, and survey several key instances where these appear; in particular, we repeatedly encounter the problem of $p$-typification. This discussion is the core of the narrative of the paper. In \sec\ref{sub:perfd} we recall the notion of perfectoid ring and Fontaine's ring $\Ainf$, followed by the homotopical properties of perfectoid rings. Finally, in \sec\ref{sub:prisms} we discuss the reformulation of perfectoid rings as perfect prisms, and explore some homotopical features of prisms.

\subsection{\tops{$q$}q-analogues and \tops{$p$}p-typification}
\label{sub:q-an}

\begin{definition}
Let $n\in\N$. The \emph{$q$-analogue of $n$} is the formal expression
\[ [n]_q \defeq \frac{q^n - 1}{q-1} = 1 + \dotsb + q^{n-1} \in \Z[q]. \]
\end{definition}
These give a one-parameter deformation of the natural numbers, and many notions in mathematics admit so-called \emph{$q$-deformations}, recovering the classical notion when $q=1$. Such $q$-deformations arise naturally in counting problems over finite fields, combinatorics, and a wide variety of other contexts. The relevance to $\THH$ is ultimately that $\TP$ can be used to construct $q$-de Rham cohomology, as envisioned in \cite{ScholzeQ} and carried out in \cite{BMS2}.

For us, the key feature of $\Z[q]$ is that it is a \emph{$\Lambda$-ring}: we have commuting ring endomorphisms $\psi^n(q) = q^n$ for $n\in\N^\times$, which are \emph{Frobenius lifts} in the sense that
\[ \psi^p(x) = x^p + p\delta_p(x) \]
for some (unique, since $\Z[q]$ is torsionfree) $\delta_p(x)\in\Z[q]$, for all primes $p$. This gives an action $n\mapsto\psi^n$ of $\N^\times$ on $\Z[q]$ in the category of commutative rings. This action interacts with $q$-analogues in two ways:
\begin{enumerate}
\item The map $\N^\times\to(\Z[q],\times)$ sending $n\mapsto[n]_q$ is not a map of commutative monoids, but rather is semi-multiplicative with respect to this action, in the sense that
\[ [mn]_q = [m]_q \psi^m([n]_q) = [n]_q\psi^n([m]_q) \]
for $m,n\in\N^\times$. Note that $[m]_q$ still divides $[mn]_q$.

\item For $m,n,k\in\N^\times$, there is a congruence
\begin{align*}
  \psi^{mk}([n]_q) &\equiv n\bmod [m]_q
\end{align*}
since $q^m\equiv1\bmod[m]_q$. We will explore the significance of this congruence in \sec\ref{sub:prisms}; it enforces a subtle constraint on the way in which $n\mapsto[n]_q$ deviates from being truly multiplicative.
\end{enumerate}
We will mostly restrict attention to the submonoid $p^\N\subset\N^\times$ and write $\phi\defeq\psi^p$, $\delta\defeq\delta_p$ (although it would be good to formulate some of this more globally). Note that $\phi^n$ thus means $\psi^{p^n}$. Prisms can be viewed as an axiomatization and generalization of the $p$-typical part of the above structure.

We now examine several case studies involving $q$-analogues.

\subsubsection{Perfectoid algebras}
This concerns the specialization
\[ \Z[q] \to \Z_p[q^{1/p^\infty}]^\wedge_{(p,q-1)}; \]
the reference for this section is \cite[Example 3.16]{BMS1} and the next few propositions.

The starting point is the observation that $[p]_q$ is the minimal polynomial of $\zeta_p$, and thus $\Z_p[q]/[p]_q = \Z_p[\zeta_p]$. We think of this as ``characteristic close to $p$''; note that $\zeta_p=1$ implies $p=0$. Similarly, $\phi^n([p]_q)$ is the minimal polynomial of $\zeta_{p^{n+1}}$, so $\Z_p[q]/\phi^n([p]_q)=\Z_p[\zeta_{p^{n+1}}]$. Thus, taking perfections yields
\[\xymatrix{
  \Z_p[q] \ar[d] \ar[r]^-\phi & \Z_p[q] \ar[d] \ar[r]^-\phi & \dotsb \ar[d] \ar[r] & \Z_p[q^{1/p^\infty}] \ar[d]\\
  \Z_p[\zeta_p] \ar[r] & \Z_p[\zeta_{p^2}] \ar[r] & \dotsb \ar[r] & \Z_p[\zeta_{p^\infty}]
}\]
We want $q=1$ to be the ``base case'', so we complete at $q-1$ and let $q^{1/p}$, rather than $q$, be a primitive $p^\th$ root of unity. Thus we set
\begin{align*}
  A &= \Z_p[q^{1/p^\infty}]^\wedge_{(p,q-1)}\\
  [n]_A &= [n]_{q^{1/p}}\\
  R &= A/[p]_A
\end{align*}
We warn that $A/\phi^n([p]_A)$ is now isomorphic to $\Zpcycl=\Z_p[\zeta_{p^\infty}]^\wedge_p$ for all $n$; however, these have different $A$-algebra structures, namely $q\mapsto\zeta_{p^{n-1}}$.

What about the other $q$-analogues? It turns out that there are canonical identifications
\[ A/[p^n]_A \iso W_n(R) \]
such that the projections $A/[p^{n+1}]_A\to A/[p^n]_A$ are identified with the Frobenius maps $F\colon W_{n+1}(R)\morph W_n(R)$. (We would get a different $A$-algebra structure if we wanted the $R$ maps to be $A$-linear.) Since the $A$-algebra structure is important to us, we will almost exclusively write $A/[p^n]_A$ rather than $W_n(R)$.

The ring $R$ is an example of a \emph{perfectoid ring} and the triple $(A,\phi,[p]_A)$ is the corresponding \emph{perfect prism}. It suffices to think of this example for the remainder of the paper (except when we explicitly restrict to perfect $\F_p$-algebras).

\subsubsection{Legendre's formula}
Throughout this paper we will need to count the number of integers in $\{1,\dotsc,n\}$ having a particular $p$-adic valuation. One instance of this is the following classical formula for the $p$-adic valuation of a factorial.
\begin{proposition}[Legendre]
The $p$-adic valuation of $n!$ is given by
\begin{align*}
  v_p(n!) &= \sum_{r=1}^\infty \lf\frac n{p^r}\rf
\end{align*}
\end{proposition}




The $q$-factorial is defined by
\[ [n]_q!\defeq[1]_q\dotsm[n]_q. \]
Ansch\"utz-le Bras have supplied a $q$-deformation of Legendre's formula in their work identifying the cyclotomic trace in degree 2 as a $q$-logarithm.
\begin{lemma}[{\cite[Lemma 4.8]{AClBTrace}}; ``$q$-Legendre formula'']
\label{lem:q-lgdnr}
In the specialization $\Z[q] \to \Z_p\psr{q-1}$, we have
\begin{align*}
  [n]_q!
    &= u\prod_{r=1}^\infty \phi^{r-1}([p]_q)^{\lf n/p^r \rf}\\
    &= u\prod_{r=1}^\infty [p^r]_q^{\lf n/p^r\rf - \lf n/p^{r+1}\rf}
\end{align*}
for a unit $u\in\Z_p\psr{q-1}^\times$.
\end{lemma}
In what follows, we will apply this to $[n]_A!$ rather than $[n]_q!$.

An illustration of the $q$-Legendre formula is given in Figures \ref{fig:q-lgndr-2} and \ref{fig:q-lgndr-3}. In the classical Legendre formula, all bars would be the same color. However, as we have already noted, $q$-analogues are only semimultiplicative: $[p^n]_q = [p]_q\dotsm\phi^{n-1}([p]_q)$ rather than $[p]_q^n$.

\begin{figure}
  \begin{center}
  \includegraphics[width=\textwidth]{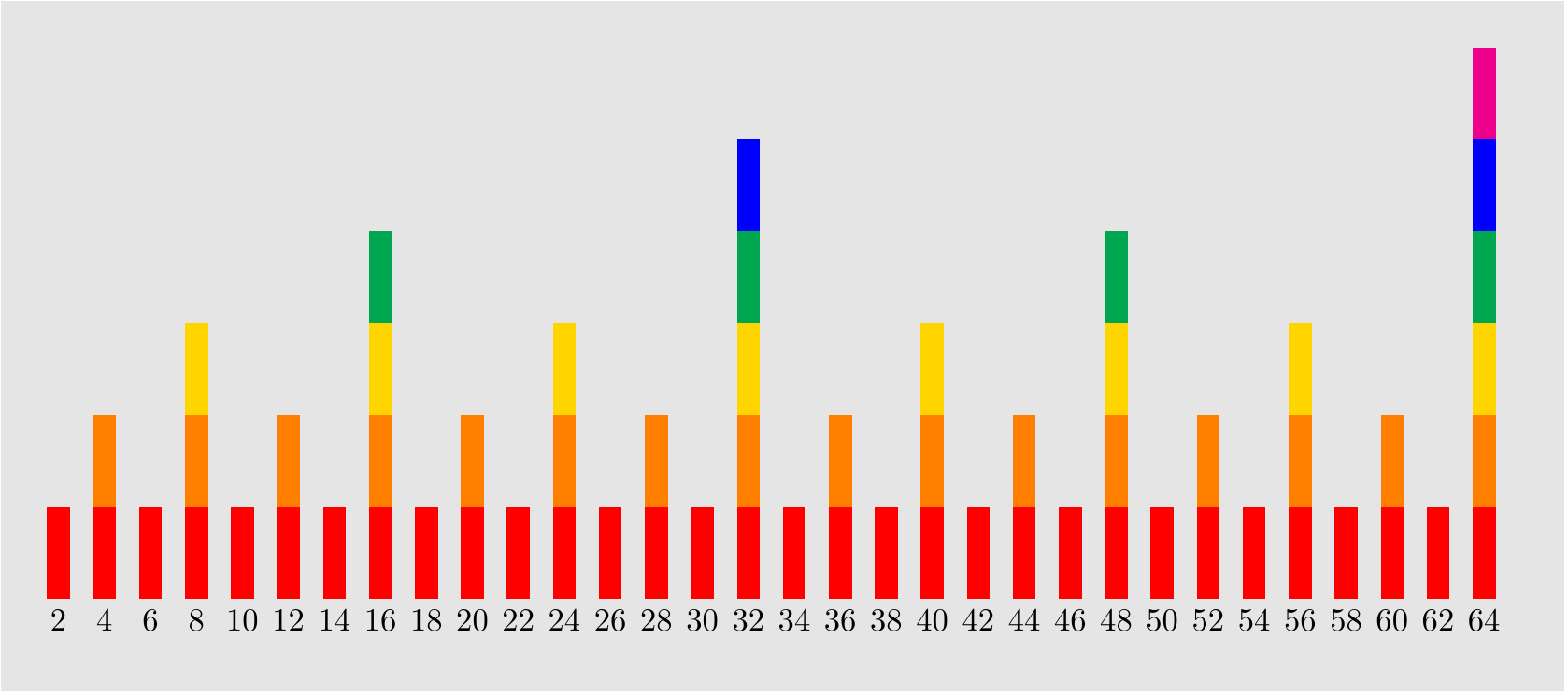}
  \end{center}
  \caption{The $q$-Legendre formula at $p=2$}
  \label{fig:q-lgndr-2}
\end{figure}

\begin{figure}
  \begin{center}
  \includegraphics[width=\textwidth]{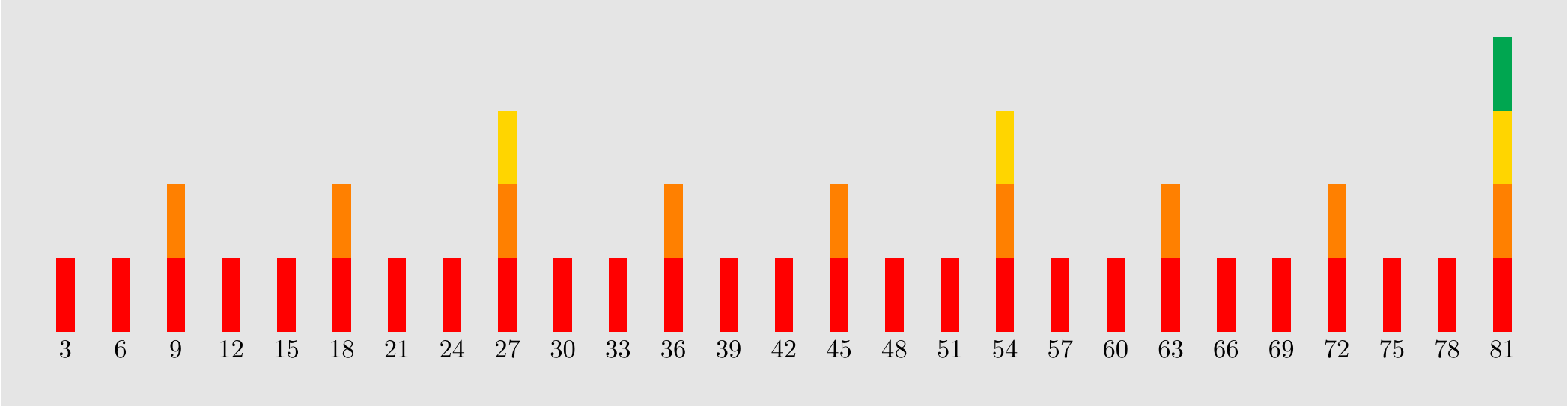}
  \end{center}
  \caption{The $q$-Legendre formula at $p=3$}
  \label{fig:q-lgndr-3}
\end{figure}

In \sec\ref{sub:thh-slices}, we will see that the appearance of the floor function in the preceding formulas is intimately linked to the appearance of the ceiling function in the slice filtration. The identity
\[ \left\lceil\frac{n+1}{p^k}\right\rceil - 1 = \left\lfloor\frac n{p^k}\right\rfloor \]
is useful for translating between the two.

\begin{remark}
Bhargava \cite{Bhargava!} has an interesting framework for generalized factorials.
\end{remark}

\subsubsection{Circle representations} This concerns the specialization
\[ \Z[q] \to \Z[\lambda^\pm] = R(\T), \]
where $R(\T)$ is the complex representation ring of the circle group $\T$. In this case, $\lambda$-analogues are something familiar:
\[ \mathop{\downarrow^\T_{C_n}} [n]_\lambda = \C[C_n] = \aInd e{C_n}\C \]
is another name for the complex regular representation of $C_n$. We prefer the notation $[n]_\lambda$ as it fits with the overall theme of $q$-analogues. An important consequence is that
\[ \aRes\T{C_n} S^{[n]_\lambda} = N^{C_n}_e S^2, \]
so $S^{[n]_\lambda}$ is $\ge2n$ by \cite[Corollary I.5.8]{Ullman}; this is ultimately the reason for our interest in $\lambda$-analogues. We will also use the (admittedly atrocious) notation
\begin{align*}
  \{n\}_\lambda
    &\defeq \lambda^1+\dotsb+\lambda^n\\
    &= [n+1]_\lambda - \lambda_\infty\\
    &= \lambda[n-1]_\lambda
\end{align*}
which has the property that $\mathop{\downarrow^\T_{C_{n+1}}}\{n\}_\lambda$ is the complex \emph{reduced} regular representation of $C_{n+1}$.

As noted earlier, when working $p$-locally it suffices to consider the representations $\lambda_i \defeq \lambda^{p^i}$.

\begin{observation}
\label{obs:q-lgndr}
Determining the decomposition of $\{n\}_\lambda$ into $\lambda_i$'s is isomorphic to the problem considered in the $q$-Legendre formula. More precisely, it corresponds to decomposing $[pn]_A!=[p]_A^n \phi([n]_A!)$, since we always pick up a new irreducible representation when passing from $\{n\}_\lambda$ to $\{n+1\}_\lambda$, whereas we only pick up new (non-unit) factors in $[n]_A!$ when we hit multiples of $p$.
\end{observation}
For example, with $p=3$ the representation $\{4\}_\lambda=\lambda+\lambda^2+\lambda^3+\lambda^4$ has $p$-typical decomposition
\begin{align*}
  \{4\}_\lambda &= 3\lambda_0 + \lambda_1\\
  \intertext{and dimension-sequence}
  d_\bullet(\{4\}_\lambda) &=(4,1).\\
  \intertext{This parallels the decomposition}
  [4p]_A!
    &= u[p]_A^3 [p^2]_A\\
    &= u[p]_A^4 \phi([p]_A).
\end{align*}

\begin{corollary}
\label{cor:irred-decomp}
The dimension-sequences of the representations $[n]_\lambda$ and $\{n\}_\lambda$ are
\[
  d_s([n]_\lambda) = \lc\frac n{p^s}\rc,\qquad
  d_s(\{n\}_\lambda) = \lf\frac{n}{p^s}\rf,
\]
and their $p$-typical irreducible decompositions are
\begin{align*}
  [n]_\lambda &= \lambda_\infty + \sum_{s=0}^\infty\left(\lc\frac n{p^s}\rc - \lc\frac n{p^{s+1}}\rc\right) \lambda_s,\\[.4em]
  \{n\}_\lambda &= \sum_{s=0}^\infty \left(\lf\frac n{p^s}\rf - \lf\frac n{p^{s+1}}\rf\right) \lambda_s.
\end{align*}
\end{corollary}

\begin{remark}
The representation $[n]_\lambda$ is also familiar to homotopy theorists as the \emph{Bott cannibalistic class} $\rho^n(\lambda)$ arising in topological $K$-theory. In fact, since the Bott element is given by $\beta=\lambda-1$, we can write $\pi_*\ku=\Z\psr{\lambda-1}$. Thus, another way to think of $\lambda$-analogues and $\lambda$-deformations is that we are deforming the trivial line bundle to a general line bundle.

This is related to the present situation as follows. By \cite{SuslinKOC} we have $\K(\O_{\C_p};\Z_p)=\ku_p$, and the cyclotomic trace \[ \pi_2\ku_p = \K_2(\O_{\C_p};\Z_p)\to\TP_2(\O_{\C_p};\Z_p) \] sends $\lambda$ to (a possible choice of) $q$ by \cite[Lemma 3.2.3]{LarsOCp}. This can be generalized to $\Zpcycl$ algebras, see \cite[Example 5.5]{MathewK1}.
\end{remark}

\subsubsection{Tate cohomology}
Our final example is more philosophical in nature. It concerns the specialization
\[ \Z[q] \to \Z[q]/(q^n - 1) = \Z[C_n] \]
in which $q$ is viewed as a generator of $C_n$. In this case $[n]_q$ gets sent to the $C_n$-\emph{trace}, so that quotienting by $[n]_q$ essentially amounts to taking Tate cohomology. This suggests yet another perspective on $q$-deformations, one which may appeal to the equivariant homotopy theorist: we are deforming from a trivial action to a non-trivial action, such that multiplication by $n$ gets deformed to a transfer for a subgroup of index $n$.

One would then hope that the equivariant \emph{norm} for a subgroup of index $n$ would correspond to some $q$-deformation of raising to the power $n$. At present, we are only able to relate the norm to existing notions of $q$-powers on elements of rank one. We record this in \sec\ref{sub:prisms}.


\subsection{Perfectoid rings}
\label{sub:perfd}

\subsubsection{Arithmetic aspects}
The reference for this section is \cite[\sec\sec 3--4]{BMS1}, see also \cite[Lecture IV]{PrismNotes}.

\begin{definition}
Let $R$ be a $p$-complete ring. The \emph{tilt} $R^\flat$ of $R$ is the inverse limit perfection of $R/p$:
\begin{align*}
  R^\flat &\defeq \lim_\phi R/p.
  \intertext{There is a multiplicative bijection}
  R^\flat &\isom \lim_{x\mapsto x^p} R,
\end{align*}
and we write $(-)^\sharp\colon R^\flat\to R$ for the multiplicative map $(x_0,x_1,\dotsc)\mapsto x_0$.
\end{definition}

\begin{example}
If $R=\Zpcycl\defeq\Z_p[\zeta_{p^\infty}]^\wedge_p$, then $R/p=\F_p[x^{1/p^\infty}]/x$ and $R^\flat = \F_p\psr{x^{1/p^\infty}}$. The $\sharp$ map is given by $x^\sharp=(\zeta_p-1)^p$.

We get the same $R/p$ and $R^\flat$ by taking instead $R=\Z_p[p^{1/p^\infty}]^\wedge_p$. In this case the $\sharp$ map is given by $x^\sharp=p$.
\end{example}

Tilting is right adjoint to the functor of ($p$-typical) Witt vectors:
\[ \adjnctn{\{\text{perfect }\F_p\text{-algebras}\}}{\{p\text{-complete }\Z_p\text{-algebras}\}}W{(-)^\flat} \]
Morally, we think of this adjunction as extension/restriction of scalars along the Teichm\"uller lift $\F_p\to\Z_p$ (which is a map of multiplicative monoids, but not a map of rings).

\begin{definition}
We write $\Ainf(R)\defeq W(R^\flat)$, and $\theta\colon\Ainf(R)\to R$ for the counit of this adjunction. By functoriality of Witt vectors, $\Ainf(R)$ inherits a Frobenius $\phi$ from $R^\flat$ (even though there is usually no Frobenius on $R$).
\end{definition}

Explicitly, every element of $\Ainf(R)$ has the form
\[
  x = \sum_{i=0}^\infty [a_i] p^i
\]
with $a_i\in R^\flat$, and the maps $\phi$ and $\theta$ are given by
\[
  \phi(x) = \sum_{i=0}^\infty [a_i^p] p^i,\qquad
  \theta(x) = \sum_{i=0}^\infty a_i^\sharp p^i.
\]

\begin{example}
\label{ex:zpcycl-q}
If $R=\Zpcycl$, then $\Ainf(R)=\Z_p[q^{1/p^\infty}]^\wedge_{(p,q-1)}$ is the ring we encountered in the previous section. The element $q$ is constructed as follows. Fix a compatible choice of $\{\zeta_{p^i}\}$, and let
\begin{align*}
  \epsilon &= (1,\zeta_p, \zeta_{p^2}, \dots)\in R^\flat\\
  q &= [\epsilon] \in \Ainf(R) = W(R^\flat)
\end{align*}
We see that $\theta(q^{1/p})=\zeta_p$, so $\theta(q)=1$ and $\ker\theta$ is generated by $[p]_{q^{1/p}}=\frac{q-1}{\phi^{-1}(q-1)}$.

In particular, $\Ainf(R)$ is a $\Z_p[q^{1/p^\infty}]^\wedge_{(p,q-1)}$-algebra for any $\Zpcycl$-algebra $R$. Traditionally $q-1$ is denoted $\mu$, $[p]_{q^{1/p}}$ is denoted $\xi$, and $[p]_q$ is denoted $\~\xi$.
\end{example}

\begin{remark}
$q$ is roughly a $p$-adic analogue of $e^{2\pi i}$. Note that $e^{2\pi i}=1$, but we can formally write $(e^{2\pi i})^{1/p} = e^{2\pi i/p} = \zeta_p$ whereas $1^{1/p} = 1$. Keeping track of an inverse system of $p$-power roots allows us to make this formal manipulation precise, and interpret $e^{2\pi i}$ as an expression (like $q$) whose ``underlying element'' is $1$, but which contains the information of the $\zeta_{p^\infty}$.
\end{remark}

\begin{definition}[{\cite[Definition 3.5]{BMS1}}]
A \emph{perfectoid ring} is a ring $R$ satisfying:
\begin{enumerate}[label=(\alph*)]
\item there is some $\pi\in R$ such that $\pi^p$ divides $p$;

\item $R$ is $\pi$-complete;

\item the Frobenius $R \morph^\phi R/p$ is surjective;

\item the kernel of $\Ainf(R) \morph^\theta R$ is principal.
\end{enumerate}
If $R$ is $p$-torsionfree, the last condition may be replaced by:
\begin{enumerate}
\item[(d')] if $x\in R[\frac1p]$ with $x^p\in R$, then $x\in R$.
\end{enumerate}
\end{definition}

\begin{example}
An $\F_p$-algebra is a perfectoid ring if and only if it is perfect; in this case $\pi=0$, and $\Ainf(R)=W(R)$. $\Z_p^{\mathrm{cycl}}$ and $\Z_p[p^{1/p^\infty}]^\wedge_p$ are both perfectoid. If $C$ is a complete algebraically closed extension of $\Q_p$, then $\O_C$ is a perfectoid ring; in particular, $\O_C$ is a $\Zpcycl$-algebra.

Some examples of rings which are not perfectoid are $\Z_p$ (a), $\Z_p[p^{1/p}]$ (c), and $\Z_p\psr{x^{1/p^\infty}}/x$ (d').
\end{example}

\begin{remark}
We are mainly interested in perfectoid rings which are either $p$-torsion or $p$-torsionfree. There is a fracture theorem expressing any perfectoid ring as a (homotopy) pullback of such perfectoid rings \cite[Proposition IV.3.2]{PrismNotes}, so we are justified in restricting to such.
\end{remark}

\begin{remark}
If $R$ is perfectoid, then $\Ainf(R)\morph^\theta R$ is the universal pro-infinitesimal formal $p$-adic thickening of $R$ \cite[Th\'eor\`eme 1.2.1]{FontaineCorps}. This is the reason for the notation $\Ainf$.
\end{remark}

Assume from now on that $R$ is perfectoid. In this case, there is an alternative description of $\Ainf$ which is very useful. Although by definition
\[ \Ainf(R) \defeq \lim_{n, R} W_n(R^\flat), \]
it turns out \cite[Lemma 3.2]{BMS1} that there is also a canonical isomorphism
\begin{equation}\label{eq:Ainf-F}
\Ainf(R) \isom \lim_{n, F} W_n(R).
\end{equation}
This is very important for the topological story: $\TR_0(R) = W(R)$, but $\TF_0(R) = \Ainf(R)$.

Under the isomorphism \eqref{eq:Ainf-F}, we define $\~\theta_n\colon \Ainf(R) \morph W_n(R)$ to be the projection. Explicitly, for $x = (x^{(0)}, x^{(1)}, \dots) \in R^\flat$, we have $\~\theta_n([x]) = [x^{(n)}]$. The map $\theta$ introduced above is the same as $\~\theta_1\phi$, and in fact we will mainly be concerned with $\~\theta_n\phi$ (which is relevant to $\TF$) rather than $\~\theta_n$ (which is relevant to $\TP$). We also will not use the maps $\theta_n\defeq\~\theta_n\phi^n$ (for $n>1$), which are relevant to $\TR$.

The crucial point for us is how these maps interact with the $F$, $R$, $V$ operators on $W_n(R)$.

\begin{lemma}[{\cite[Lemma 3.4]{BMS1}}]
\label{lem:theta-compats}
The following diagrams commute:
\[
  \vcenter{\xymatrix{
    \Ainf(R) \ar[r]^-{\~\theta_{n+1}} \ar[d]_-{\phi^{-1}} & W_{n+1}(R) \ar[d]^-{R}\\
    \Ainf(R) \ar[r]_-{\~\theta_n} & W_n(R)
  }}
  \quad
  \vcenter{\xymatrix{
    \Ainf(R) \ar[r]^-{\~\theta_{n+1}} \ar@{=}[d] & W_{n+1}(R) \ar[d]^-F\\
    \Ainf(R) \ar[r]_-{\~\theta_n} & W_n(R)
  }}
  \quad
  \vcenter{\xymatrix{
    \Ainf(R) \ar[r]^-{\~\theta_{n+1}} & W_{n+1}(R)\\
    \Ainf(R) \ar[r]_-{\~\theta_n} \ar[u]^-{\~\lambda_{n+1}} & W_n(R) \ar[u]_-V
  }}
\]
Here $\~\lambda_{n+1}$ is an element of $\Ainf(R)$ satisfying $\~\theta_{n+1}(\lambda_{n+1}) = V(1)\in W_{n+1}(R)$.
\end{lemma}

By definition, $\ker\theta$ is a principal ideal; we let $[p]_A$ be a choice of generator. (This is \emph{not} a Teichm\"uller representative.) Then the elements
\begin{align*}
  [p^n]_A &\defeq [p]_A\phi([p]_A)\dotsm \phi^{n-1}([p]_A)\\
  \phi([p^n]_A) &= \phi([p]_A)\dotsm\phi^n([p]_A)
\end{align*}
generate $\ker(\~\theta_n\phi)$ and $\ker\~\theta_n$, respectively \cite[Lemma 3.12]{BMS1}. Note in particular that for $n\le m$,
\[ [p^m]_A = \phi^{n}([p^{m-n}]_A)[p^n]_A. \]
This gives the following, which will turn out to be related to the cell structure of $S^{\lambda_{n-1}}$ (\sec\ref{sub:T-repns}, Lemma \ref{lem:rog-pos}).

\begin{lemma}[{\cite[Remark 3.19]{BMS1}}]
\label{lem:a-seq-bms}
For $n<m$, the bottom row of the following commutative diagram is exact:
\[\xymatrix{
  & \Ainf(R) \ar[d]_-{\~\theta_n\phi} \ar[r] & \Ainf(R) \ar[d]^-{\~\theta_m\phi} \ar[r]^-{[p^n]_A} & \Ainf(R) \ar[d]^-{\~\theta_m\phi} \ar@{=}[r] & \Ainf(R) \ar[d]^-{\~\theta_n\phi}\\
  0 \ar[r] & W_n(R) \ar[r]_-{V^{m-n}} & W_m(R) \ar[r] & W_m(R) \ar[r]_-{F^{m-n}} & W_n(R) \ar[r] & 0
}\]
\end{lemma}


\subsubsection{Topological aspects}
The fundamental theorem of topological Hochschild homology is that $\THH_*(\F_p) = \F_p[\sigma]$, with $|\sigma| = 2$. This is due to B\"okstedt \cite{Bokstedt}, but for several decades his paper was available only via clandestine channels; fortunately, a public proof is now available, see \cite[\sec1.2]{HNHandbook} or \cite[\sec1]{KrauseNikolausBP}. The key local calculation of \cite{BMS2} is that B\"okstedt periodicity continues to hold for perfectoid rings.

\begin{theorem}[{\cite{LarsOCp}, \cite[Theorem 6.1]{BMS2}}]
Let $R$ be a perfectoid ring. Then $\THH_*(R;\Z_p) = R[\sigma]$, for some choice of $\sigma\in\THH_2(R;\Z_p)$. 
\end{theorem}

\begin{proposition}[{\cite[Propositions 6.2 and 6.3]{BMS2}}]
Let $R$ be a perfectoid ring, and let $A=\Ainf(R)$. We can choose generators $[p]_A\in\ker\theta$, $\sigma\in\TC^-_2(R;\Z_p)$, $t\in\TC^-_{-2}(R;\Z_p)$, and $\tau\in\TP_{-2}(R;\Z_p)$ to give identifications
\begin{align*}
  \TC^-_*(R;\Z_p) &= \frac{A[\sigma,t]}{\sigma t-[p]_A}\\[.2em]
  \TP_*(R;\Z_p) &= A[\tau^{\pm1}]
\end{align*}
such that $\xymatrix@1{\TC^-(R;\Z_p) \ar@<1ex>[r]^-\can \ar@<-1ex>[r]_-\varphi & \TP(R;\Z_p)}$ act as
\begin{align*}
  \can(\sigma) &= [p]_A\tau^{-1} & \varphi(\sigma) &= \tau^{-1}\\
  \can(t) &= \tau & \varphi(t) &= \phi([p]_A)\tau
\end{align*}
\end{proposition}
Since we will mainly be concerned with the canonical map, we will make a slight abuse of notation and write $t$ rather than $\tau$.

\subsection{Prisms}
\label{sub:prisms}
The classical definition of perfectoid rings given in the previous section is most useful for recognizing perfectoid rings in the wild. However, for our purposes it is more natural to view perfectoid rings as equivalent to \emph{perfect prisms}, as the emphasis will be on $\Ainf(R)$ rather than $R$. We hope to convince the reader that prisms are natural from the perspective of equivariant homotopy theory. Although we will only use perfect prisms in the remainder of the paper, the general setting is illuminating; we conjecture that our results hold for general prisms, but this is nontrivial to prove as $\TR$ and  $\TF$ are not yet understood in the general case. References for this section are \cite{Prismatic}, \cite{PrismNotes}, and \cite{AClBTrace}.

\begin{definition}
A \emph{$\delta$-structure} on a ring $A$ is a ring homomorphism $\phi\colon A\to A$ which is a \emph{derived lift of Frobenius}. When $A$ is $p$-torsionfree, this just means that $\phi(x)\equiv x^p\bmod p$; in this case, we can uniquely solve for the element ``witnessing'' this congruence, and thus define a function $\delta\colon A\to A$ such that
\[ \phi(x) = x^p + p\delta(x) \]
for all $x\in A$. ``Derived lift'' means that for general $A$, we specify $\delta$ rather than $\phi$, where $\delta$ is required to satisfy whatever identies are needed for $\phi$ to be a ring homomorphism.

A \emph{$\delta$-ring} is a ring together with a $\delta$-structure. A $\delta$-ring is \emph{perfect} if $\phi$ is an isomorphism.
\end{definition}

\begin{remark}
A $\delta$-structure on $A$ is equivalent to a ring section of $W_2(A)\morph^R A$, given in Witt coordinates by $x\mapsto(x,\delta(x))$ and in ghost coordinates by $x\mapsto(x,\phi(x))$. This can be seen from the pullback diagram
\[\xymatrix{
  W_2(A) \ar[r]^-R \ar[d]_-F \pullback & A \ar[d]^-{\varphi}\\
  A \ar[r]_-\can & A\mmod p \defeq A\tnsr^\L_{\Z} \F_p \hspace{-5em}
}\]
Again, this pullback diagram follows from injectivity of the ghost map in the torsionfree case, and in general by Kan extending. This perspective is apparently due to Rezk \cite{RezkEtale}.
\end{remark}

\begin{definition}[{\cite[Definition 3.2]{Prismatic}}]
A \emph{prism} consists of a $\delta$-ring $A$ together with an ideal $I$ such that:
\begin{itemize}
\item $I$ defines a Cartier divisor on $\Spec A$;

\item $A$ is derived $(p,I)$-complete;

\item (``prism condition'') $p\in I+\phi(I)A$.
\end{itemize}
We write $R=A/I$. A prism is:
\begin{itemize}
\item \emph{perfect} if $\phi$ is an isomorphism;

\item \emph{crystalline} if $I=(p)$;

\item \emph{orientable} if $I$ is principal, in which case a choice of generator is called an \emph{orientation}. We denote an orientation by $[p]_A$;

\item \emph{transversal} if it is orientable and $(p,[p]_A)$ is a regular sequence. (This implies that $R$ is $p$-torsionfree.)
\end{itemize}
\end{definition}

It is usually safe to assume all prisms are orientable. In this case, the first condition requires that $[p]_A$ is a non-zerodivisor, and the prism condition becomes
\[ \phi([p]_A) \equiv up \bmod [p]_A \]
for a unit $u\in A^\times$, or equivalently $\delta([p]_A)\in A^\times$.

\begin{example}
Here are examples of prisms.
\begin{itemize}
\item A crystalline prism is simply a $\delta$-ring which is $p$-torsionfree and $p$-complete.

\item The category of perfectoid rings is equivalent to the category of perfect prisms via $R\mapsto(\Ainf(R),\ker\theta)$ and $(A,I)\mapsto A/I$. Perfect $\F_p$-algebras correspond to perfect crystalline prisms, while torsionfree perfectoid rings correspond to perfect transversal prisms. Different untilts of a perfect $\F_p$-algebra $k$ correspond to different prism structures on $W(k)$.

\item Let $K/\Q_p$ be a finite extension with residue field $k$ and fixed uniformizer $\pi$. Let $\BK=W(k)\psr z$, with $\delta$-struture given by the usual Frobenius on $W(k)$ and $\phi(z)=z^p$. There is a surjection $\BK\to\O_K$ given by $z\mapsto\pi$, with kernel generated by an Eisenstein polynomial $E(z)$. Then $(\BK,(E(z)))$ is a prism, said to be of ``Breuil-Kisin type''. Different local fields $K$ with residue field $k$ correspond to different prism structures on $\BK$, so we can think of $K$ as an untilt (in the sense of ``characteristic zero incarnation'') of $k(\!(z)\!)$.

\item Let $A=\Z_p\psr{q-1}$ with $\phi(q)=q^p$, and let $I=([p]_q)$ with orientation $[p]_A\defeq [p]_q$. This ``$q$-crystalline'' prism has the interesting property that $\delta([p]_A)\equiv1\bmod[p]_A$.
\end{itemize}
\end{example}



Our main goal in this section is to explain the significance of the prism condition. First, we need an elaboration of it. We again set $[p^n]_A = [p]_A\phi([p]_A)\dotsm \phi^{n-1}([p]_A)$. The $i=1$ case of the following lemma appears as \cite[Lemma 3.5]{AClBTrace}.

\begin{proposition}
\label{prop:prism-extended}
For $i\le j$, there is a congruence
\[ \phi^j([p]_A) \equiv u_{i,j}p \bmod [p^i]_A \]
for some unit $u_{i,j}\in A^\times$.
\end{proposition}
\begin{proof}
The case $(i,j)=(1,1)$ is the prism condition, with $u_{1,1}=\delta([p]_A)$. To induct in the direction $(i,i)\implies(i+1,i+1)$, we write
\begin{align*}
  \phi^{i+1}([p]_A)
    &= \phi^i([p]_A)^p + p\phi^i u_{1,1}\\
    &= (u_{i,i}p+x[p^i]_A)^{p-1} \phi^i([p]_A) + p\phi^i u_{1,1}\\
    &\equiv (u_{i,i}^{p-1}p^{p-2}\phi^i([p]_A)+\phi^i u_{1,1})p \mod [p^{i+1}]_A.
\end{align*}
The final parenthesized expression is a unit because $\phi^i([p]_A)\in\rad(A)$. To induct in the $(i,j)\implies(i,j+1)$ direction, we write
\begin{align*}
  \phi^{j+1}([p]_A)
    &= \phi^j([p]_A)^p + p\phi^j u_{1,1}\\
    &\equiv (u_{i,j}^p p^{p-1}+\phi^j u_{1,1})p \mod [p^i]_A.
\end{align*}
The final parenthesized expression is a unit because $p\in\rad(A)$.
\end{proof}
\begin{corollary}
For $i\le \min\{j,r\}$ there is a congruence
\[ \phi^r([p^{j-i}]_A) \equiv up^{j-i} \bmod [p^i]_A \]
for some unit $u\in A^\times$.
\end{corollary}

\begin{remark}[Algebro-geometric interpretation of the prism condition]
  Geometrically, the prism condition says that the closed subschemes of $\Spec(A)$ cut out by any two $\phi$-iterates $\phi^i([p]_A)$ and $\phi^j([p]_A)$ intersect only in characteristic $p$:
  \[ V(\phi^i[p]_A) \cap V(\phi^j[p]_A) \subset V(p). \]
  We imagine shining a beam of characteristic $p$ into $\Spec A$, refracting it into beams of characteristic $\phi^\bullet([p]_A)$, which can then be studied one at a time; this is the reason for the name ``prism''. Our proof of Proposition \ref{prop:norm-lift} is an example of this idea.
  \begin{center}
    \begin{tikzpicture}
    \draw (0, 0) -- (0.5,{sin(60)}) -- (1, 0) -- cycle;
    \draw (0.5, 0) node[below] {$\Spec A$};

    \draw (-1,0) -- ({0.4},{0.4 * tan(60)});
    \draw (-1,0) node[below] {\tiny$V(p)$};

    \draw[red] ({0.5+0.1},{0.4 * tan(60)}) -- (2,2);
    \draw[red] (2,2) node[right] {\tiny$V([p]_A)$};
    \draw[orange] ({0.5+0.1},{0.4 * tan(60)}) -- (2,1.5);
    \draw[orange] (2,1.5) node[right] {\tiny$V(\phi([p]_A))$};
    \draw[Goldenrod] ({0.5+0.1},{0.4 * tan(60)}) -- (2,1);
    \draw[Goldenrod] (2,1) node[right] {\tiny$V(\phi^2([p]_A))$};
    \draw[Green] ({0.5+0.1},{0.4 * tan(60)}) -- (2,0.5);
    \draw[Green] (2,0.5) node[right] {\dots};
    \draw[blue] ({0.5+0.1},{0.4 * tan(60)}) -- (2,0);
    \draw[blue] (2,0) node[right] {\dots};
    \draw[violet] ({0.5+0.1},{0.4 * tan(60)}) -- (2,-0.5);
    \draw[violet] (2,-0.5) node[right] {\dots};
    \end{tikzpicture}
  \end{center}
\end{remark}

There is also an equivariant interpretation of the prism condition: it is just what we need to build a Mackey functor (specifically, to satisfy the axiom $\res\circ\tr=p$).

\begin{corollary}[Equivariant interpretation of the prism condition]
The prescription
\begin{align*}
  \m W(\T/C_{p^n}) &= A/[p^{n+1}]_A
\end{align*}
defines a Mackey functor $\m W$, with restriction maps given by the natural projections, by defining the transfers as indicated:
\[\xymatrix{
  \vdots \mdown{}\\
  A/[p^3]_A \mdown{1} \mup{}\\
  A/[p^2]_A \mdown{1} \mup{\phi^2([p]_A)u_{2,2}^{-1}}\\
  A/[p]_A \mup{\phi([p]_A)u_{1,1}^{-1}}
}\]
\end{corollary}
In subsequent Lewis diagrams we will omit the $u_{i,i}^{-1}$ to save space (and since we will be working up to units anyway), but they do matter.

\begin{remark}
Proposition \ref{prop:prism-extended} is not optimal: already in the base case we in fact have $\phi([p]_A)\equiv u_{1,1}p\bmod [p]_A^p$. The congruences we have given suffice to construct the Mackey structure, but it seems that sharper statements are needed to study the interaction of the norm maps with the slice filtration.
\end{remark}

Let us explain the connection to topology. Assume there exists a connective $\E_\infty$-ring $\S_A$ such that $\S_A\tnsr_\S\Z=A$; this can be constructed by hand in all the examples we have given, starting from Lurie's spherical Witt vectors \cite[Example 5.2.7]{Ell2}. In the case of a perfectoid ring $R$, the canonical map
\[ \THH(R;\Z_p)\to\THH(R/\S_A) \]
is an equivalence (\cite[Proposition 3.5]{KrauseNikolausBP}, \cite[Lemma 14]{Zhouhang}); for general prisms, $\THH(R/\S_A)$ is what one should study. The Mackey functor $\m W$ arises as $\mpi_0\THH(R/\S_A)^{E\T_+}$, which agrees with $\mpi_0\THH(R/\S_A)$ when $A$ is perfect. The reader should consult \cite{Zhouhang} for more in this direction, such as a Hopkins-Mahowald result for perfectoid rings and complete regular local rings.

The multiplication on $A$ descends to $\m W$, making it a \emph{Green functor}. In the situation of the preceding paragraph, the framework of equivariant homotopy theory implies that $\m W$ has even more structure: it is a \emph{Tambara functor}. We refer to \cite[Definition 2.11]{MazurTambara} for a full definition, but essentially this means that $\m W$ comes with \emph{norm maps} which are multiplicative analogues of transfers. In the remainder of this section we will try to identify the norm algebraically. This material is not used elsewhere in the paper, but we expect it to be important to further developments.

\newcommand{\Norm}[1]{N^{p^{#1}}_{p^{#1-1}}}
\begin{definition}
Let $A$ be a perfect prism. A function $A\morph^{\frN} A$ is \emph{a lift of the norm $\Norm n$} if
\[\xymatrix{
  A \ar[r]^-{\frN} \ar[d] & A \ar[d]\\
  A/[p^n]_A \ar[r]_-N & A/[p^{n+1}]_A
}\]
commutes, where
\[ A/[p^n]_A=W_n(A/[p]_A) \morph^N W_{n+1}(A/[p]_A)=A/[p^{n+1}]_A \]
is Angeltveit's norm map for Witt vectors \cite{AngeltveitNorm}. We also say that $\frN(x)$ lifts $N_{p^{n-1}}^{p^n}(x)$ if this diagram commutes for a particular $x\in A$.
\end{definition}

\begin{proposition}
\label{prop:norm-lift}
Let $A$ be a perfect prism. $\frN(x)$ is a lift of $\Norm n(x)$ if and only if
\begin{align*}
  \frN(x) &\equiv \phi(x) \bmod \phi^n([p]_A)\\
  \frN(x) &\equiv x^p \bmod [p^n]_A
\end{align*}
\end{proposition}
\begin{proof}
By functoriality, we may assume that $A$ is transversal, so that $A/[p]_A$ is $p$-torsionfree. This has the advantage that the ghost map $\gh\colon W_n(A/[p]_A)\to(A/[p]_A)^n$ is injective. The identification
\[  A/[p^n]_A \iso W_n(A/[p]_A) \]
is given in ghost coordinates by
\[ x \mapsto (\phi^{-(n-1)}(x)\bmod{[p]_A},\dotsc, x\bmod [p]_A). \]
By \cite[Theorem 1.4]{AngeltveitNorm}, the norm is given in ghost coordinates by
\begin{align*}
  W_n(A/[p]_A) &\morph^N W_{n+1}(A/[p]_A)\\
  (w_0,w_1,\dotsc,w_{n-1}) &\mapsto (w_0, w_0^p, w_1^p, \dotsc, w_{n-1}^p).
\end{align*}
Thus we get the congruences
\begin{align*}
  \frN(x) &\equiv \phi(x) \bmod \phi^n([p]_A)\\
  \frN(x) &\equiv x^p \bmod \phi^i([p]_A),\ 0\le i<n.
\end{align*}
Using transversality again, the second line is equivalent to $\frN(x) \equiv x^p \bmod [p^n]_A$ by \cite[Lemma 3.6]{AClBTrace}.
\end{proof}
Since $\phi$ and $x\mapsto x^p$ are both multiplicative, we immediately get $\frN(xy)\equiv\frN(x)\frN(y)\bmod[p^{n+1}]_A$. We do not expect there to be lifts of the norm which are multiplicative as maps $A\to A$, however.

\begin{proposition}
\label{prop:norms-q}
Let $A$ be a prism over $(\Z_p[q^{1/p^\infty}]^\wedge_{(p,q-1)},[p]_{q^{1/p}})$, and let $x,y\in A$ with $\delta(x)=\delta(y)=0$. Then the $q$-Pochhammer symbol
\[ (x,-y;q^{p^{n-1}})_p \defeq \prod_{i=0}^{p-1} (x-q^{ip^{n-1}}y) \]
is a lift of $\Norm n(x-y)$.
\end{proposition}
\begin{proof}
This follows from $q^{p^{n-1}}\equiv\zeta_p\bmod\phi^n([p]_{q^{1/p}})$, $q^{p^{n-1}}\equiv1\bmod [p^n]_{q^{1/p}}$, and $\phi(x-y)=x^p-y^p$.
\end{proof}

The following is an adaptation of Borger's formula for the norm \cite[Definition 1.1]{AngeltveitNorm}.

\begin{proposition}
The function
\[ \frN_n(x) = \phi(x) - \frac{\phi^n([p]_A)}{u_{n,n}}\delta(x) \]
is a lift of the norm $\Norm n$.
\end{proposition}
\begin{proof}
Certainly $\frN_n(x)\equiv\phi(x)\bmod\phi^n([p]_A)$, and $\frN_n(x)\equiv x^p\bmod[p^n]_A$ by Proposition \ref{prop:prism-extended}.
\end{proof}

\begin{warning}
The $q$-Pochhammer symbol $(x,-y;q^{p^{n-1}})_p$ does not agree with our handicrafted norm $\frN_n(x-y)$ as functions $A\to A$ unless $n=1,p=2$. It seems highly non-trivial to write down lifts of the norm that make sense for arbitrary prisms and which specialize to the $q$-Pochhammer symbol. The best we have found in this direction is the identity
\[ (x,-y;q)_3 = \psi^3(z) - [3]_q\delta_3(z) - [3]_q([2]_q-2)(z\delta_2(z) - \delta_3(z)), \]
where $z=x-y$, obtained through trial and error with the help of Sage.
\end{warning}

\newpage
\part{Results}
\label{part:results}
\section*{Notation}

Throughout $R$ is a fixed perfectoid ring, $A=\Ainf(R)$ equipped with its lift of Frobenius $\phi$, $[p]_A$ is a generator of $\ker(A\morph^\theta R)$, and
\begin{align*}
  [p^n]_A &= [p_A] \phi([p_A])\dotsm\phi^{n-1}([p]_A)\\
  [n]_A &\defeq [p^{v_p(n)}]_A\\
  [n]_A! &= [1]_A [2]_A \dotsm [n]_A.
\end{align*}
(This definition of $[n]_A$ does not agree with $[n]_{q^{1/p}}$, but it does up to units, which is enough for our purposes.)

We will simply write $\THH$, $\TC^-$, etc.\ for $\THH(R;\Z_p)$, $\TC^-(R;\Z_p)$, etc. We caution the reader that we write $\THH$ for \emph{both} the cyclonic spectrum and its underlying Borel $\T$-spectrum; the intended meaning should always be clear from context. We also write $\TC^-_*\defeq\pi_*\TC^-=\pi_*\TC^-(R;\Z_p)$, etc.

The circle group is denoted by $\T$. The ring of real representations of $\T$ is denoted $\RO(\T)$, while the ring of complex representations is denoted by $R(\T)$. We write $*$ for $\Z$-grading and $\rog$ for $RO(\T)$-grading. Occasionally we may write $\rog$ where only an actual representation would make sense; we hope this is clear from context. We write $a_V\in\pi_{-V}^G\S$ for the Euler class associated to an actual representation $V$.

$\lambda^n$ is the one-dimensional complex $\T$-representation where $z\in\T$ acts by $z^n$, $\lambda_i\defeq\lambda^{p^i}$, and $\lambda_\infty\defeq\lambda^0$. We also define
\begin{align*}
  [n]_\lambda &= 1 + \lambda + \dotsc + \lambda^{n-1}\\
  \{n\}_\lambda &= \lambda + \dotsc + \lambda^n.
\end{align*}

Given $\alpha\in\RO(\T)$, we write $\alpha^{(r)}$ for the fixed space $\alpha^{C_{p^r}}$ pulled back along the root isomorphism $\T\iso\T/C_{p^r}$. We then set $d_r(\alpha)=\dim_\C(\alpha^{(r)})$. Explicitly, $\lambda_\infty'=\lambda_\infty$, $\lambda_i'=\lambda_{i-1}$ for $i>1$, and $\lambda_0'=0$. Thus for a representation
\begin{align*}
  \alpha &= k_0 \lambda_0 + \dotsb + k_n\lambda_n + k_\infty\lambda_\infty,\\
  \intertext{we get}
  d_r(\alpha) &= k_r + k_{r+1} + \dotsb + k_n + k_\infty.
\end{align*}
When a single representation $\alpha$ is in play, we may abbreviate $d_r(\alpha)$ to $d_r$.

The (equivariant) sphere spectrum is denoted by either $\S$ or $S^0$. The smash product of spectra, including $G$-spectra, is denoted by $\tnsr$. We do not distinguish between an abelian group (or Mackey functor) and its associated Eilenberg-Mac Lane (equivariant) spectrum, nor between a $G$-space and its suspension spectrum.

We write $\regconn n X$ for the $n^\th$ slice cover of a $G$-spectrum $X$, and $\regslice n X$ for its $n$-slice.

\section{\tops{$\RO(\T)$}{RO(T)} calculations}
\label{sec:rog}
In this section we study the portion of the $\RO(\T)$-graded homotopy $\mpi_\rog\gT$ that will be needed for the slice computations in \sec\ref{sec:slice}. The Mackey functors we will be using are introduced in \sec\ref{sub:rog-mackey}. In \sec\ref{sub:rog-tools}, we review the tools we will use for computing these groups, and derive the $q$-gold relation (Lemma \ref{lem:q-au}). Our actual calculations are carried out in \sec\ref{sub:computations}.

The author learned to work with the gold elements from Zeng's beautiful paper \cite{Zeng}, and we urge the reader interested in learning to do these types of calculations to read Zeng's paper (as well as the earlier version \cite{Zengv1}, which shows several alternative ways to do the same calculation).

\subsection{Mackey functors}
\label{sub:rog-mackey}

\begin{definition}
The \emph{Witt Mackey functor} $\m W$ (c.f.\ \sec\ref{sub:prisms}) is the $\T$-Mackey functor given on objects by
\[ \m W(\T/C_{p^k}) = A/[p^{k+1}]_A. \]
Restrictions are given by the natural quotient maps, and for $i\le j$ the transfer is defined as multiplication by $\dfrac{[p^{j+1}]_A}{[p^{i+1}]_A}u_{i,j}=\phi^{i+1}([p^{j-i}]_A)u_{i,j}$, where $u_{i,j}\in A^\times$ is an appropriate unit (c.f.\ Proposition \ref{prop:prism-extended}). We will generally suppress this unit from the notation. We also write $\m W^{(n)}\defeq\mathop{\downarrow^\T_{C_{p^n}}}\m W$ for the restriction of $\m W$ to a $C_{p^n}$-Mackey functor. Abstractly we have $\m W(\T/C_{p^k}) \isom W_{k+1}(R)$, but the $A$-algebra structure will be important for us.
\end{definition}

\begin{definition}
Let $\m M$ be a $G$-Mackey functor, and let $H\le G$ be a subgroup. Then $\tr_H M$ is defined to be the sub-Mackey functor of $\m M$ generated under transfers by $\downarrow^G_H \m M$, while $\Phi^H \m M$ is defined as the quotient
\[ 0 \to \tr_H \m M \to \m M \to \Phi^H \m M \to 0. \]
Importantly, we have $\Phi^K \tr_H \m M = 0$ for $H\le K$.
\end{definition}

\begin{remark}
If $G$ is finite, then $\tr_H \m M$ is just the image of the canonical map $\uparrow^G_H \downarrow^G_H \m M \to \m M$.
\end{remark}

\begin{example}
\label{ex:mackey}
For $G=C_{p^2}$, these sequences are
\[\xymatrix{
   0 \ar[r] \mdown{} & A/[p]_A \mdown p \ar[r]^-{\phi([p^2]_A)} & A/[p^3]_A \mdown 1 \ar[r] & A/\phi([p^2]_A) \mdown 1 \ar[r] & 0 \mdown{}\\
   0 \ar[r] \mup{} \mdown{} & A/[p]_A \mdown p \mup 1 \ar[r]^-{\phi([p]_A)} & A/[p^2]_A \mdown 1 \mup {\phi^2([p]_A)} \ar[r] & A/\phi([p]_A) \mdown{} \mup{\phi^2([p]_A)} \ar[r] & 0 \mup{} \mdown{}\\
  0 \mup{} \ar[r] & A/[p]_A \mup 1 \ar[r]_-1 & A/[p]_A \mup{\phi([p]_A)} \ar[r] & 0 \mup{} \ar[r] & 0 \mup{}\\
  0 \ar[r] & \tr_e \m W^{(2)} \ar[r] & \m W^{(2)} \ar[r] & \Phi^e \m W^{(2)} \ar[r] & 0
}\]
and
\[\xymatrix{
   0 \ar[r] \mdown{} & A/[p^2]_A \mdown p \ar[r]^-{\phi^2([p]_A)} & A/[p^3]_A \mdown 1 \ar[r] & A/\phi^2([p]_A) \mdown{} \ar[r] & 0 \mdown{}\\
   0 \ar[r] \mup{} \mdown{} & A/[p^2]_A \mdown 1 \mup 1 \ar[r]^-1 & A/[p^2]_A \mdown 1 \mup {\phi^2([p]_A)} \ar[r] & 0 \mdown{} \mup{} \ar[r] & 0 \mup{} \mdown{}\\
  0 \mup{} \ar[r] & A/[p]_A \mup {\phi([p]_A)} \ar[r]_-1 & A/[p]_A \mup{\phi([p]_A)} \ar[r] & 0 \mup{} \ar[r] & 0 \mup{}\\
  0 \ar[r] & \tr_{C_p} \m W^{(2)} \ar[r] & \m W^{(2)} \ar[r] & \Phi^{C_p} \m W^{(2)} \ar[r] & 0
}\]
and
\[\xymatrix{
   0 \ar[r] \mdown{} & A/[p]_A \mdown p \ar[r]^-{\phi([p]_A)} & A/[p^2]_A \mdown p \ar[r] & A/\phi([p]_A) \mdown{p} \ar[r] & 0 \mdown{}\\
   0 \ar[r] \mup{} \mdown{} & A/[p]_A \mdown p \mup 1 \ar[r]^-{\phi([p]_A)} & A/[p^2]_A \mdown 1 \mup 1 \ar[r] & A/\phi([p]_A) \mdown{} \mup{1} \ar[r] & 0 \mup{} \mdown{}\\
  0 \mup{} \ar[r] & A/[p]_A \mup 1 \ar[r]_-1 & A/[p]_A \mup{\phi([p]_A)} \ar[r] & 0 \mup{} \ar[r] & 0 \mup{}\\
  0 \ar[r] & \tr_e \m W^{(2)} \ar[r] & \tr_{C_p} W^{(2)} \ar[r] & \Phi^e \tr_{C_p} \m W^{(2)} \ar[r] & 0.
}\]
\end{example}

\begin{abuse}
In our present case $G=C_{p^\infty}$, we will write $\tr_{C_n}$ when we really mean $\tr_{C_{p^{v_p(n)}}}$. We expect that our results hold as stated without this abuse when $\gT$ is regarded as an integral cyclonic spectrum, but we have not carefully verified this.
\end{abuse}


\subsection{Tools}
\label{sub:rog-tools}
The most powerful tool for computing $\RO(\T)$-graded $\THH$ is the isotropy separation square. This takes the form
\begin{equation}
\label{eq:iso-TF}
\vcenter{\xymatrix{
  \pi_\rog \Sigma\THH_{h\T} \ar[r] \ar@{=}[d] & \TF_\rog \ar[r]^-R \ar[d] & \TF_{\rog'} \ar[d]^-\varphi\\
  \pi_\rog \Sigma\THH_{h\T} \ar[r] & \TCmin_\rog \ar[r] & \TP_\rog
}}\end{equation}
at infinite level, and
\begin{equation}
\label{eq:iso-TR}
\vcenter{\xymatrix{
  \pi_\rog \THH_{hC_{p^n}} \ar[r] \ar@{=}[d] & \TR^{n+1}_\rog \ar[r]^-R \ar[d] & \TR^n_{\rog'} \ar[d]^-\varphi\\
  \pi_\rog \THH_{hC_{p^n}} \ar[r] & \pi_\rog\THH^{hC_{p^n}} \ar[r] & \pi_\rog\THH^{tC_{p^n}}
}}
\end{equation}
for finite subgroups.

We will understand $\TF_\rog$ by relating it to $\TCmin_\rog$ and $\TP_\rog$. In doing so, we will rely heavily on the following $\RO(\T)$-graded version of Tsalidis' theorem \cite[Theorem 2.4]{Tsalidis}. Our reliance on this result is the main difficulty in generalizing our results to imperfect prisms.

\begin{theorem}[{\cite[Theorem 5.1]{ROS1TR}}]
Write $\rog=\alpha+*$ for $*\in\Z$ and $\alpha\in\RO(\T)$. The vertical maps in the isotropy separation square \eqref{eq:iso-TR} are isomorphisms for
\[ *\ge2\max\{-d_1(\alpha),\dotsc,-d_n(\alpha)\}. \]
Consequently, the vertical maps in \eqref{eq:iso-TF} are isomorphisms for
\begin{align*}
* &\ge2\max\{-d_i(\alpha)\mid i\ge1\}\\
&\ge0.
\end{align*}
\end{theorem}

To begin, we had better know the $\Z$-graded homotopy of $\TF$ and $\TR^n$, which was not spelled out in \cite{BMS2}. Abstractly the following tells us that $\TR^n_* = W_n(R)[\sigma]$; this was already known (at least on $\pi_0$) by \cite[Theorem 3.3]{HMFinite}), but the $A$-module structure is crucial for what follows.
\begin{proposition}
The $\Z$-graded homotopy of $\TF$ and $\TR^{n+1}$ are
\begin{align*}
  \TF_* &= A[\sigma]\\
  \TR^n_* &= (A/[p^n]_A)[\sigma]
\end{align*}
\end{proposition}
\begin{proof}
We already know that $\TP_*=A[t^\pm]$, and we have by \cite[Remark 6.6]{BMS2} that the Tate cohomology at finite level is
\[ \pi_* \THH^{tC_{p^n}} = (A/\phi([p^n]_A))[t^{\pm}]. \]
The result now follows from Tsalidis' theorem (in $\Z$ degrees). Note the map $\varphi$ is $\phi$-linear with respect to $A$-module structures, which is why we get $(A/[p^n]_A)[\sigma]$ instead of $(A/\phi([p^n]_A))[t^{-1}]$.
\end{proof}

Our next task is to understand the bottom row of the isotropy separation square. As in the $\Z$-graded case, this can be computed by a spectral sequence
\[ H^*(B\T, \TR^1_{\alpha+*}) \conv \TCmin_{\alpha+*}. \]
On the underlying, we have $\TR^1_{\alpha+*} = \TR^1_{|\alpha|+*}$, and we let $u_\alpha\in\TR^1_{|\alpha|-\alpha}$ be the image of $1\in\TR^1_0\iso\TR^1_{|\alpha|-\alpha}$. The action of $\T$ on this group is necessarily trivial, and the spectral sequence collapses for degree reasons, so $u_\alpha$ lifts to $\TCmin_{|\alpha|-\alpha}$; by construction, it is invertible. A similar discussion applies to $\TP$, so the whole bottom row of the isotropy separation sequence is $u_{\lambda_i}$-periodic. We also recall (\sec\ref{sub:rog-grading}) that the right column is $a_{\lambda_0}$-periodic. As Zeng \cite{Zeng} puts it, the clash between the $u_{\lambda_i}$-periodicity of $\Sigma\THH_{h\T}$ and the $a_{\lambda_0}$-periodicity of the rightmost $\TF$ produces a lot of classes in the middle $\TF$. Note that the classes $u_{\lambda_i}$ lift to $\TF_{2-\lambda_i}$ by Tsalidis' theorem, but are no longer invertible there.

\begin{remark}
We explain the relation between our account in terms of the gold elements and that given by Angeltveit-Gerhardt. Their analysis of the Tate spectral sequence begins by observing that the $E^2$ page of the $\alpha$-graded HOSS, HFPSS, or TSS is simply a shift by $2d_0(\alpha)$ of the usual one: this is $u_\rog$-periodicity. Then they point out that while the Tate \emph{spectral sequence} obviously depends on $d_0(\alpha)$, the Tate \emph{spectrum} $\THH^t$ only depends on $\alpha'$ (which is not true of $\THH_h$ or $\THH^h$). To exploit this, they reindex the $\alpha$-graded Tate spectral sequence to make it isomorphic to the usual one, but with a different meaning of ``first/second quadrant'': this trick is essentially $a_{\lambda_0}$-periodicity. The correspondence between our names and Angeltveit-Gerhardt's is $u_\alpha^{-1} \longleftrightarrow t^{d_0(\alpha)}[-\alpha]$.
\end{remark}

In principle---although it is not the strategy we will use here---, this allows one to completely compute $\TF_\rog$, using $a_{\lambda_0}$-periodicity for the induction and $u_{\lambda_i}$-periodicity for the base case. The remaining ingredient needed is to understand the relation between the various classes $a_{\lambda_i}$, $u_{\lambda_i}$, $\sigma$, and $t$. For this, we will bring in the other tool we have for computing $\RO(\T)$-graded homotopy (and the one we will use in this paper): cell structures, c.f.\ \sec\ref{sub:T-repns}. As $C_n$-spectra, the representation spheres $S^{\lambda_j}$ have cell structures of length 2, which we call the ``long cell structure''. However, as pointed out to us by Hill, these have ``short cell structures'' \emph{as $\T$-spectra}: there is an exact triangle
\[ \T/C_{p^j+} \to S^0 \to S^{\lambda_j}. \]
The following is a $q$-deformation of the ``gold relation'' from \cite[Lemma 3.6(vii)]{HHR_KR}.

\begin{lemma}[$q$-gold relation]
\label{lem:q-au}
The following relations hold in $\TF_\rog$. For $j\ge0$,
\[ \sigma a_{\lambda_j} = [p^{j+1}]_A u_{\lambda_j}. \]
For $0\le i < j$,
\begin{align*}
  a_{\lambda_j} u_{\lambda_i}
    &= \tr_{C_{p^i}}^{C_{p^j}}(1) a_{\lambda_i} u_{\lambda_j}\\[.5em]
    &= \frac{[p^{j+1}]_A}{[p^{i+1}]_A} a_{\lambda_i} u_{\lambda_j}\\[.5em]
    &= \phi^{i+1}([p^{j-i}]_A) a_{\lambda_i} u_{\lambda_j}.
\end{align*}
\end{lemma}
\begin{proof}
By Tsalidis' theorem, we have that $\TF_{\lambda_j} = \TCmin_{\lambda_j}$, and $\TCmin_{\lambda_j} = A\<\sigma u_{\lambda_j}^{-1}\>$ since $\TCmin$ is $u_{\lambda_j}$-periodic.

On the other hand, mapping the short cell structure into $\THH$ gives a short exact sequence
\[\xymatrix{
  0 \ar[r] & \TF_{\lambda_j} \ar[r]^-{a_{\lambda_j}} \ar@{=}[d] & \TF_0 \ar[r] \ar@{=}[d] & \TR^{j+1}_0 \ar[r] \ar@{=}[d] & 0\\
  0 \ar[r] & A\<\sigma u_{\lambda_j}^{-1}\> \ar[r] & A \ar[r] & A/[p^{j+1}]_A \ar[r] & 0
}\]
This gives the relation $\sigma u_{\lambda_j}^{-1} a_{\lambda_j} = [p^{j+1}]_A$, which proves the first claim. Now write
\[ \sigma = [p^{i+1}]_A\frac{u_{\lambda_i}}{a_{\lambda_i}} = [p^{j+1}]_A\frac{u_{\lambda_j}}{a_{\lambda_j}}. \]
Since $[p^{j+1}]_A = \phi^{i+1}([p^{j-i}]_A) [p^{i+1}]_A$, multiplying out gives the second claim.
\end{proof}

In particular
\begin{align*}
  a_{\lambda_j} &= \phi([p^j]_A) a_{\lambda_0} u_{\lambda_j}u_{\lambda_0}^{-1},\\
  t &= a_{\lambda_0}u_{\lambda_0}^{-1},
\end{align*}
so from now on we will stop writing $t$. We also see that when working in $\TC^-$ or $\TP$ (where the $u_{\lambda_j}$ are invertible), it suffices to consider only the classes $\sigma$, $a_{\lambda_0}$, and $u_{\lambda_j}$.

\begin{remark}
The two forms of the $q$-gold relation,
  \begin{align*}
  \sigma a_{\lambda_j} &= [p^{j+1}]_A u_{\lambda_j}\\
  a_{\lambda_j} u_{\lambda_i} &= \phi^{i+1}([p^{j-i}]_A) a_{\lambda_i} u_{\lambda_j},
  \end{align*}
  can be combined by formally setting $u_{\lambda_{-1}}\defeq\sigma$, $a_{\lambda_{-1}}\defeq 1$. This is one possible explanation of the `philosophical' role of the B\"okstedt generator.
\end{remark}

\subsection{Computations}
\label{sub:computations}

\begin{lemma}
\label{lem:rog-pos}
For $j<n$,
\[
  \TR^{n+1}_{\lambda_j + *} =
  \begin{cases}
    \tr_{C_{p^j}}^{C_{p^n}}(1) u_{\lambda_j}^{-1} & * = -2\\
    \sigma^i u_{\lambda_j}^{-1} & *=2i\ge0\\
    0 & \text{else}
  \end{cases}
\]
These generate the Mackey functors $\tr_{C_{p^j}}\m W^{(n)}$ for $*=-2$ and $\m W^{(n)}$ for $*=2i\ge0$.
\end{lemma}
\begin{proof}
Since $\TR^{n+1}_*$ is even, the long cell structures give long exact sequences
\[ 0 \to \TR^{j+1}_{2+*} \xra{V^{n-j}} \TR^{n+1}_{\lambda_j+*} \morph^{a_{\lambda_j}} \TR^{n+1}_* \xra{F^{n-j}} \TR^{j+1}_* \to 0, \]
described algebraically by Lemma \ref{lem:a-seq-bms}. When $*=-2$, this says that the transfer $\TR^{j+1}_0\to\TR^{n+1}_{\lambda_j-2}$ is an isomorphism. The case $*=2i\ge0$ follows from Tsalidis' theorem.
\end{proof}

While the statement of the next lemma is imposing, its content is very simple. To compute $\TF_{2i-\alpha}$, we start with $\TF_{2i}$ generated by $\sigma^i$, and exchange $\sigma$s for $u_{\lambda_j}$s, starting with the largest value of $j$. If we run out of $\sigma$s, we continue by adding $a_{\lambda_j}$s to land in the correct $\RO(\T)$-graded degree. Alternatively, we can start with $\TF_{-\alpha}$ generated by $a_\alpha$, and exchange $a_{\lambda_j}$s for $u_{\lambda_j}$s, or $\sigma$s once we run out of $a_{\lambda_j}$s. For example:
\begin{align*}
  \TF_4 &= A\<\sigma^2\> & \TF_{4-2\lambda_0-\lambda_1} &= A\<a_{\lambda_0}u_{\lambda_0}u_{\lambda_1}\>\\
  \TF_{4-\lambda_1} &= A\<\sigma u_{\lambda_1}\> & \TF_{2-2\lambda_0-\lambda_1} &= A\<a_{\lambda_0}^2 u_{\lambda_1}\>\\
  \TF_{4-\lambda_0-\lambda_1} &= A\<u_{\lambda_0} u_{\lambda_1}\> & \TF_{-2\lambda_0-\lambda_1} &= A\<a_{\lambda_0}^2 a_{\lambda_1}\>
\end{align*}

\begin{lemma}
\label{lem:rog-neg}
The portion of $\TF_\rog$ of the form $\rog=*-\alpha$, with $*\in\Z$ and $\alpha$ an actual representation, is
\[ A[\sigma,a_{\lambda_i},u_{\lambda_i}]. \]
Explicitly, let $\alpha = k_0 \lambda_0 + \dotsb + k_{n-1} \lambda_{n-1}$ be an actual, fixed-point free representation of $\T$. Write
\[ \alpha[s,t) \defeq k_s \lambda_s + \dotsb + k_{t-1} \lambda_{t-1}. \]
Then for $i\ge0$,
\begin{align*}
  \TF_{2i-\alpha} &=
  \begin{cases}
    A\<a_{\alpha[0,r-1)} a_{\lambda_{r-1}}^{d_{r-1}(\alpha)-i} u_{\lambda_{r-1}}^{i-d_r(\alpha)} u_{\alpha[r,n)}\> & d_r(\alpha) \le i < d_{r-1}(\alpha)\\
    A\<\sigma^{i-d_0(\alpha)} u_\alpha\> & d_0(\alpha) \le i
  \end{cases}\\[1em]
 \TR^{n+1}_{2i-\alpha} &=
  \begin{cases}
    A/\phi([p^{n-r}]_A)\<a_{\alpha[0,r-1)} a_{\lambda_{r-1}}^{d_{r-1}(\alpha)-i} u_{\lambda_{r-1}}^{i-d_r(\alpha)} u_{\alpha[r,n)}\> & d_r(\alpha) \le i < d_{r-1}(\alpha)\\
    A/[p^{n+1}]_A\<\sigma^{i-d_0(\alpha)} u_\alpha\> & d_0(\alpha) \le i
  \end{cases}
\end{align*}
When $d_0(\alpha)\le i$, the corresponding Mackey functor is $\m W$. When $d_r(\alpha)\le i<d_{r-1}(\alpha)$, the corresponding Mackey functor is $\Phi^{C_{p^{r-1}}} \m W$.
\end{lemma}
\begin{proof}

We proceed by induction on $(k_0,\dotsc,k_{n-1})\in\N^n$ in dictionary order. Let $\beta$ be an actual $\T$-representation whose restriction to $C_{p^r}$ is trivial, and let $\alpha = \beta + \lambda_r$. The short cell structure gives
\[\xymatrix{
  0 \ar[r] &
  \TF_{2i-\beta} \ar[r]^-{a_{\lambda_r}} \ar@{=}[d] \ar@{=}[d] &
  \TF_{2i-\alpha} \ar[r] \ar@{=}[d] &
  \TR^{r+1}_{2i-2d_r(\alpha)} \ar[r] \ar@{=}[d] &
  0\\
  0 \ar[r] &
  A\<g_\beta\> \ar[r] &
  A\<g_\alpha\> \ar[r] &
  A/[p^{r+1}]_A\<\sigma^{i-d_r(\alpha)}\> \ar[r] &
  0
}\]
for some generators $g_\beta$, $g_\alpha$. From this we see that
\[
  g_\alpha = 
  \begin{cases}
    g_\beta a_{\lambda_r}  & d_r(\alpha)>i\\
    g_\beta u_{\lambda_r}\sigma^{-1} & d_r(\alpha) \le i
  \end{cases}
\]
which implies the description of the generators. The $A$-module structure and Mackey structure follow from the fact that $a_{\lambda_r}$ kills transfers from $C_{p^r}$, along with $\TR^{n+1}=\TF/a_{\lambda_n}$ (the short cell structure).
\end{proof}

In the case of a perfect $\F_p$-algebra, these groups were already known by \cite[Proposition 9.1]{HMFinite} and \cite[Theorem 8.3]{ROS1TR}. Our result is finer, as it identifies the multiplicative and Mackey structure.

\begin{remark}
We point out some important structural features of $\TF_\rog$ that are implicit in the preceding computations; proofs will be given in \cite{SulROGTHH}. For $\alpha\in\RU(\T)$, one can show that the (nonzero) groups $\TF_\alpha$ are all isomorphic to $A$, and vanish for $d_\infty(\alpha)<0$; one can even determine the generators relatively quickly. On the other hand, the groups $\TF_{\alpha-1}$ are all torsion, need not be cyclic, and are more difficult to determine.

This is why we used the long cell structure in the proof of Lemma \ref{lem:rog-pos} rather than the short cell structure: $\TF_{\lambda_j-2}$ is \emph{zero}, and instead $\TR^{n+1}_{\lambda_j-2}=\TF_{\lambda_j+\lambda_n-3}$. It is easier to calculate the $\TR^{n+1}_{}$ group directly than to calculate the latter $\TF$ group. This problem does not arise in Lemma \ref{lem:rog-neg} since there are no odd-dimensional $\TF$ classes in that range.

There is also a ``gap'' phenomenon present in Lemma \ref{lem:rog-pos}. For $k\ge1$ the groups $\TF_{k\lambda_j-(2k-1)},\dotsc,\TF_{k\lambda_j-3}$ are nonzero, and yet $\TF_{k\lambda_j-1}=0$. This is essential to the proof of Theorem \ref{thm:slice-tower}, in order to be able to isolate a single Mackey functor by ``clipping off the Tsalidis part''. It also allows us, for an actual representation $V$, to read off the homotopy Mackey functor $\mpi_{V+*}\THH$ as quotients of $\TF_{V+*}$ for $*\ge0$ (Tsalidis' theorem guarantees we can do this for $*\ge1$). This gap is easy to understand using the HOTFSS (\cite[\sec3]{ROS1TR}): for $\alpha=k_0\lambda_0+\dotsc+k_{n-1}\lambda_{n-1}$, this is the spectral sequence computing $\TF_{\alpha+*}$ from
\begin{gather*}
  \pi_{\alpha+*}(\Sigma\THH_{h\T})\\
  a_{\lambda_0}^\pm\pi_{\alpha'+*}(\Sigma\THH_{h\T})\\
  \vdots\\
  a_{\lambda_0}^\pm\dotsm a_{\lambda_{n-2}}^{\pm} \pi_{\alpha^{(n-1)}+*}(\Sigma\THH_{h\T})\\
  a_{\lambda_0}^\pm\dotsm a_{\lambda_{n-1}}^{\pm} \TF_*
\end{gather*}
The integer degrees of the orbit terms depend on the dimension sequence of $\alpha$, but the bottom term will always start in degree 0. For virtual representations that are not too wild (in particular, those of the form $*+V$ for an actual representation $V$), this kills any potential contributions to degree $-1$.
\end{remark}

\section{Slices of THH}
\label{sec:slice}
In this section, we study the slice filtration on $\THH$ and prove the main theorems. In \sec\ref{sub:thh-slices}, we identify the slices and slice covers of $\THH$ in terms of $\RO(\T)$-graded suspensions (Theorem \ref{thm:slice-tower}). In \sec\ref{sub:filtration-tf}, we work out the filtration this gives on homotopy groups. We analyze the slice spectral sequence in \sec\ref{sub:rsss}.

\subsection{The slice tower}
\label{sub:thh-slices}
We begin by explaining the idea. Non-equivariantly, B\"okstedt periodicity implies that the Whitehead tower of $\THH$ is
\[\xymatrix{
  \gT & \Sigma^2\gT \ar[l]_-\sigma & \Sigma^4\gT \ar[l]_-\sigma & \ar[l] \cdots  
}\]
Since the slice filtration restricts to the Postnikov filtration on the underlying spectrum, it is reasonable to guess that equivariantly, the slice covers will be given by
\[ \regconn{2n}\gT = \Sigma^{V_n}\gT \]
for appropriate $\T$-representations $V_n$. Since
\[ (S^V \tnsr \gT)^{\Phi H} = S^{V^H} \tnsr \THH \]
by cyclotomicity, we are reduced to finding $\T$-representations $V_n$ with $d_0(V_n)=n$ and
\[ 2d_k(V_n) \ge \lc\frac{2n}{p^k}\rc \quad\text{for all }k\ge0. \]

Using the irreducible representations $\lambda_i$, these are uniquely determined $p$-locally. From the $p$-typical description it is not so easy to see a pattern, but contemplating the $q$-Legendre principle (Observation \ref{obs:q-lgndr}), we find that $V_n = [n]_\lambda$. For example, with $p=3$ we get
\begin{align*}
  2d_0(V_4) &= 8\\
  2d_1(V_4) &\ge 3\\
  2d_r(V_4) &\ge 1 \quad \forall r\ge2;
\end{align*}
these force $V_4 = 2\lambda_0 + \lambda_1 + \lambda_\infty$, which we recognize as $p$-locally equivalent to $\lambda^0 + \lambda^1 + \lambda^2 + \lambda^3 = [4]_\lambda$.

Educated guesses like this are actually a standard way to compute slices, thanks to the recognition principle \cite[Lemma 4.16]{HHR}. To verify it, we must produce maps $\Sigma^{[n+1]_\lambda}\gT \to \Sigma^{[n]_\lambda}\gT$ restricting to $\sigma$ on underlying spectra, and check that the cofibers are $n$-slices.

\begin{theorem}
\label{thm:slice-tower}
There is a canonical identification of exact triangles
\[\xymatrix{
  \regconn{2n+2}\gT \ar[r] \ar@{=}[d] & \regconn{2n}\gT \ar[r] \ar@{=}[d] & \regslice{2n}\gT \ar@{=}[d] \\
  \Sigma^{\drep{n+1}}\gT \ar[r] & \Sigma^{\drep n} \gT \ar[r] & \Sigma^{\{n\}_\lambda}\tr_{C_n}\m W
}\]
The bottom row is identified with
\[ \Sigma^{\{n\}_\lambda}\left(\Sigma^{\lambda_\infty} \gT \xra{\sigma u_{\lambda^n}^{-1}} \Sigma^{\lambda_\infty-\lambda^n} \gT \to \tr_{C_n}\m W \right) \]
or equivalently with
\[ \Sigma^{[n]_\lambda}\left(\Sigma^{\lambda^n} \gT \xra{\sigma u_{\lambda^n}^{-1}} \gT \to \m W/[pn]_A \right) \]
\end{theorem}
\begin{proof}
There are four steps.
\begin{enumerate}
\item Show that $\Sigma^{\drep n}\gT$ is slice $2n$-connective.

\item Produce the map $\Sigma^{\drep{n+1}}\gT \morph \Sigma^{\drep n} \gT$ with cofiber $\Sigma^{\{n\}_\lambda} \tr_{C_n} \m W$.

\item Check that $\Sigma^{\{n\}_\lambda}\tr_{C_n}\m W$ is $\ge 2n$.

\item Check that $\Sigma^{\{n\}_\lambda}\tr_{C_n}\m W$ is $\le 2n$.
\end{enumerate}
We have already observed (Corollary \ref{cor:irred-decomp}) that
$
  d_s([n]_\lambda) = \lc\frac n{p^s}\rc,
$ 
which verifies (1) by the preceding discussion together with the inequality $2\lc\frac n{p^s}\rc \ge \lc\frac{2n}{p^s}\rc$. The exact triangle
\[ \Sigma^{\lambda_\infty} \gT \xra{\sigma u_{\lambda^n}^{-1}} \Sigma^{\lambda_\infty-\lambda^n} \gT \to \tr_{C_n}\m W \]
follows from Lemma \ref{lem:rog-pos}, which takes care of (2).

The alternative expression for the cofiber sequence can be obtained either from Lemma \ref{lem:rog-neg}, or by using cell structures (Example \ref{ex:rog-m-neg}) to see that
\[
  \Sigma^{\lambda^n-\lambda_\infty} \tr_{C_n} \m W = \m W/[pn]_A.
\]

Now we check that $\Sigma^{\{n\}_\lambda}\tr_{C_n}\m W \ge 2n$. Let us write $\conn X$ for the connectivity of an ordinary spectrum $X$. We must show that
\[ \conn\Phi^{C_{p^k}}\left(\Sigma^{\{n\}_\lambda}\tr_{C_n}\m W\right) \ge \left\lceil\frac{2n}{p^k} \right\rceil. \]
We have
\begin{align*}
  \conn\Phi^{C_{p^k}} S^{\{n\}_\lambda} &= 2\lc\frac{n+1}{p^k}\rc - 2\\
  \conn\Phi^{C_{p^k}} \tr_{C_n}\m W &= \begin{cases}0&k\le v_p(n)\\\infty&k>v_p(n)\end{cases}
\end{align*}
which proves (3) since $k\le v_p(n)$ implies that $2\lc\frac{n+1}{p^k}\rc-2 = \frac{2n}{p^k}=\lc\frac{2n}{p^k}\rc$.

Finally, (4) follows from Lemma \ref{lem:mack-trunc} below.
\end{proof}

\begin{lemma}
\label{lem:mack-trunc}
If $\m M$ is any $\T$-Mackey functor, then $\Sigma^{\{n\}_\lambda}\m M \le 2n$.
\end{lemma}
\begin{proof}
This requires showing that
\[ [ S^{s\rho_{p^k}}, \downarrow^\T_{C_{p^k}} \Sigma^{\{n\}_\lambda} \m M] = 0\]
for all $sp^k > 2n$. We write $r\rho_{p^k} = \frac r2[p^k]_\lambda$, and note that
\begin{align*}
  \downarrow^\T_{C_{p^k}} S^{\{n\}_\lambda} \m M = \lf\frac n{p^k}\rf \lambda_\infty + \sum_{r=0}^{k-1}\left(\lf\frac n{p^r}\rf - \lf\frac n{p^{r+1}}\rf\right) \lambda_r
\end{align*}
so we need to compute $[S^V, \m M]$ where
\begin{align*}
  V &= \left(\frac s2 - \lf\frac n{p^k}\rf\right)\lambda_\infty + \frac s2 \alpha - \beta,\\[.5em]
  \alpha &= \sum_{r=0}^{k-1}(p^{k-r} - p^{k-(r+1)})\lambda_r,\\[.5em]
  \beta &= \sum_{r=0}^{k-1}\left(\lf\frac n{p^r}\rf - \lf\frac n{p^{r+1}}\rf\right)\lambda_r.
\end{align*}

Our assumption $sp^k > 2n$ gives $\ds\frac{sp^{k-r}}2 > \frac n{p^r}\ge\lf\frac n{p^r}\rf$, which implies that in the irreducible decomposition
\[ V = k_0 \lambda_0 + \dotsb + k_{k-1}\lambda_{k-1} + k_\infty \lambda_\infty, \]
we have $k_i\ge0$ for $0\le i<\infty$ and $k_\infty > 0$. The desired result then follows from Lemma \ref{lem:mack-bound}.
\end{proof}

\begin{example}
When $p=3$, the regular slice tower up to $\regconn{18}$ is
\[\xymatrix{
  \regconn{18}=\Sigma^{6\lambda_0+2\lambda_1+\lambda_\infty}\gT \ar[r] \ar[d] & \Sigma^{6\lambda_0+2\lambda_1+\lambda_\infty}\tr_{C_{p^2}}\m W = \regslice{18} & \regconn8=\Sigma^{2\lambda_0+\lambda_1+\lambda_\infty}\gT \ar[r] \ar[d] & \Sigma^{3\lambda_0+\lambda_1}\tr_e\m W =\regslice8\\
  \regconn{16}=\Sigma^{5\lambda_0+2\lambda_1+\lambda_\infty}\gT \ar[r] \ar[d] & \Sigma^{6\lambda_0+2\lambda_1}\tr_e\m W =\regslice{16} & \regconn6=\Sigma^{2\lambda_0+\lambda_\infty}\gT \ar[d] \ar[r] & \Sigma^{2\lambda_0+\lambda_1}\tr_{C_p}\m W=\regslice6\\
  \regconn{14}=\Sigma^{4\lambda_0+2\lambda_1+\lambda_\infty}\gT \ar[r] \ar[d] & \Sigma^{5\lambda_0+2\lambda_1}\tr_e\m W =\regslice{14} & \regconn4=\Sigma^{\lambda_0+\lambda_\infty}\gT \ar[d] \ar[r] & \Sigma^{2\lambda_0}\tr_e\m W=\regslice4\\
  \regconn{12}=\Sigma^{4\lambda_0+\lambda_1+\lambda_\infty}\gT \ar[r] \ar[d] & \Sigma^{4\lambda_0+2\lambda_1}\tr_{C_p}\m W =\regslice{12} & \regconn2=\Sigma^{\lambda_\infty}\gT \ar[d] \ar[r] & \Sigma^{\lambda_0}\tr_e\m W =\regslice2\\
  \regconn{10}=\Sigma^{3\lambda_0+\lambda_1+\lambda_\infty}\gT \ar[r] & \Sigma^{4\lambda_0+\lambda_1}\tr_e\m W =\regslice{10} & \regconn0=\gT \ar[r] & \m W = \regslice0
}\]
\end{example}

\begin{remark}
The slice filtration is not a filtration by cyclotomic spectra: instead, we have
\[ \Phi^{C_{p^k}} \regconn{2n} \gT = \regconn{2\lc n/p^k\rc} \gT. \]
\end{remark}

\subsection{The slice filtration}
\label{sub:filtration-tf}

In this section we work out the filtration induced on homotopy groups. We will treat both $\pi_{[i]_\lambda}\TF$ and $\pi_{2i}\TF$; the latter is what one would be probably interested in a priori, but the former appears to be more natural. The filtration is defined as
\begin{align*}
\Fil^{2j}_\rmS \pi_{[i]_\lambda} \TF &= \im(\pi_{[i]_\lambda}^\T \regconn{2(i+j)} \THH \to \pi_{[i]_\lambda} \TF)\\
\Fil^{2j}_\rmS \pi_{2i} \TF &= \im(\pi_{2i}^\T \regconn{2(i+j)} \THH \to \pi_{2i} \TF)
\end{align*}
Since $\TF$ is even in this range, our formulas will apply just as well to $\mpi_{[i]_\lambda}\gT$ and $\mpi_{2i}\gT$.

\begin{theorem}
\label{thm:slice-filtration}
The slice filtration takes the following form on homotopy. When $j\le0$ or $i=0$, $\Fil^{2j}_\rmS\pi_{[i]_\lambda}\TF$ is all of $\pi_{[i]_\lambda}\TF$. Otherwise, $\Fil^{2j}_\rmS\pi_{[i]_\lambda}\TF$ is generated by
\[ \frac{[p(i+j-1)]_A!}{[p(i-1)]_A!}. \]
When $j\le0$ or $i=0$, $\Fil^{2j}_\rmS\pi_{2i}\TF$ is all of $\pi_{2i}\TF$. Otherwise, $\Fil^{2j}_\rmS\pi_{2i}\TF$ is generated by
\[ \frac{[p(i+j-1)]_A!}{[p^r]_A^{i-1} \phi^r\left(\left[\lf\tfrac{i+j-1}{p^{r-1}}\rf\right]_A!\right)}, \]
where $r=\lc\log_p\left(\frac{i+j}i\right)\rc$. In particular, taking $i=1$ in either case gives
\[ \Fil^{2j}_\rmS\pi_2 \TF = [pj]_A!\pi_2\TF. \]
\end{theorem}

\begin{remark}
These formulas are easier to understand from our illustration of the $q$-Legendre formula (Figures \ref{fig:q-lgndr-2} and \ref{fig:q-lgndr-3}). The generator of $\Fil^{2j}_\rmS\pi_{[i]_\lambda}$ corresponds to taking the first $i+j-1$ columns, and then discarding the first $i-1$ columns. The generator of $\Fil^{2j}_\rmS\pi_{2i}$ corresponds to taking the first $i+j-1$ columns, and then discarding $i-1$ columns, \emph{tallest columns first}.
\end{remark}

\begin{proof}
The key identity, which follows from the $q$-gold relations and the $q$-Legendre principle, is
\[ \sigma^k a_{\{k\}_\lambda} u_{\{k\}_\lambda}^{-1} = [pk]_A!. \]

Starting with the case of $\pi_{[i]_\lambda}$, we want to identify the image of
\[
  \pi_{[i]_\lambda}^\T\regconn{2(i+j)}\gT
  = \pi_{[i]_\lambda-[i+j]_\lambda}\TF
  \xra{\sigma^{i+j} u_{[i+j]_\lambda}^{-1}}
  \pi_{[i]_\lambda}\TF.
\]
The terms are
\begin{align*}
  \pi_{[i]_\lambda}\TF &= A\<\sigma^i u_{[i]_\lambda}^{-1}\> \quad \text{by Tsalidis' theorem}\\
  \pi_{[i]_\lambda-[i+j]_\lambda}\TF &=
  \begin{cases}
    A\<\sigma^{-j} u_{[i]_\lambda}^{-1} u_{[i+j]_\lambda}\> & j\le 0 \quad \text{by Tsalidis' theorem}\\
    A\<a_{\{i+j-1\}_\lambda} a_{\{i-1\}_\lambda}^{-1}\> & j> 0 \quad \text{by Lemma \ref{lem:rog-neg}}
  \end{cases}
\end{align*}
Noting that $u_{[i]_\lambda}=u_{\{i-1\}_\lambda}$, this gives all the claims about $\pi_{[i]_\lambda}\TF$.

For the case of $\pi_{2i}$, we must identify the image of
\[
  \pi_{2i}^\T\regconn{2(i+j)}\gT
  = \pi_{2i-[i+j]_\lambda}\TF
  \xra{\sigma^{i+j} u_{[i+j]_\lambda}^{-1}}
  \pi_{2i}\TF = A\<\sigma^i\>
\]
Let $\alpha=\{i+j-1\}_\lambda$ with $p$-typical decomposition $\alpha=k_0\lambda_0 + \dotsb + k_{n-1}\lambda_{n-1}$. By Lemma \ref{lem:rog-neg},
\begin{align*}
  \pi_{2i-[i+j]_\lambda}\TF &= \pi_{2(i-1)-\alpha} \TF\\
  &=
  \begin{cases}
    A\<a_{\alpha[0,r-1)} a_{\lambda_{r-1}}^{d_{r-1}(\alpha)-i+1} u_{\lambda_{r-1}}^{i-1-d_r(\alpha)} u_{\alpha[r,n)}\> & d_r(\alpha) \le i-1 < d_{r-1}(\alpha)\\[.5em]
    A\<\sigma^{i-1-d_0(\alpha)} u_\alpha\> & d_0(\alpha) \le i-1
  \end{cases}
\end{align*}
We have $d_r(\alpha)=\lc\frac {i+j}{p^r}\rc-1$, so the dimension conditions become
\[
  (p^{r-1}-1)i < j \le (p^r-1)i \quad\text{or}\quad j\le0,
\]
equivalently $r=\lc\log_p\left(\frac{i+j}i\right)\rc$. This takes care of the case $j\le0$.

The case $j>0$ is more involved. We rewrite
\begin{align*}
    &\phantom{{}={}} \sigma^{i+j} u_\alpha^{-1} a_{\alpha[0,r-1)} a_{\lambda_{r-1}}^{d_{r-1}(\alpha)-i+1} u_{\lambda_{r-1}}^{i-1-d_r(\alpha)} u_{\alpha[r,n)}\\
  &= \frac{\sigma^{i+j-1} a_\alpha u_\alpha^{-1}}{\sigma^{i-1} a_{\lambda_{r-1}}^{i-1-d_r(\alpha)} u_{\lambda_{r-1}}^{-(i-1-d_r(\alpha))} a_{\alpha[r,n)}^{} u_{\alpha[r,n)}^{-1}} \sigma^i\\
  &= \frac{[p(i+j-1)]_A!}{[p^r]_A^{i-1} (a_{\lambda_{r-1}}^{-1} u_{\lambda_{r-1}})^{d_r(\alpha)} a_{\alpha[r,n)}^{} u_{\alpha[r,n)}^{-1}} \sigma^i
\end{align*}
To deal with the remaining term on the bottom, we write
\begin{align*}
  a_{\alpha[r,n)}^{} u_{\alpha[r,n)}^{-1} &= \prod_{s=r}^{n-1} (a_{\lambda_s}^{} u_{\lambda_s}^{-1})^{k_s}\\
  (a_{\lambda_{r-1}}^{-1} u_{\lambda_{r-1}}^{})^{d_r(\alpha)} a_{\alpha[r,n)}^{} u_{\alpha[r,n)}^{-1} &= \prod_{s=r}^{n-1} \phi^r([p^{s-r+1}]_A)^{k_s}\\
  &= \phi^r \prod_{\ell=1}^{n-r} [p^{\ell}]_A^{k_{\ell+r-1}}
\end{align*}
Since $k_s=\lf\frac{i+j-1}{p^s}\rf-\lf\frac{i+j-1}{p^{s+1}}\rf$, and $\lf\frac{\lf x/m\rf}{n}\rf=\lf\frac x{mn}\rf$, we can apply the $q$-Legendre formula (Lemma \ref{lem:q-lgdnr}) to conclude that this final expression is
\[ \phi^r\left(\left[\lf\tfrac{i+j-1}{p^{r-1}}\rf\right]_A!\right). \qedhere \]
\end{proof}

\subsection{The slice spectral sequence}
\label{sub:rsss}

Finally, we interpret the slice filtration using the regular slice spectral sequence (RSSS). The RSSS for a $G$-spectrum $X$ has signature
\[ E_2^{s,\alpha} = \mpi_{\alpha-s}\regslice{|\alpha|} X \conv \mpi_{\alpha-s} X \]
for $\alpha\in\RO(G)$. We draw this in the plane using Adams indexing, so $E^{s,\alpha+t}_2$ is placed at $(\alpha+t-s,s)$. The differentials go $d_r\colon E^{s,\alpha}_r \to E^{s+r,\alpha+(r-1)}_r$, or in terms of the plane display, translate by $(-1,r)$. We will again treat the two cases $\mpi_*\THH$ and $\mpi_{[*]_\lambda}\THH$. We suggest that the reader study the charts at the end before digesting the proofs in this section.


It will be useful to recall that the Mackey functors $\tr_{C_{p^r}}\m W$ and $\Phi^{C_{p^r}}\m W$ are given explicitly by
\[
  (\tr_{C_{p^r}} \m W)(\T/C_{p^k}) =
  \begin{cases}
    A/[p^{k+1}]_A & 0\le k\le r\\
    A/[p^{r+1}]_A & r<k
  \end{cases}
  \qquad
  (\Phi^{C_{p^r}} \m W)(\T/C_{p^k}) =
  \begin{cases}
    0 & 0\le k\le r\\
    A/\phi^{r+1}([p^{k-r}]_A) & r<k
  \end{cases}
\]

\begin{theorem}
\label{thm:htpy-slice}
The homotopy Mackey functors of the slices are given in even degrees by
\[
  \mpi_{2i}\regslice{2n} \gT =
  \begin{cases}
    \m W & 0=i=n\\
    \m R & 0<i=n\\
    \Phi^{C_{p^m}}\m W / [p^{h+1}]_A & 0< i < n
  \end{cases}
\]
where $\m R$ is the constant Mackey functor on $R$, and
\[
  m = \lc\log_p(n/i)\rc-1,
  \quad
  h=\begin{cases}
    \min\{v_p(n),\lf\log_p(n/i)\rf\} & n/i\text{ not a power of }p\\
    \lf\log_p(n/i)\rf & n/i\text{ a power of }p.
  \end{cases}
\]

If $R$ is $p$-torsionfree, then these are the only non-vanishing homotopy Mackey functors. If $R$ is a perfect $\F_p$-algebra, then
\[
  \mpi_{2i+1}\regslice{2n}\gT =
  \begin{cases}
    \tr_{C_{p^{m+h+1}}} \Phi^{C_{p^m}} \m W & n/i\text{ not a power of }p\\
    \tr_{C_{p^{m+h+1}}} \Phi^{C_{p^{m+1}}} \m W & n/i\text{ a power of }p.
  \end{cases}
\]
\end{theorem}
\begin{proof}
We use the exact sequence
\[ 0 \to \mpi_{2i+1}\regslice{2n}\gT \to \mpi_{2i} \regconn{2n+2} \gT \xra{\sigma u_{\lambda^n}^{-1}} \mpi_{2i} \regconn{2n} \gT \to \mpi_{2i}\regslice{2n}\gT \to 0. \]
By Lemma \ref{lem:rog-neg} and Corollary \ref{cor:irred-decomp}, we have $\mpi_{2i}\regconn{2n}\gT = \m W$ if $i=n$, and $\mpi_{2i}\regconn{2n}\gT = \Phi^{C_{p^m}}\m W$ if $0<i<n$. Let us write $g_1$ for the generator of $\mpi_{2i}\regconn{2n}\gT$ and $g_2$ for the generator of $\mpi_{2i}\regconn{2n+2}\gT$. The degrees are related by \[|g_2|=|g_1|-\lambda^n = |g_1|-\lambda_{v_p(n)}.\] Note that $m$ is the highest value of $j$ such that $a_{\lambda_j}$ appears in $g_1$.

If we write $e^k(j)$ for the exponent of $a_{\lambda_j}$ in $g_k$, then there is a unique $h$ such that $e^2(h)=e^1(h)+1$, $e^2(j)=e^1(j)$ for $j\ne h$. The map $\mpi_{2i}\regconn{2n+2} \gT \to \mpi_{2i} \regconn{2n} \gT$ will then hit $[p^{h+1}]_A$ times $g_1$. If $i=n$, then $h=0$ and we get $\m W/[p]_A=\m R$. Otherwise, there are three possibilities:
\begin{itemize}
\item If $v_p(n)\le m$, then $h=v_p(n)$.

\item If $v_p(n)> m$ and $n/i$ is not a power of $p$, then $h=m$.

\item If $v_p(n)>m$ and $n/i$ is a power of $p$, then $h=m+1$.
\end{itemize}
Some examples of the three cases with $p=3$ are
\[\xymatrix@R=1em{
  \Phi^{C_p}\m W\<a_{\lambda_0}^3 a_{\lambda_1}\> \ar[r]^-{[p]_A} & \Phi^{C_p}\m W\<a_{\lambda_0}^2 a_{\lambda_1}\> \ar[r] & \mpi_{2}\regslice{8}\gT\\
  \Phi^{e} \m W\<a_{\lambda_0}^2 u_{\lambda_1}\> \ar[r]^-{[p]_A} & \Phi^{e}\m W\<a_{\lambda_0} u_{\lambda_0}\> \ar[r] & \mpi_4\regslice6\gT\\
  \Phi^{C_p}\m W\<a_{\lambda_0}^4 a_{\lambda_1} u_{\lambda_1}\> \ar[r]^-{[p^2]_A} & \Phi^{e}\m W\<a_{\lambda_0}^4 u_{\lambda_1} \> \ar[r] & \mpi_4\regslice{12}\gT\\
}\]
We then combine these cases into the stated formula for $h$ by noting that $\lf\log_p(n/i)\rf=\lc\log_p(n/i)\rc$ if $n/i$ is a power of $p$, and $m=\lf\log_p(n/i)\rf$ otherwise.

When $R$ is $p$-torsionfree, the extended prism condition (Proposition \ref{prop:prism-extended}) shows that the above maps are injective, so the odd homotopy Mackey functors vanish. When $R$ is a perfect $\F_p$-algebra, there are two cases.


When $n/i$ is not a power of $p$, we get
\begin{align*}
  (\mpi_{2i+1}\regslice{2n}\THH)(\T/C_{p^k}) &=
  \begin{cases}
    0 & 0 \le k \le m\\
    A/p^{k-m} & 0 < k - m < h+1\\
    p^{k-m-(h+1)}A/p^{k-m}A & h + 1 \le k - m
  \end{cases}\\
  \intertext{and thus}
  \mpi_{2i+1}\regslice{2n}\THH &= \tr_{C_{p^{m+h+1}}} \Phi^{C_{p^m}} \m W.
\end{align*}
When $n/i$ is a power of $p$, we instead get
\begin{align*}
  (\mpi_{2i+1}\regslice{2n}\THH)(\T/C_{p^k}) &=
  \begin{cases}
    0 & 0 \le k \le m+1\\
    A/p^{k-(m+1)} & 0 < k - (m+1) < h\\
    p^{k-(m+1) - h}A/p^{k-(m+1)}A & h\le k - (m+1)
  \end{cases}\\
  \intertext{and thus}
  \mpi_{2i+1}\regslice{2n}\THH &= \tr_{C_{p^{m+h+1}}} \Phi^{C_{p^{m+1}}} \m W. \qedhere
\end{align*}
\end{proof}

The RSSS thus collapses at $E_2$ when $R$ is $p$-torsionfree. When $R$ is a perfect $\F_p$-algebra, the RSSS is very complicated; however, since we identified the entire slice tower (i.e.\ the maps between the slice covers, not just the slices), the $E_\infty$ page can be read off from Theorem \ref{thm:slice-filtration}.

\begin{corollary}
Let $R$ be a perfect $\F_p$-algebra, and define $h=h(n,i)$ as in Theorem \ref{thm:htpy-slice}. The entry on the $E_\infty$ page of the RSSS corresponding to $\mpi_{2i}\regslice{2n}\THH(R)$ is $\Phi^{C_{p^{f+1}}} \m W/p^{h(n,i)+1}$, where $f=\sum_{i\le m<n} (h(m,i)+1)$.
\end{corollary}

\begin{proposition}
The homotopy Mackey functors $\mpi_{[i]_\lambda}$ of the slices are
\[
  \mpi_{[i]_\lambda}\regslice{2n} \gT =
  \begin{cases}
    \m W & 0=i=n\\
    \m W/[pn]_A & 0<i=n\\
    \Phi^{C_{p^{\ell(i,n)}}}\m W / [pn]_A & 0 < i < n
  \end{cases}
\]
where $\ell(i,n)=\max\{v_p(i),\dotsc,v_p(n-1)\}=\min\{r \mid \lc n/p^r \rc = \lc i/p^r \rc\}$.
\end{proposition}
\begin{proof}
We use the exact sequence
\[ \mpi_{[i]_\lambda} \regconn{2n+2} \gT \xra{\sigma u_{\lambda^n}^{-1}} \mpi_{[i]_\lambda} \regconn{2n} \gT \to \mpi_{[i]_\lambda}\regslice{2n}\gT \to 0. \]
By Lemma \ref{lem:rog-neg}, the terms are
\begin{align*}
  \mpi_{[i]_\lambda-[n+1]_\lambda}\gT &=
  \begin{cases}
    0 & 0=i<n\\
    \Phi^{C_n}\m W\<a_{\lambda^n}\> & 0<i=n\\
    \Phi^{C_{p^{\ell(i,n+1)}}}\m W\<a_{\{n\}_\lambda} a_{\{i-1\}_\lambda}^{-1}\> & 0<i<n
  \end{cases}\\
  \mpi_{[i]_\lambda-[n]_\lambda}\gT &=
  \begin{cases}
    \m W & 0\le i=n\\
    \Phi^{C_{p^{\ell(i,n)}}}\m W\<a_{\{n-1\}_\lambda} a_{\{i-1\}_\lambda}^{-1}\> & 0<i<n
  \end{cases}
\end{align*}
so the result follows by the $q$-gold relation.
\end{proof}

We provide charts of the RSSS below. It is customary to use hieroglyphics to denote Mackey functors in spectral sequences. Our notations are listed in Figure \ref{fig:mackey-hiero}; Lewis diagrams for these can be found in Example \ref{ex:mackey}. Colors correspond to vanishing lines: for example, after restricting to $\{e\}$ we would only see the classes in red, after restricting to $C_p$ we would only see the red and orange classes, etc. In particular, we see that after restricting to any finite subgroup, the spectral sequence is bounded in each degree.

The $E_2$ page of the $\Z$-graded RSSS, in both the $p$-torsionfree and torsion cases, is depicted in Figures \ref{fig:slice-2-perfd}--\ref{fig:slice-3-tors}. We indicate the $E_\infty$ page of the $\Z$-graded RSSS for a perfect $\F_p$-algebra in Figure \ref{fig:slice-tors-oo}. Here, an entry $(f\quad h)$ means that the entry in plane coordinate $(2i,s)$ is $\Phi^{C_{p^{f+1}}} \m W/p^h$. Finally, the filtration on $\mpi_{[i]_\lambda}\THH$ (when $R$ is $p$-torsionfree) is depicted in Figures \ref{fig:slice-2-good} and \ref{fig:slice-3-good}; here the group $\mpi_{[i]_\lambda}$ appears in the $2i$ column, so this is not really a spectral sequence chart.

\begin{figure}
  \label{fig:mackey-hiero}
  \definecolor{orange}{HTML}{FF8000}
  \definecolor{yellow}{HTML}{DDCC00}
  \definecolor{green}{HTML}{009933}
  \definecolor{blue}{HTML}{0066FF}
  \definecolor{indigo}{HTML}{9933CC}
  \definecolor{violet}{HTML}{FD56BD}
  \begin{tabular}{|*9{>{$\vcenter\bgroup\hbox\bgroup$}c<{$\egroup\egroup$}|}}
    \hline
    & & /[p]_A & /[p^2]_A & /[p^3]_A & /[p^4]_A &/[p^5]_A & /[p^6]_A & /[p^7]_A\\\hline
    \m W & \color{red} \MWitt & \color{red} \Ma & \color{red} \Mb & \color{red} \Mc & \color{red} \Md & \color{red} \Me & \color{red} \Mf & \color{red} \Mg\\\hline
    \Phi^e\m W && \color{orange} \Ma & \color{orange} \Mb & \color{orange} \Mc & \color{orange} \Md & \color{orange} \Me & \color{orange} \Mf & \color{orange} \Mg\\\hline
    \Phi^{C_p}\m W && \color{yellow} \Ma & \color{yellow} \Mb & \color{yellow} \Mc & \color{yellow} \Md & \color{yellow} \Me & \color{yellow} \Mf & \color{yellow} \Mg\\\hline
    \Phi^{C_{p^2}}\m W && \color{green} \Ma & \color{green} \Mb & \color{green} \Mc & \color{green} \Md & \color{green} \Me & \color{green} \Mf & \color{green} \Mg\\\hline
    \Phi^{C_{p^3}}\m W && \color{blue} \Ma & \color{blue} \Mb & \color{blue} \Mc & \color{blue} \Md & \color{blue} \Me & \color{blue} \Mf & \color{blue} \Mg\\\hline
    \Phi^{C_{p^4}}\m W && \color{indigo} \Ma & \color{indigo} \Mb & \color{indigo} \Mc & \color{indigo} \Md & \color{indigo} \Me & \color{indigo} \Mf & \color{indigo} \Mg\\\hline
    \Phi^{C_{p^5}}\m W && \color{violet} \Ma & \color{violet} \Mb & \color{violet} \Mc & \color{violet} \Md & \color{violet} \Me & \color{violet} \Mf & \color{violet} \Mg\\\hline
  \end{tabular}
  \vspace{1em}

  \begin{tabular}{|*9{>{$\vcenter\bgroup\hbox\bgroup$}c<{$\egroup\egroup$}|}}
    \hline
    & \tr_{C_{p^{*+1}}} & \tr_{C_{p^{*+2}}} & \tr_{C_{p^{*+3}}} & \tr_{C_{p^{*+4}}} & \tr_{C_{p^{*+5}}}\\\hline
    \Phi^e\m W & \color{orange} \Mta & \color{orange} \Mtb & \color{orange} \Mtc & \color{orange} \Mtd & \color{orange} \Mte\\\hline
    \Phi^{C_p}\m W & \color{yellow} \Mta & \color{yellow} \Mtb & \color{yellow} \Mtc & \color{yellow} \Mtd & \color{yellow} \Mte\\\hline
    \Phi^{C_{p^2}}\m W & \color{green} \Mta & \color{green} \Mtb & \color{green} \Mtc & \color{green} \Mtd & \color{green} \Mte\\\hline
    \Phi^{C_{p^3}}\m W & \color{blue} \Mta & \color{blue} \Mtb & \color{blue} \Mtc & \color{blue} \Mtd & \color{blue} \Mte\\\hline
    \Phi^{C_{p^4}}\m W & \color{indigo} \Mta & \color{indigo} \Mtb & \color{indigo} \Mtc & \color{indigo} \Mtd & \color{indigo} \Mte\\\hline
    \Phi^{C_{p^5}}\m W & \color{violet} \Mta & \color{violet} \Mtb & \color{violet} \Mtc & \color{violet} \Mtd & \color{violet} \Mte\\\hline
  \end{tabular}

  \caption{Hieroglyphics for Mackey functors}
\end{figure}

\begin{figure}
  \begin{center}\includegraphics[width=0.6\textwidth]{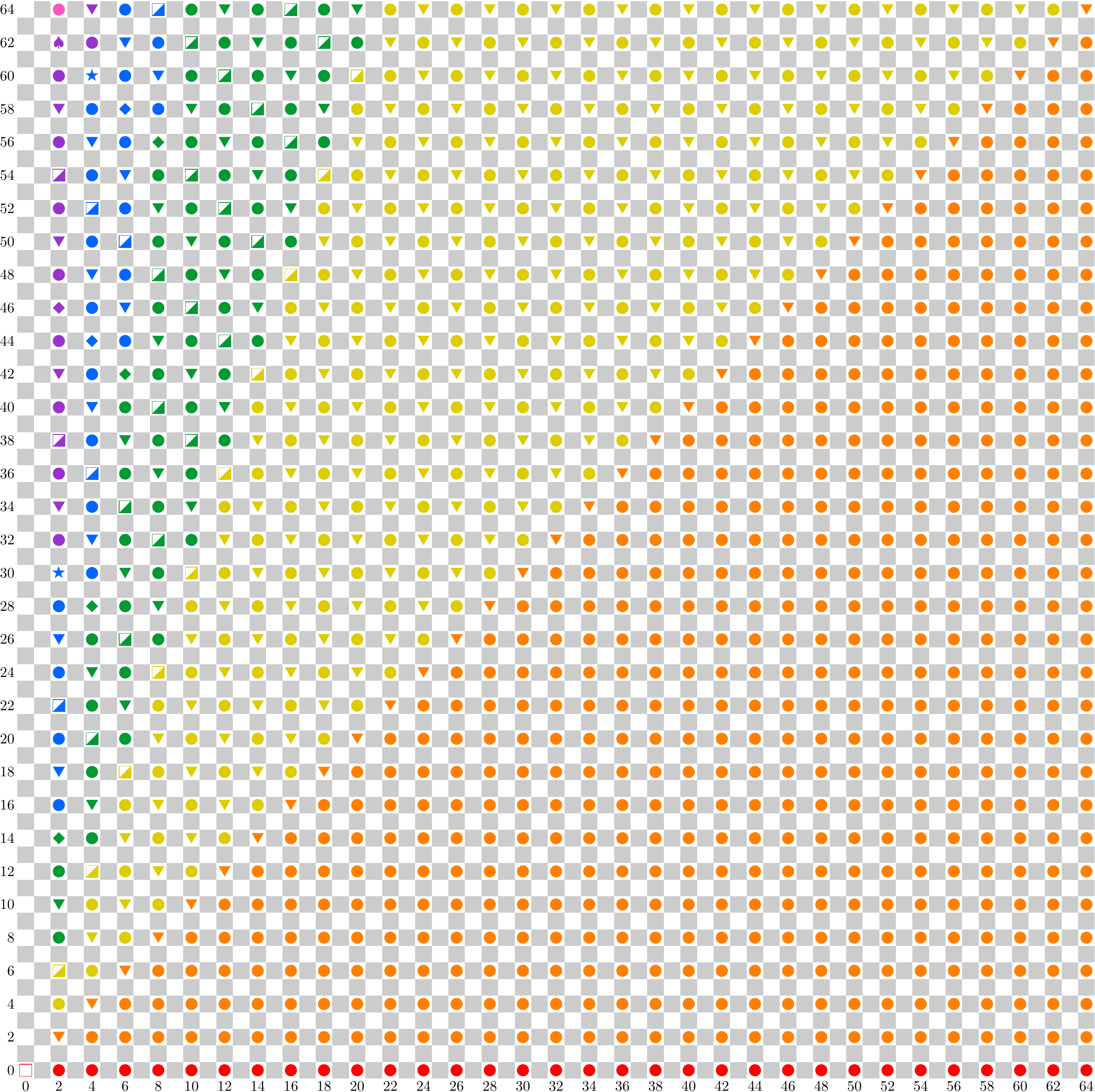}\end{center}
  \caption{$E_2$ page of the RSSS for $\THH(\Z_2^{\mathrm{cycl}}; \Z_2)$}
  \label{fig:slice-2-perfd}
\end{figure}

\begin{figure}
  \begin{center}\includegraphics[width=0.6\textwidth]{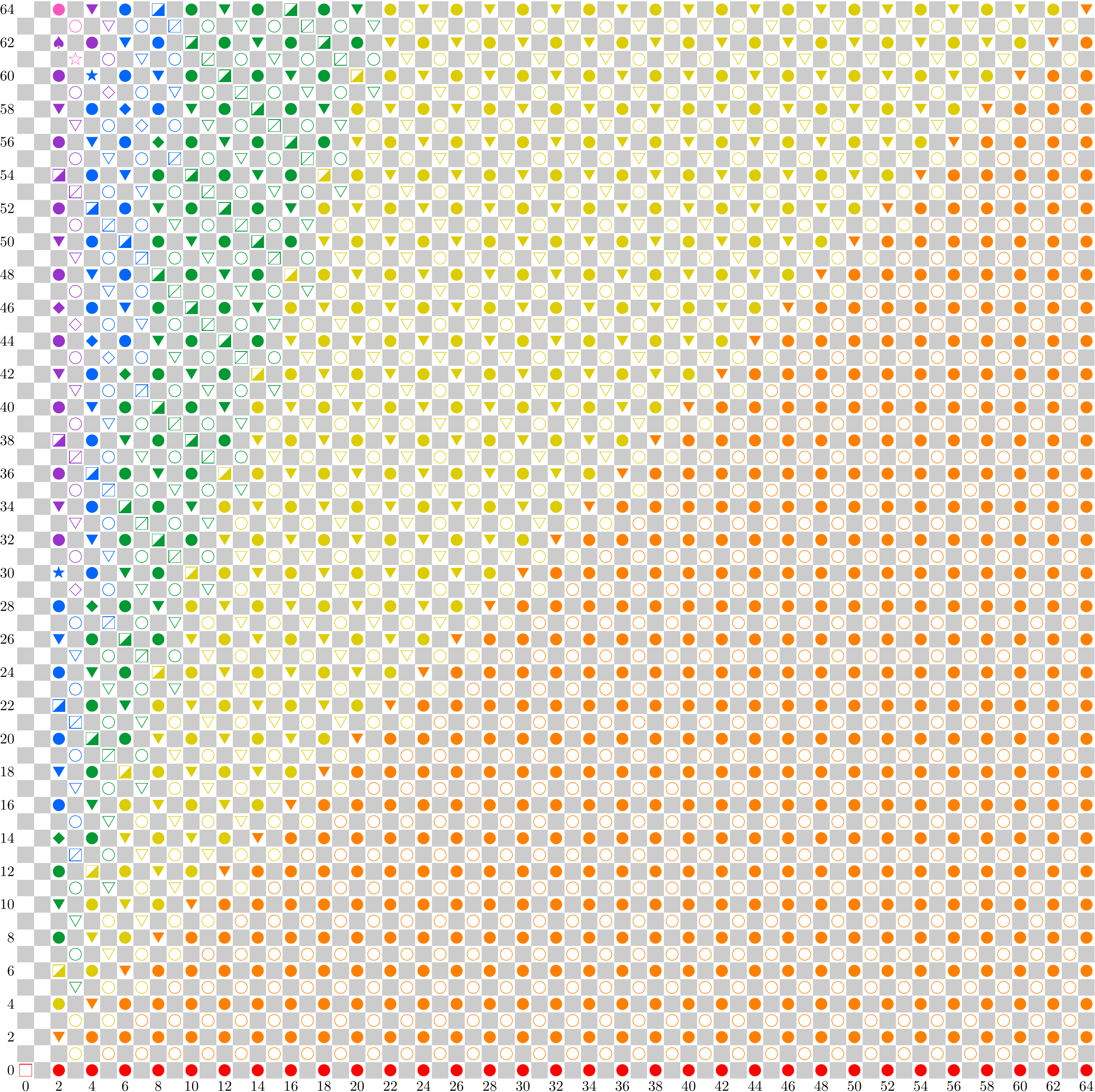}\end{center}
  \caption{$E_2$ page of the RSSS for $\THH(\F_2)$}
  \label{fig:slice-2-tors}
\end{figure}

\begin{figure}
  \begin{center}\includegraphics[width=0.6\textwidth]{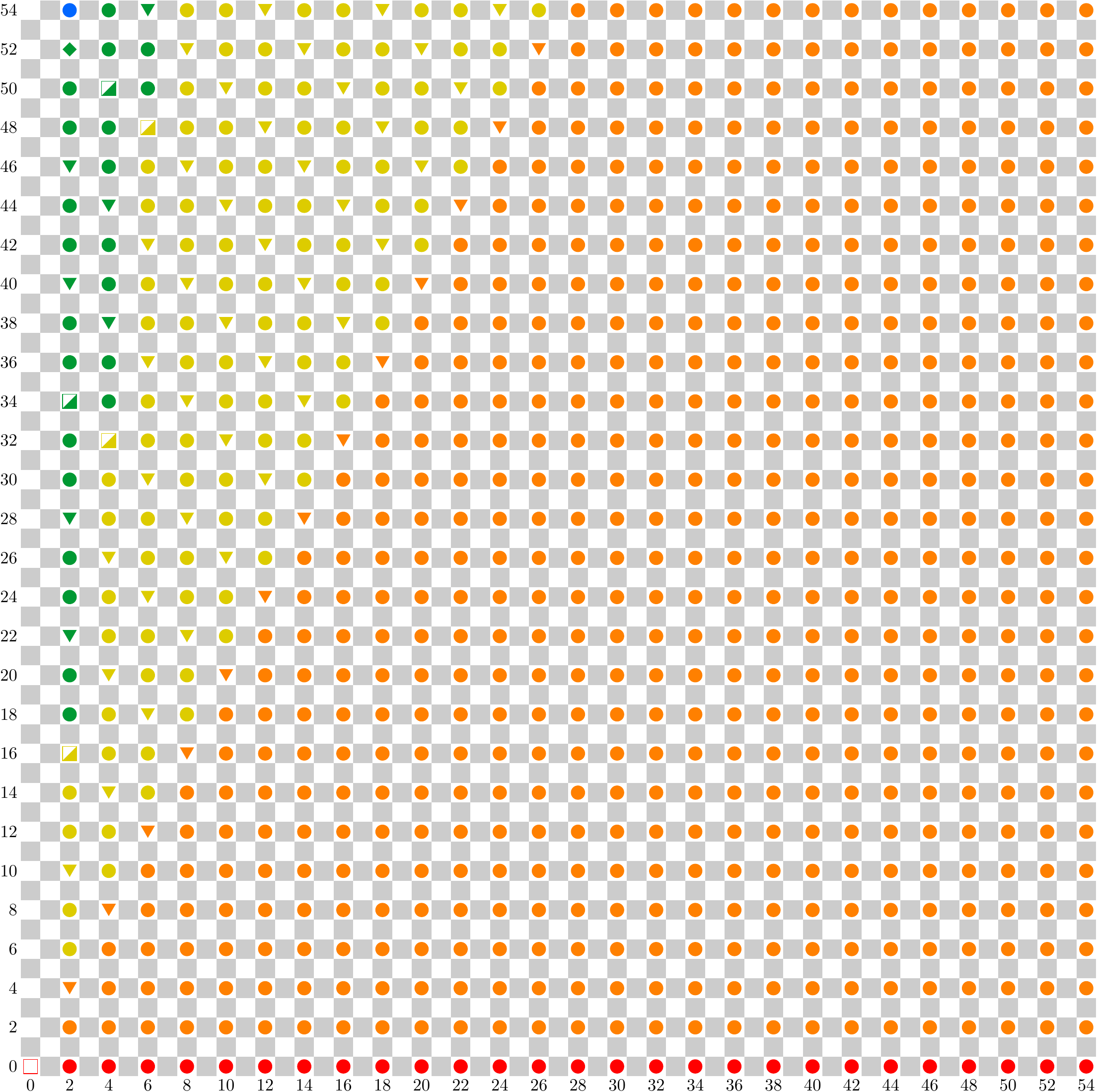}\end{center}
  \caption{$E_2$ page of the RSSS for $\THH(\Z_3^{\mathrm{cycl}}; \Z_3)$}
  \label{fig:slice-3-perfd}
\end{figure}

\begin{figure}
  \begin{center}\includegraphics[width=0.6\textwidth]{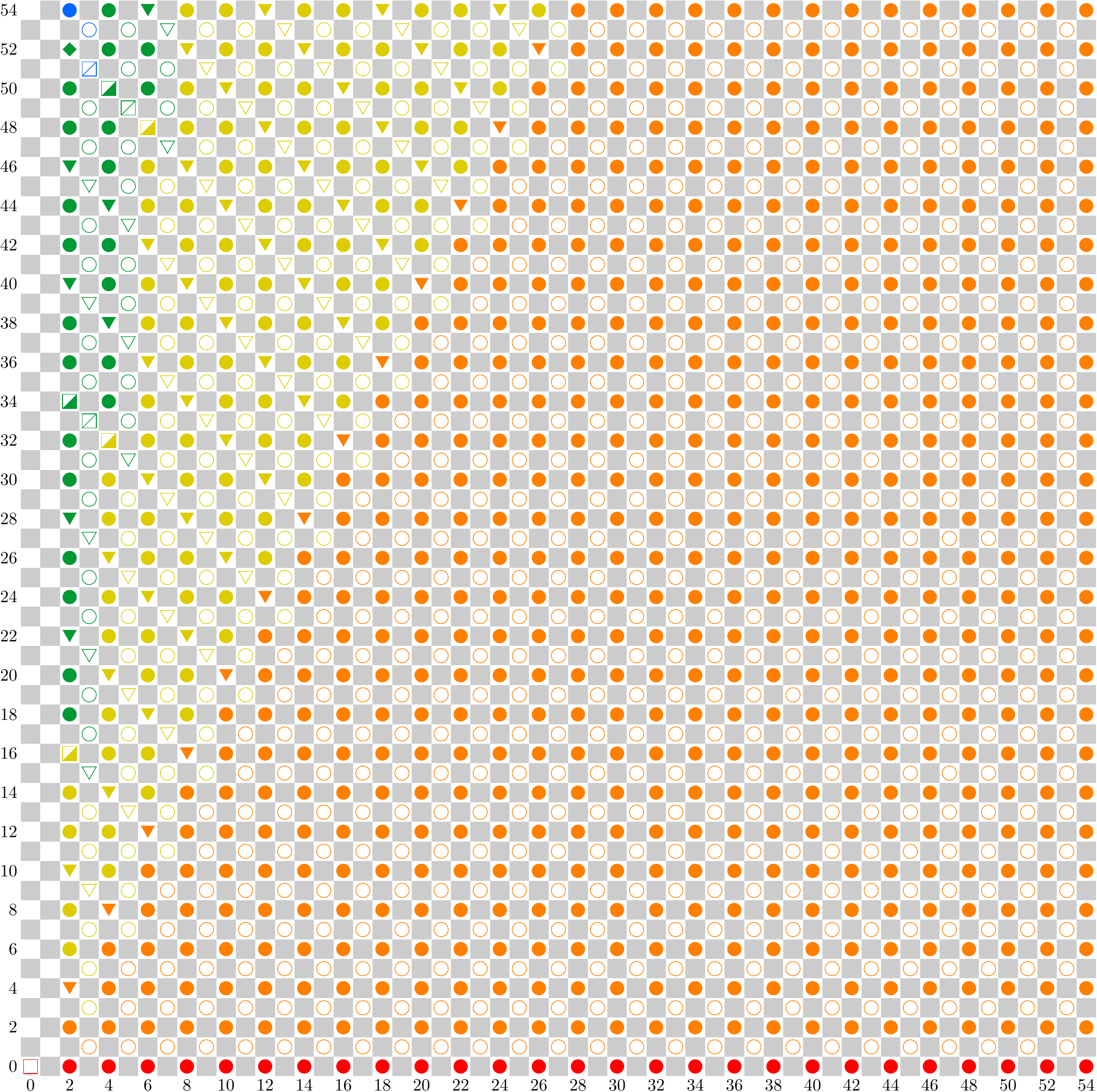}\end{center}
  \caption{$E_2$ page of the RSSS for $\THH(\F_3)$}
  \label{fig:slice-3-tors}
\end{figure}

\begin{figure}
  \begin{center}
    \includegraphics[width=0.4\textwidth]{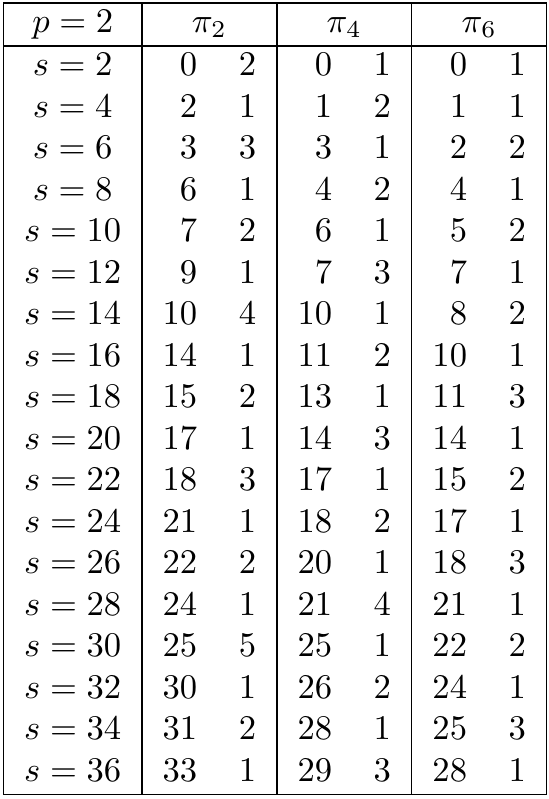}
    \includegraphics[width=0.4\textwidth]{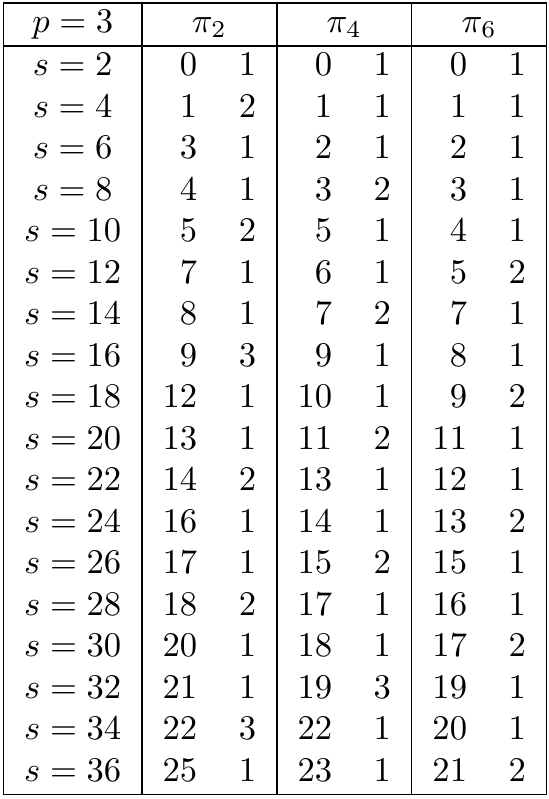}
  \end{center}
  \caption{$E_\infty$ page of the RSSS for $\THH(\F_2)$ and $\THH(\F_3)$.}
  \label{fig:slice-tors-oo}
\end{figure}

\begin{figure}
  \begin{center}\includegraphics[width=0.6\textwidth]{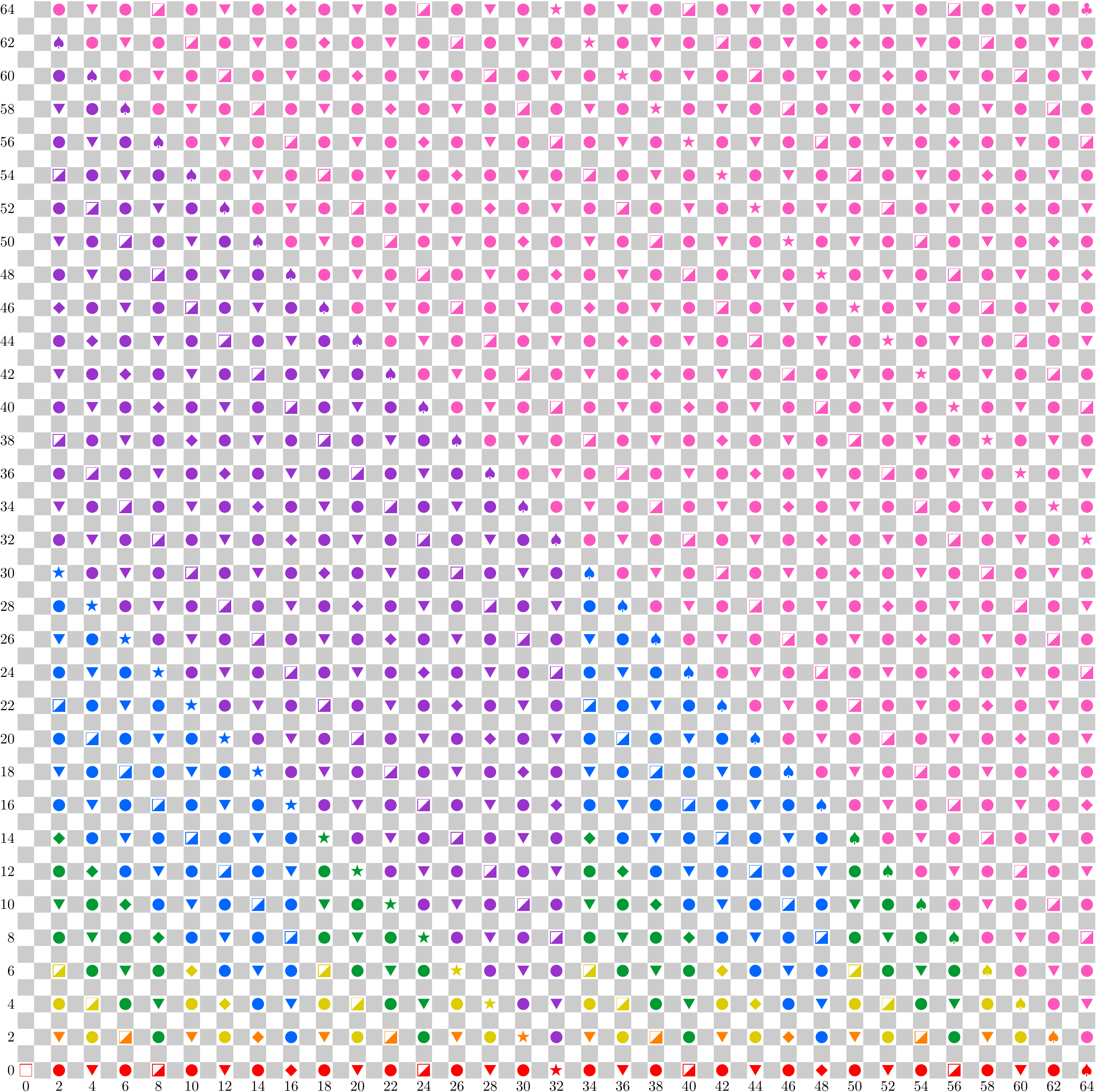}\end{center}
  \caption{Filtration on $\mpi_{[*]_\lambda}\THH(\Z_2^{\mathrm{cycl}}; \Z_2)$}
  \label{fig:slice-2-good}
\end{figure}

\begin{figure}
  \begin{center}\includegraphics[width=0.6\textwidth]{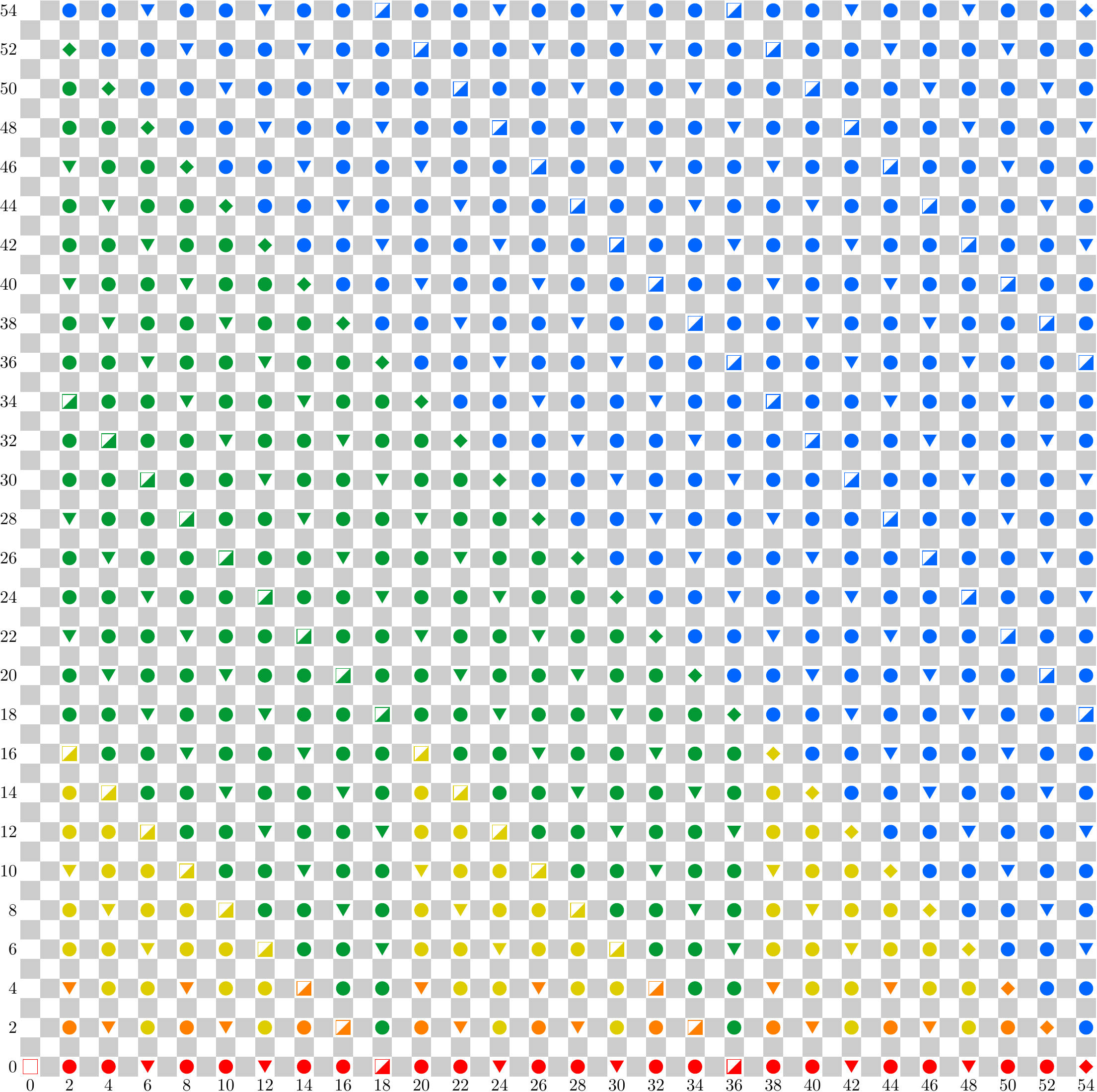}\end{center}
  \caption{Filtration on $\mpi_{[*]_\lambda}\THH(\Z_3^{\mathrm{cycl}}; \Z_3)$}
  \label{fig:slice-3-good}
\end{figure}

\section{Further questions}
\label{sec:future}

We close with some questions to stimulate ideas.

\begin{question}
Let $G$ be a finite cyclic group, and let $X$ be a $G$-spectrum whose underlying is 2-polynomial. By this we mean that $\pi_*^e(X) = \pi_0^e(X)[x]$ for a class $x\in\pi_2^e(X)$, so that $x^n$ exhibits $\Sigma^{2n}X$ as the $2n^\th$ Postnikov cover of $X^e$. Under what conditions is $X$ ``slice 2-polynomial'', in the sense that the odd slices vanish and
\begin{align*}
  \regconn{2n}X &= \Sigma^{[n]_\lambda}X?\\
  \intertext{For example, Ullman \cite[Theorem 5.1]{Ullman} shows that}
  \regconn{2n}\KU_G &= \Sigma^{[n]_\lambda}\ku_G,
\end{align*}
so $\ku_G$ is slice 2-polynomial (and $\KU_G$ is ``slice 2-periodic''). What should we expect for higher periodicities?
\end{question}

\begin{question}
If $C_n$ acts on $R$, then the fixed point Mackey functor $\m R$ is a $C_n$-Tambara functor. When $n$ is prime to $p$, Angeltveit-Blumberg-Gerhardt-Hill-Lawson-Mandell show that the ``relative $\THH$'' $N^\T_{C_n}\m R$ is a $p$-cyclotomic spectrum \cite[Theorem 8.6]{ABGHLM}. For example, $R$ could be a perfectoid $\Zpcycl$-algebra, and $C_n$ a finite quotient of $\Gal(\Qpcycl/\Q_p)=\Z_p^\times$. What does the slice filtration on $(N^\T_{C_n}\m R)^\wedge_p$ look like in this case?
\end{question}

\begin{question}
Is there a formal proof of Theorem \ref{thm:slice-tower}, relating the slice filtration of any cyclotomic spectrum to its Postnikov filtration?
\end{question}

\begin{question}
The $q$-gold relation provides a dictionary between representations and $q$-analogues; for example, $\{n\}_\lambda$ corresponds to $[pn]_q!$. (Bhargava's perspective \cite{Bhargava!} makes the additional factor of $p$ less distressing.) What $q$-analogues do other families of representations (symmetric powers, wedge powers, cannibalistic classes, \dots) correspond to? What representations correspond to $q$-multinomial coefficients, $q$-Catalan numbers, \dots? Are these useful for $\THH$?
\end{question}

\begin{question}
We have used the regular slice filtration due to its superior multiplicative properties, but there are other versions of the slice filtration. For example, a back-of-the-envelope calculation suggests that the classical slice filtration is given by
\[ \clsconn{2n}\gT = \Sigma^{\{2n\}_\lambda - \{n\}_\lambda}\gT=\Sigma^{[2n+1]_\lambda - [n+1]_\lambda}\gT. \]
Wilson describes a general framework for slice filtrations in \cite[\sec1.3]{Wilson}. Our proof of Theorem \ref{thm:slice-tower} was relatively formal; can it be generalized to identify the slice filtration on $\THH$ for an \emph{arbitrary} dimension function $\nu$? If so, what is the arithmetic interpretation of $\nu$?
\end{question}

\begin{question}
In the cyclotomic $t$-structure \cite{TCart}, the ``cyclotomic homotopy groups'' $\pi_i^{\mathrm{cyc}}$ of $\THH$ are essentially given by the ordinary homotopy groups $\pi_i$ of $\TR$. Slices are generalizations of homotopy groups, and $\TR$ is again a cyclotomic spectrum. Do the slices of $\TR$, or the $RO(\T)$-graded homotopy groups of $\TR$, correspond to something in the cyclotomic $t$-structure?
\end{question}

\begin{question}
Wilson \cite{Wilson} has given algebraic descriptions of the category of $n$-slices, as well as an algorithm for computing slices. What does this look like when $G=\T$, and how does it compare to our method?
\end{question}

\bibliographystyle{amsalpha}
\bibliography{./bibliography}

\end{document}